\documentclass[11pt,a4paper 
]{article}

\usepackage[all]{xy}
\usepackage{latexsym}
\usepackage[margin=3cm]{geometry}
\usepackage[dvipdfmx]{graphicx}
\usepackage{bm}

\usepackage{comment}

\usepackage{mymacros2}

\newcommand{\tm}{{\tilde{m}}}
\newcommand{\tp}{{\tilde{p}}}
\newcommand{\Ind}{{\operatorname{\bf{Ind}}}}

\DeclareMathOperator*{\indlim}{``\varinjlim"}
\newcommand{\tE}{{\tilde{E}}}

\newcommand{\musupp}{{\operatorname{SS}}}

\newcommand{\indotimes}{\mathbin{\overset{!}{\otimes}}}

\renewcommand{\mod}{\mathop{\mathrm{\bf{mod}}}}
\newcommand{\hSs}{{\hSigma(\hsigma)}}

\newcommand{\mulSh}{\operatorname{\bf{\mu{Sh}}}^\diamondsuit}
\DeclareMathOperator{\Indcoh}{{\operatorname{\Ind\Coh}}}
\DeclareMathOperator{\spanning}{{\mathrm{span}}}

\newcommand{\op}{\mathrm{op}}

\newcommand{\Loc}{\mathbf{\mathrm{Loc}}}

\newcommand{\Lsbd}[1]{{\Lambda_{\hSigma_{#1},\beta}}}
\newcommand{\Lsb}{{\Lambda_{\hSigma,\beta}}}
\newcommand{\cXsb}{{\cX_{\hSigma,\beta}}}
\newcommand{\cXssb}{{\cX_{\hSs,\beta}}}
\newcommand{\ksb}{{\kappa_{\hSigma,\beta}}}

\newcommand{\muhom}{{\mu{hom}}}
\newcommand{\T}{{\bm{\mathrm{T}}}}
\newcommand{\varhocolim}[1]{\mathop{\holim{\substack{\longrightarrow \\ {#1}}}}}
\newcommand{\varholim}[1]{\mathbin{\holim{\substack{\longleftarrow \\ {#1}}}}}

\DeclareMathOperator{\Int}{\mathrm{Int}}

\newcommand{\Ls}{{\Lambda_\Sigma}}

\newcommand{\tens}[2]{%
  \mathbin{\mathop{\otimes}\limits_{#1}^{#2}}%
}
\newcommand{\tim}[2]{%
  \mathbin{\mathop{\times}\limits_{#1}^{#2}}%
}
\renewcommand{\lim}[1]{%
  \mathbin{\mathop{\mathrm{lim}}\limits_{#1}}%
}
\renewcommand{\holim}[1]{%
  \mathbin{\mathop{\mathrm{holim}}\limits_{#1}}%
}
\newcommand{\ssigma}{{\bm{\sigma}}}

\renewcommand{\S}{{\Sigma}}

\newcommand{\hsigma}{{\hat{\sigma}}}
\newcommand{\hSigma}{{\hat{\Sigma}}}
\newcommand{\crho}{{\check{\rho}}}
\title{The nonequivariant coherent-constructible correspondence for toric stacks}
\author{Tatsuki Kuwagaki}
\date{}
\begin{document}

\maketitle

\begin{abstract}The nonequivariant coherent-costructible correspondence is a microlocal-geometric interpretation of homological mirror symmetry for toric varieties conjectured by Fang--Liu--Treumann--Zaslow.
We prove a generalization of this conjecture for a class of toric stacks which includes any toric varieties and toric orbifolds. Our proof is based on gluing descriptions of $\infty$-categories of both sides.
\end{abstract}

\section{Introduction}
Mirror symmetry is a mysterious relationship between complex and symplectic geometry which predicts various mathematical consequences. To give a unified understanding of mirror symmetry, Kontsevich proposed homological mirror symmetry (HMS) in 1994 \cite{KontsevichHMS}. HMS predicts an equivalence between two different kinds of categories: the derived category of coherent sheaves on a variety $X$ and a Fukaya-type category of a mirror of $X$. If $X$ is a toric Fano variety, a mirror of $X$ is given by a regular function $W$ over $(\bC^*)^n$ \cite{Givental,HV}. Then HMS predicts an equivalence
\begin{equation}\label{HMS}
\Coh X\simeq \Fuk (W)
\end{equation}
where $\Coh X$ is the derived $\infty$-category of coherent sheaves on $X$ and $\Fuk(W)$ is the derived Fukaya--Seidel $\infty$-category of $W$. The $\infty$-category $\Fuk(W)$ is defined by Lagrangian intersection Floer theory of Lefschetz thimbles of $W$ \cite{Seidelmutation}.

Recent developments of microlocal geometry suggest that such symplectic geometry can be captured by microlocal sheaf theory. The key notion in microlocal sheaf theory developed by Kashiwara--Schapira \cite{KS} is {\em microsupport}, which assigns to a sheaf a certain subset of the cotangent bundle of the manifold on which the sheaf lives. For a constructible sheaf, the microsupport becomes a Lagrangian subvariety of the cotangent bundle. This relation was further developed by Nadler--Zaslow \cite{NZ, Nad}: their theorem says that the derived category of constructible sheaves on a real analytic manifold $Z$ is equivalent to the derived infinitesimally wrapped Fukaya category of the cotangent bundle $T^*Z$. 

The Nadler--Zaslow equivalence mentioned above can be considered as a first example of {\em topological Fukaya categories}, i.e., topological description of Fukaya categories. Nadler--Zaslow's work is based on Fukaya--Oh's work \cite{FO} on Morse theory which is motivated by Witten's pioneering work \cite{WittenCS}. More generally, Kontsevich \cite{Kon2} suggested that Fukaya categories of Stein manifolds can be captured in terms of topological language. In the work of Dyckerhoff--Kapranov \cite{DK, Dyc} and Haiden--Katzarkov--Kontsevich \cite{HKK}, they constructed Fukaya-type categories for marked punctured Riemann surfaces without using Floer theory. Later, Nadler \cite{Nadwrapped} constructed the categories equivalent to the ones described in \cite{DK, HKK} via microlocal sheaf theory. In general, Kontsevich's conjecture are pursued by Nadler and Ganatra--Pardon--Shende \cite{NadcatMorse,GPS}. Along this line, Tamarkin and Tsygan are trying to construct microlocal categories expected to be equivalent to Fukaya categories of closed symplectic manifolds \cite{Tamcat, Tsygan}.

Topological description is also expected for Fukaya categories for Landau--Ginzburg models and wrapped Fukaya categories. Fang--Liu--Treumann--Zaslow \cite{FLTZ, FLTZ2} provided candidates of topological Fukaya categories for mirrors of toric varieties, which is the central topic in this paper. Sibilla--Treumann--Zaslow \cite{SiTZ} provided topological description for mirrors of chains of projective lines. Nadler \cite{NadW, NadW2} provided such examples for mirrors of pair of pants. Also, in the aforementioned work of \cite{DK, HKK}, some marked punctured Riemann surfaces can be considered as Landau--Ginzburg models. Topological nature of wrapped Fukaya categories of cotangent bundles have been described by Abbondandolo--Schwarz and Abouzaid \cite{AbbondandoloSchwarz,Abloop,Abcot}. Microlocal approach to wrapped Fukaya categories pursued by Nadler \cite{Nadwrapped} will be described later in this section.

Homological mirror symmetry by using those topological Fukaya categories has also been discussed. Fang--Liu--Treumann--Zaslow \cite{FLTZ, FLTZ3} proved torus-equivariant version of the coherent-constructible correspondence for smooth complete toric varieties and toric orbifolds. Here, the coherent-constructible correspondence is a version of homological mirror symmetry which is the central topic of this paper. Sibilla--Treumann--Zaslow \cite{SiTZ} proved the coherent-constructible correspondence for chains of projective lines. Nadler proved homological mirror symmetry of pair of pants as both A and B models \cite{NadW, NadW2, Nadwrapped}. Pascaleff--Sibilla \cite{PS} proved homological mirror symmetry for punctured Riemann surfaces by using Dyckerhoff--Kapranov's topological Fukaya category.

In this paper, we will discuss homological mirror symmetry for toric stacks by using topological Fukaya categories originally introduced by Fang--Liu--Treumann--Zaslow mentioned above. Namely, we replace the right hand side of (\ref{HMS}) by a category defined in terms of microlocal geometry. The following setting was first introduced by Fang--Liu--Treumann--Zaslow \cite{FLTZ, FLTZ2, FLTZ3} after the pioneering work of Bondal \cite{BO}.

In this paper, we always work over $\bC$.
Let $M$ be a free abelian group of rank $n$ and $N$ be its dual. Let further $\Sigma$ be a fan defined in $N_\bR$. We write $X_\Sigma$ for the toric variety associated with $\Sigma$. We set
\begin{equation}
\Lambda_\Sigma:=\bigcup_{\sigma\in \Sigma}p(\sigma^\perp)\times (-\sigma)
\end{equation}
where $p\colon M_\bR\rightarrow M_\bR/M=:T^n$ is the projection and $\sigma^\perp:= \lc m\in M_\bR\relmid m(\sigma)=\{0\} \rc$. We consider $\Lambda_\Sigma$ as the subset in $T^*T^n\cong M_\bR/M\times N_\bR$. Then $\Ls$ is a Lagrangian subvariety of $T^*T^n$ with respect to its standard symplectic structure.

One can generalize these setups to a class of toric stacks of Tyomkin and Geraschenko--Satoriano\cite{Tyomkin, GStoric}. Namely, let $\beta\colon L\rightarrow N$ be a homomorphism between finite rank free abelian groups with finite cokernel and $\hSigma$ and $\Sigma$ are fans defined in $L_\bR$ and $N_\bR$ respectively. In this paper, we assume the following condition unless otherwise stated.
\begin{condition}\label{condition}
The map $\beta_\bR$ induces a combinatorial equivalence between $\hSigma$ and $\Sigma$.
\end{condition}
In other words, we assume that the images of cones in $\Sigma$ by $\beta$ again form a fan and the induced morphism between fans is an isomorphism as a morphism of posets. We write $\cXsb$ for the associated toric stack. This class of toric stacks contains any toric varieties and any toric orbifolds in the sense of \cite{FLTZ3} which is defined as toric DM stacks without generic stabilizers in the sense of Borisov--Chen--Smith \cite{BCS, Iwanari, FMN}. We can generalize the construction of $\Ls$ to $(\hSigma,\beta)$ and write it $\Lambda_{\hSigma,\beta}$ (see Section \ref{cccformulation} below). 

On the coherent side, we will use the following four stable $\infty$-categories:
\begin{enumerate}
\item $\Indcoh \cXsb$: the $\infty$-category of ind-coherent sheaves on $\cXsb$,
\item  $\Qcoh \cX_{\hSigma,\beta}$: the derived $\infty$-category of quasi-coherent sheaves on $\cX_{\hSigma,\beta}$,
\item $\Coh \cXsb$: the bounded derived $\infty$-category of coherent sheaves on $\cXsb$,
\item $\perf \cX_{\hSigma,\beta}$: the $\infty$-category of perfect complexes on $\cX_{\hSigma,\beta}$.
\end{enumerate}
The category $\perf \cXsb$ (resp. $\Coh \cXsb$) is the full subcategory of $\Qcoh \cXsb$ (resp. $\Indcoh \cXsb$) spanned by compact objects.

Let $\cSh(T^n)$ (resp. $\lSh(T^n)$) be the derived $\infty$-category of (quasi-)constructible sheaves on $T^n$  . Then we will use the following three stable $\infty$-categories on the constructible side:
\begin{enumerate}
\item $\lSh_\Lsb(T^n)$: the full subcategory of $\lSh(T^n)$ spanned by objects whose microsupports are contained in $\Lsb$,
\item $\wSh_{\Lsb}(T^n)$: the full subcategory of $\lSh_\Lsb(T^n)$ spanned by compact objects,
\item $\cSh_\Lsb(T^n)$: the full subcategory of $\cSh(T^n)$ spanned by objects whose microsupports are contained in $\Lsb$.
\end{enumerate}
The category of wrapped constructible sheaves $\wSh_{\Lsb}(T^n)$ recently introduced by Nadler \cite{Nadwrapped} is a microlocal counterpart of (patially) wrapped Fukaya category~\cite{ASviterbo,AurouxHeegard,Sylvan}.

Our main theorem in this paper is the following.
\begin{theorem}[Theorem \ref{main}, Corollary \ref{complete}]\label{thm: ccc} There exists an equivalence of  $\infty$-categories
\begin{equation}
\Coh\cX_{\hSigma,\beta} \simeq \wSh_\Lsb(T^n), \label{equivalence1}
\end{equation}
If $\cXsb$ is complete, there exists an equivalence of $\infty$-categories
\begin{equation}
\perf\cXsb\simeq  \cSh_\Lsb(T^n). \label{equivalence2}
\end{equation}
If $\cXsb$ is smooth and complete, we moreover have an equivalence of $\infty$-categories 
\begin{equation}
\wSh_\Lsb(T^n)\simeq \cSh_{\Lsb}(T^n),
\end{equation}
in particular, 
\begin{equation}\label{eq: ccc}
\Coh\cXsb\simeq \cSh_\Lsb(T^n).
\end{equation}
\end{theorem}
Although this theorem holds for arbitrary toric varieties and toric orbofolds, smooth cases had been conjectured and discussed.
Originally, Bondal \cite{BO} sketched an equivalence between the derived categories for a class of toric varieties (which are in particular Fano) and the category of ``constructible sheaves with respect to certain decompositions of $T^n$". After Bondal's announcement, Fang--Liu--Treumann--Zaslow formulated and conjectured Theorem \ref{main} for smooth complete fans and called it {\em the coherent-constructible correspondence} \cite{FLTZ, FLTZ2, Tr}. Scherotzke--Sibilla~\cite{SS} generalized the conjecture for complete oribfolds, and Ike and the author \cite{IK} generalized further for smooth fans.

For complete fans, the conjecture was proved when $\Sigma$ is (i) zonotpal by Treumann \cite{Tr}, (ii) cragged by Scherotzke--Sibilla \cite{SS}, and (iii) 2-dimensional by the present author \cite{Kuw}. For noncomplete fans, Ike and the author \cite{IK} proved that noncomplete cases follow from complete cases. In their original paper \cite{FLTZ}, Fang--Liu--Treumann--Zaslow proved the equivariant version. 

For complete fans, we can see a relation between mirror symmetry and the coherent-constructible correspondence explicitly. By the Nadler--Zaslow equivalence, the right hand side of (\ref{eq: ccc}) can be viewed as the full subcategory of the Fukaya category of $T^*T^n$ consisting of Lagrangian submanifolds along $\Ls$. Then Fang--Liu--Treumann--Zaslow \cite{FLTZ2} showed that the coherent-constructible correspondence for line bundles exhibits T-duality \cite{SYZ}. However, relations to HMS for toric varieties formulated in terms of other types of Fukaya categories  \cite{Abouzaidtoric, AKO, AKO2, Ueda, UY} are still unclear at present time (some discussions can be found in \cite{FLTZ2}).

We obtain Theorem \ref{thm: ccc} as a corollary of the following.
\begin{theorem}[Theorem \ref{main}]\label{thm: qccc}
There exists an equivalence of $\infty$-categories
\begin{equation}
\Indcoh \cXsb \simeq \lSh_{\Lsb}(T^n).
\end{equation}
\end{theorem}Theorem \ref{thm: ccc} (\ref{equivalence1}) is obtained by taking compact objects on both sides of Theorem \ref{thm: qccc}. Then Theorem \ref{thm: ccc} (\ref{equivalence2}) is obtained as the dual of (\ref{equivalence1}) by taking the $\infty$-categories of exact functors to perfect $\bC$-modules.

Our proof of Theorem \ref{thm: qccc} is based on gluing argument, i.e., first we prove the affine cases (Proposition \ref{affine}), then we glue those equivalences and deduce Theorem \ref{thm: qccc}.

Roughly speaking, our proof of affine case is as follows. In smooth cases, the affine case is essentially known in the literature. In singular cases, we resolve singularities by toric resolutions and reduce to smooth cases.

Then we glue up the coheret-constructible correspondence for affines to that for $\cXsb$. On the coherent side, required gluing theorem is known as Zariski descent (Proposition \ref{gluingcoh}). On the constructible side, we give the following gluing theorem. We present $\Lsb$ as a union $\bigcup_{\sigma}\Lambda_\sigma$ of its closed subsets $\{\Lambda_\sigma\}_{\sigma\in \Sigma}$ which corresponds to the affine toric covering of $\cXsb$. Let $C(\Sigma)$ be the $\mathrm{\check{C}ech}$ poset of $\Sigma$ (see Sections \ref{section:gluecoh}).
\begin{theorem}[Theorem \ref{main2}]\label{thm: glue} There exists an equivalence of $\infty$-categories
\begin{equation}
\lSh_{\Lsb}(T^n)\simeq  \lim{\substack{\longleftarrow \\C(\Sigma)}}\lSh_{\Lambda_\bullet}(T^n).
\end{equation}
\end{theorem}
This theorem is proved by techniques of Tamarkin's projector \cite{Tam} sophisticated by Guillermou--Schapira \cite{GS}. Gluing theorem for topological Fukaya categories for closed Lagrangian subsets are also considered in the work of Pascaleff--Sibilla \cite{PS}. They treat only 1-dimensional Lagrangians (graphs), and Theorem \ref{thm: glue} can be considered as a first example of its generalization to higher-dimensional cases.

Although it is known that $\cSh_\Lsb(T^n)$ has the sheaf property called Kashiwara--Schapira stack \cite{KS, Guillermou, NadcatMorse} (see also Section 3.3), gluing proerty presented here is different from it.
Relations to the locality expected to Fukaya categories \cite{Kon2,NadcatMorse,DK, Dyc, GPS} and some proofs of homological mirror symmetry obtained by gluing \cite{SiTZ, PS, Nadwrapped} are future interests.

There may also be potential applications of Theorem \ref{thm: ccc} to problems in the derived category of coherent sheaves on toric varieties. In fact, Fang--Liu--Treumann--Zaslow applied their equivariant version to prove Kawamata's semi-orthogonal decomposition \cite{FLTZ4} and Scherotzke--Sibilla applied to construction of tilting complexes in the derived categories of coherent sheaves of cragged toric stacks \cite{SS}.

This paper is organized as follows. In Section 2, we recall some categorical generalities to fix our notations. In Section 3, we recall ind-coherent sheaves from \cite{GaitsgoryIndCoh}. Section 4, we recall microlocal sheaf theory of (wrapped) constructible sheaves and its relation to Fukaya category. We also recall and expand Tamarkin's techniques for convolution products. In Section 5, we define toric stacks used in this paper. In Section 6, we formulate the coherent-constructible correspondence for toric stacks. In Section 7,we prove the smooth affine case. In Section 8, we provide a proof of Zariski descent. In Section 9, we give a candidate of the functor which provides the coherent-constructible correspondence and see it is a generalization of the original formulation of \cite{FLTZ}. In Section 10, we describe the unit objects for convolution products, which is crucial for the proof of main theorems. In Section 11, we prove Theorem \ref{thm: ccc}, Theorem \ref{thm: qccc} and Theorem \ref{thm: glue} for the smooth case. In Section 12, we define the functor for singular case and prove that it is an equivalence by reducing the singular case to the smooth case.

\section*{Acknowledgments}
The author expresses his gratitude to Kazushi Ueda for many valuable comments and continuous encouragements. He also thanks Yuichi Ike. Some ideas of this work stems from the collaboration with him \cite{IK}.
While preparing this paper, the author learned that a parallel work was announced by Dmitry Vaintrob \cite{Vaintrob} from Eric Zaslow. Thereom \ref{thm: qccc} for smooth complete toric varieties was also obtained by him independently.  He thanks Eric Zaslow for the information and Dmitry Vaintrob for his comments. He also thanks Vivek Shende for making some remarks on the construction of Tamarkin's projector and for his lectures on microlocal geometry which affects the second draft of this paper, David Nadler for giving the author a proof of presentability of microlocal categories, and Alexei Bondal for giving the author a thorough explanation of his work \cite{BO}. He also thanks Harold Williams who pointed out some errors in the statement of early preprints. This work was supported by a Grant-in-Aid for JSPS Fellows 16J02358 and the Program for Leading Graduate Schools, MEXT, Japan.

\section{Preliminaries on categorical generalities}
For $\infty$-categories, we refer to Lurie's \cite{HTT, HA}. For the Morita model structure, we refer to \cite{Tabuada, Cohn}. We also refer to Section~1 of \cite{DK} as a useful summery.

In this paper, we always work over $\bC$, hence dg-category always means $\bC$-linear dg-category. All functors in this paper are appropriately derived.
First, we recall the quasi-equivalent model structure of the dg-category of dg-categories. We say that a dg-functor between two dg-categories is a quasi-equivalence if it induces an equivalence between the graded $\bC$-linear categories which are the homotopy categories of the dg-categories. 
\begin{definition}[The quasi-equivalent model structure \cite{Tabuada}]
The quasi-equivalent model structure on the dg-category of dg-categories is specified as follows: A dg-functor $f\colon \cC\rightarrow \cD$ is 
\begin{enumerate}
\item a weak equivalence if it is a quasi-equivalence,
\item a fibration if it is surjective on Hom-spaces and there exists a quasi-isomorphism $y\rightarrow x$ for any object $x\in \cD$ with $y=f(y')$ for some $y'\in \cC$,
\item a cofibration if it has the left lifting property with respect to trivial fibrations.
\end{enumerate}
\end{definition}

Let $\Mod(\cC)$ be the dg-category of dg-modules over a dg-category $\cC$.
A dg functor $f\colon \cC\rightarrow \cD$ induces a dg-functor $\Mod(\cD)\rightarrow \Mod(\cC)$. We say $f$ is a Morita equivalence if the induced functor is a quasi-equivalence.

\begin{definition}[The Morita model structure \cite{Tabuada}]
The Morita model structure on the dg category of dg-categories is specified by the following: a dg-functor $f\colon \cC\rightarrow \cD$ is 
\begin{enumerate}
\item a weak equivalence if it is a Morita equivalence,
\item a cofibration if it is a cofibration of the quasi-equivalent model structure,
\item a fibration if it has right lifting property with respect to trivial cofibrations.
\end{enumerate}
\end{definition}
We use the following properties of the Morita model structure.
\begin{proposition}[\cite{Tabuada}]\label{propertyMorita}
\begin{enumerate}
\item The Morita model structure is a left Bousfield localization of the quasi-equivalent model structure. 
\item The class of fibrant objects of the Morita model structure is the class of idempotent-complete pretriangulated dg-categories.
\end{enumerate}
\end{proposition}
\begin{proof}
(i) follows from the definition of left Bousfield localization. (ii) follows from the shape of generating cofibrations of the Morita model structure (See \cite[Remarque 5.3]{Tabuadaadditif}).
\end{proof}

In this paper, we always use $\bC$-linear stable $\infty$-categories for statements of theorems, but use $\bC$-linear dg-categories for proofs. This treatment is justified by the following theorem. Let $\cP r^L_{st,\omega}$ be the $\infty$-category of compactly generated stable $\infty$-categories whose 1-morphisms are functors preserving colimits and compact objects. Let $\Mod(H\bC)$ be the $\infty$-category of modules over the Eilenberg--Maclane spectrum  $H\bC$ of $\bC$, which defines a commutative algebra object in $\cP r^L_{st,\omega}$ (\cite{HA, Cohn}).
\begin{theorem}[\cite{Cohn}]\label{Cohn}
The following two $\infty$-categories are equivalent:
\begin{itemize}
\item $\mathrm{dgcat}^{\mathrm{Morita}}_\infty$ : the $\infty$-category obtained from the Morita model structure of the category of dg-categories. 
\item $\Mod_{\Mod(H\bC)}(\cP r^L_{st,\omega})$ : the $\infty$-category of modules over $\Mod(H\bC)$ in $\cP r^L_{st,\omega}$.
\end{itemize}
\end{theorem}

Moreover, by \cite[Proposition 1.3.4.23, 24]{HA}, (co)limits in $\mathrm{dgcat}^{\mathrm{Morita}}_\infty$ (equivalently, in $\Mod_{\Mod({H\bC})}(\cP r^L_{st,\omega})$) are calculated as homotopy (co)limits in the Morita model structure. Since the homotopy theory of ${\Mod}_{\Mod(H\bC)}(\cP r^L_{st,\omega})$ does not depend on any choice of model structures, Theorem \ref{Cohn} supports the canonicity of the Morita model structure.

\section{Preliminaries on Ind-coherent sheaves}\label{indcoherent}
For Ind-objects, we refer to Kashiwara-Schapira \cite{MR1827714} and Lurie \cite{HA}. For ind-coherent sheaves, we refer to Gaitsgory \cite{GaitsgoryIndCoh}.
\subsection{Ind-objects}
For an $\infty$-category $\cC$, the $\infty$-category of Ind-objects $\Ind(\cC)$ of $\cC$ is the filtered cocomplete closure of $\cC$ in the category of presheaves on $\cC$. We sometimes call this colimit {\em formal colimit} and write it $\indlim$. The original category is recovered by taking compact objects if the original category is idempotent-complete. Note that if $\cC$ is stable then $\Ind(\cC)$ is also stable.

Let $F\colon \cC\rightarrow \cD$ be a functor $\cC\rightarrow \cD$. Then there exists a natural functor $\Ind (F)\colon \Ind(\cC)\rightarrow \Ind(\cD)$ which extends $F$ and preserves formal colimits. We also use $F$ for $\Ind F$, if there are no confusions.
\begin{proposition}[{\cite[Proposition 1.1.3]{MR1827714}}]
If $F\colon \cC\rightarrow \cD$ is fully faithful, the fucntor $F\colon \Ind \cC\rightarrow \Ind \cD$ is also fully faithful.
\end{proposition}

\subsection{Ind-coherent sheaves}
Let $\cX$ be a Deligne--Mumford stack of finite type over $\bC$. Let $\Qcoh \cX$ be the derived $\infty$-category of quasi-coherent sheaves on $\cX$ and $\Coh X$ be the bounded derived $\infty$-category of coherent sheaves on $\cX$. Since $\Qcoh \cX$ is cocomplete, we have the functor $\Psi_\cX:=\varhocolim \colon \Indcoh\cX\rightarrow \Qcoh \cX$. Let $\Qcoh^b \cX$ be the bounded part of $\Qcoh \cX$.

\begin{proposition}[{\cite[Lemma 1.1.6, Corollary 3.3.6]{GaitsgoryIndCoh}}]\label{propertyofphi} The restriction of $\Psi_\cX$ to $\Qcoh^b \cX$ gives an equivalence onto $\Qcoh^b \cX$ in $\Qcoh \cX$.
If $\cX$ is smooth, the functor $\Psi_{\cX} \colon \Indcoh\cX\rightarrow \Qcoh \cX$ is an equivalence.
\end{proposition}

Let $f\colon \cX_1\rightarrow \cX_2$ be a morphism of Deligne--Mumford stacks.
\begin{proposition}[{\cite[Proposition 3.1.1, Lemma 3.5.8]{GaitsgoryIndCoh}}]\label{pushind}
There exists a unique cocontinuous functor
\begin{equation}
f^{\Ind}_*\colon \Indcoh\cX_1\rightarrow \Indcoh\cX_2
\end{equation}
which makes the diagram
\begin{equation}
\xymatrix{
\Indcoh \cX_1 \ar[d]^{\Psi_{\cX_1}}\ar[r]^-{f_*^{\Ind}}& \ar[d]_{\Psi_{\cX_2}}\Indcoh\cX_2 \\
\Qcoh\cX_1    \ar[r]^-{f_*} &\Qcoh\cX_2
}
\end{equation}
commute. If $f$ is proper, we have a natural isomorphsm $f_{*}^{\Ind}\simeq \Ind(f_*)$. If $f^*$ gives $\Coh \cX_2\rightarrow \Coh\cX_1$, the induced functor $f^*\colon \Indcoh \cX_2\rightarrow \Indcoh\cX_1$ is the left adjoint of $f^{\Ind}_*$.
\end{proposition}

\section{Preliminaries on microlocal sheaf theory}
\subsection{Microsupports}
Let $Z$ be a real analytic manifold. A sheaf $E$ on $Z$ is called constructible if there exists an analytic Whitney stratification of $Z$ such that $E$ is finite rank locally constant $\bC$-sheaf on each stratum \cite{KS}. If we allow $E$ has infinite rank on each stratum, we say $E$ is quasi-constructible \cite{FLTZ} (or weakly constructible in \cite{KS}, or large constructible in \cite{Nadwrapped}). We write $\cSh(Z)$ (resp. $\lSh(Z)$). for the stable $\infty$-category of complexes of constructible (resp. weakly constructible) sheaves localized at quasi-isomorphisms.

Next we define microsupports due to Kashwara--Schapira \cite{KS}. Let $\pi\colon T^*Z\rightarrow Z$ be the cotangent bundle of $Z$.
\begin{definition}\label{microsupport}
For $E\in \lSh(Z)$, its microsupport $\musupp(E)$ is the subset of $T^*Z$ which is characterized as follows: $(x;\xi)\in T^*Z$ is not an element of $\musupp(E)$ if there is an open neighborhood $U$ of $(x,\xi)$ and any smooth function $f$ with $\mathrm{Graph}(df)\subset U$, $\Gamma_{\lc z\in Z\relmid f(z)\geq f(x')\rc}(E)_{x'}\simeq 0$ holds for any $x'\in \pi(U)$.
\end{definition}
Although this definition can be applied to any sheaves on manifolds, we can know much more for quasi-constructible sheaves. Fix a conic Lagrangian $\Lambda\subset T^*M$. Let $\lSh_\Lambda(M)$ be the full subcategory of $\lSh(M)$ spanned by objects whose microsupports are contained in $\Lambda$.

Take $E\in\lSh_\Lambda(M)$ and a smooth point $(x,\xi)\in \Lambda$.
\begin{proposition}[{\cite[Propositioin 7.5.3]{KS}}]\label{microstalk}
Let $f\colon M\rightarrow \bR$ be a smooth function with $\mathrm{Graph}(df)$ intersecting $\Lambda$ transversely at $(x,\xi)$ and $f(x)=0$. Then up to shifts, the complex
\begin{equation}
\Gamma_{\lc y\relmid f(y)\geq 0\rc}(E)_x
\end{equation}
depends only on $E$ and $(x,\xi)$ and is independent of choice of $f$. If this complex is acyclic, then $(x,\xi)\not\in \musupp(E)$.
\end{proposition}
The complex $m_{(x,\xi),f}(E):=\Gamma_{\lc y\relmid f(y)\geq 0\rc}(E)_x$ is called the {\em microlocal stalk} of $E$ at $(x,\xi)$ with respect to $\Lambda$ and $f$. Proposition \ref{microstalk} says microlocal stalk completely captures microsupport.

There are standard estimates of microsupports which we will use later.
\begin{proposition}[\cite{KS}] \label{ssestimate}For $E,F\in \lSh(Z)$, the following estimates hold:
\begin{enumerate}
\item $\musupp(E\boxtimes F)\subset \musupp(E)\times \musupp(F)$,
\item $\musupp(\cHom_{Z\times Z}(p_1^{-1}E,p_2^{\bullet}F))\subset \musupp(E)^a\times \musupp(F)$ $(\bullet=-1, !)$, and
\item $\musupp(\Cone(E\rightarrow F))\subset \musupp(E)\cup \musupp(F)$.
\end{enumerate}
where $a$ is the fiberwise antipodal map and $p_i\colon Z\times Z\rightarrow Z$ is the projection to each factor.
\end{proposition}

If $E$ is constructible, $\musupp(E)$ is a Lagrangian subvariety with respect to the standard symplectic structure of $T^*Z$. This observation leads to recent applications of microlocal geometry to symplectic geometry, such as Tamarkin \cite{Tam} and Nadler--Zaslow \cite{NZ, Nad}. Here we review the latter briefly in the Section \ref{Fukayacatandconst} (we also recall some techniques of the former in Section \ref{technique}).

\subsection{Fukaya categories and constructible sheaves}\label{Fukayacatandconst}
We assume that $Z$ is compact.
We consider the derived infinitesimally wrapped Fukaya category $\Fuk(T^*Z)$ of $T^*Z$ whose objects are generated by exact Lagrangian submanifolds supposed to be tame at infinity and hom-spaces are calculated with inifinitesimal Hamiltonian perturbations which induce Reeb flows on the contact boundary when we canonically compactify $T^*Z$. Then the Nadler--Zaslow equivalence is the following.
\begin{theorem}[\cite{NZ, Nad}]
There is a quasi-equivalence of $\infty$-categories
\begin{equation}
\cSh(Z)\simeq \Fuk(T^*Z).
\end{equation}
\end{theorem}
This quasi-equivalence is roughly given as follows; the left hand side is generated by constant sheaves on small contractible closed sets. For such a sheaf, we assign an exact Lagrangian submanifold obtained by smoothing the microsupport of the sheaf. This gives an embedding of the left hand side to the right hand side. The quasi-inverse functor is given along the philosophy of family Floer homology; for an object $L$ of the right hand side, the stalk at $z\in Z$ of the corresponding object in the left hand side is given by the Floer homology between the fiber $T^*_zZ$ and $L$. By the tameness of objects, we can use an easy version of family Floer homology without technical difficulties described in \cite{Fukayafam, Abouzaidfam}.

In a rough classification, there are two types of Fukya categories defined for non-compact symplectic manifolds. One has finite-dimensional hom-spaces, the other has infinite dimensional ones. Infinitesimally wrapped Fukaya category and Fukaya--Seidel category are of first type. Microlocal counterparts of these are considered as constructible sheaves.

Fukaya categories of second type are known as partially wrapped Fukaya categories \cite{ASviterbo, AurouxHeegard, Sylvan}. A microlocal counterpart of this notion is recently proposed by Nadler \cite{Nadwrapped}. Let $\Lambda$ be a conic Lagrangian subvariety of $T^*Z$. We write $\lSh_\Lambda(Z)$ for the full subcategory of $\lSh(Z)$ spanned by objects whose microsupports are contained in $\Lambda$.
\begin{definition}[{\cite[Definition 3.12]{Nadwrapped}}]
The stable $\infty$-category $\wSh_\Lambda(Z)$ of wrapped constructible sheaves along $\Lambda$ is the full subcategory of compact objects of $\lSh_\Lambda(Z)$.
\end{definition}
The stability follows from Proposition \ref{ssestimate} (iii).
There is another (more geometric) definition of this category also due to Nadler \cite{Nadwrapped}. We will use the notation of the definition of microlocal stalk presented in Section \ref{Fukayacatandconst}. Take a smooth point $(x,\xi)$ of $\Lambda$ and a smooth function $f\colon M\rightarrow \bR$ as in Proposition \ref{microstalk}. By microlocal Bertini--Sard theorem \cite[Proposition 8.12]{KS} and the non-characteristic deformation lemma \cite[Proposition 2.7.2]{KS}, we have a neighborhood $B$ of $x$ such that
\begin{equation}
m_{(x,\xi),f}(E):=\Gamma_{\lc z\in Z\relmid f(z)\geq f(x)\rc}(E)_x\simeq \Gamma_{\lc z\in Z\relmid f(z)\geq f(x)\rc}(B, E).
\end{equation}
Note that $B$ depends on $\Lambda$, but does not depend on $E$. Since $\lSh_\Lambda(M)$ is presentable (cf. Proposition \ref{presentable} and Proposition \ref{Neemanrepresentability} below) and $\Gamma_{\lc z\in Z\relmid f(z)\geq f(x)\rc}(B, -)$ preserves limits and colimits, $m_{(x,\xi),f}\colon \lSh_\Lambda(M)\rightarrow \Mod(\bC)$ is representable by a compact object $F_{(x,\xi),f}\in \wSh_\Lambda(M)$. This $F_{(x,\xi),f}$ is called {\em microlocal skyscraper sheaf} at $(x,\xi)$ with respect to $\Lambda$ and $f$. By Proposition \ref{microstalk}, $F_{(x,\xi),f}$ does not depend on $f$ up to shifts.

\begin{lemma}[{\cite[Lemma 3.15]{Nadwrapped}}]\label{3.15}
The $\infty$-category $\wSh_\Lambda(Z)$ is split-generated by the microlocal skyscrapers $F_{(x,\xi),f}$ for $(x.\xi)$ smooth points of $\Lambda$.
\end{lemma}

The category of wrapped constructible sheaves have some properties expected for partially wrapped Fukaya categories, for example, mirror symmetry for pair of pants \cite{Nadwrapped} and categorical localization \cite{IK}. Like partially wrapped Fukaya categories sometimes have only finite dimensional hom-spaces, all wrapped costructible sheaves are sometimes constructible sheaves, for example, $\Lambda=\Lambda_{\hat{\Sigma},\beta}$ for complete $\Sigma$ (see Corollary~\ref{complete} below, see also appendix of \cite{IK}).

Finally we recall the duality between constructible and wrapped constructible sheaves from \cite{Nadwrapped}. Let $\mod(\bC)$ be the derived $\infty$-category of bounded complexes of finite-dimensional $\bC$-vector spaces. Let further $\Fun{}^{ex}(\cC,\cD)$ be the $\infty$-category of exact functors for stable $\infty$-categories $\cC$ and $\cD$.
\begin{theorem}[{\cite[Theorem 3.2.1]{Nadwrapped}}]\label{dualityconst}
There exists an equivalence of $\infty$-categories
\begin{equation}
\cSh_\Lambda(Z)\simeq \Fun{}^{\text{ex}}((\wSh_\Lambda(Z))^\op, \mod(\bC)) 
\end{equation}
given by $E\mapsto \hom(-,E)$.
\end{theorem}

\subsection{Kashiwara--Schapira stack}\label{ksstack}
The notion of Kashiwara--Schapira stack is essentially established in the theory of microlocal categories of Kashiwara--Schapira \cite{KS} and is clarified in recent various literatures \cite{NadcatMorse, Nadarb, Nadnonchar, SiTZ, STW, Guillermou}.

Firstly, we introduce microlocal categories. Let $Z$ be a real analytic manifold, $\Lambda$ be a conic Lagrangian subset of $T^*Z$ and $U$ be an open subset of ${T}^*Z$. Then we set
\begin{equation}
{\mulSh_\Lambda}^{pre}(U):=\lSh_{\Lambda}(Z)/\lSh_{(\Lambda\bs (U\cap \Lambda))}(Z).
\end{equation}
Then the assignment ${\mulSh_\Lambda}^{pre} \colon U\mapsto \mulSh_{\Lambda}(U)$ form a presheaf on ${T}^*Z$. We write $\mulSh_\Lambda$ for the sheafification of ${\mulSh_\Lambda}^{pre}$. Moreover, the support of this sheaf is $\Lambda$.

\begin{definition}
The sheaf $\mulSh_\Lambda$ on $\Lambda$ is called the Kashiwara--Schapira stack (or the stack of microlocal sheaves) along $\Lambda$.
\end{definition}

The global section of this sheaf is known as follows.
\begin{proposition}[{\cite[Propsitioin 3.5]{STW}}]
The global section of $\mulSh_\Lambda(Z)$ is equivalent to $\lSh_{T^*_ZZ\cup\Lambda}(Z)/\lSh_{T^*_ZZ}(Z)$.
\end{proposition}

The author learned the following proposition from David Nadler.
\begin{proposition}\label{presentable} Assume $Z$ is compact. The global section of $\mulSh_\Lambda$ is a compactly generated $\infty$-category.
\end{proposition}
\begin{proof} 
By the arborealization \cite{Nadarb, Nadnonchar}, $\mulSh_\Lambda(U)$ is compactly generated for any sufficiently small open set $U$ of $\Lambda$. Since the $\infty$-category of compactly generated $\infty$-category is closed under finite limits, the global section of $\mulSh_\Lambda$ is also compactly generated.
\end{proof}
The presentability allows us to use Neeman's Brown representability theorem.
\begin{proposition}[{\cite{Neeman}}]\label{Neemanrepresentability}
Let $\cC$ be a compactly generated stable $\infty$-category and $f\colon \cC\rightarrow \cD$ is an exact $\infty$-functor from $\cC$ to a stable $\infty$-category $\cD$. If $f$ is product-preserving, $f$ has a left adjoint. If $f$ is coproduct-preserving, $f$ has a right adjoint.
\end{proposition}

\subsection{Techniques of Tamarkin (after Guillermou--Schapira)}\label{technique}
In this section, we recall and slightly generalize some techniques of Tamarkin \cite{Tam}. We follow the survey by Guillermou--Schapira \cite{GS}. Although almost all statements before Lemma \ref{vanishing} can be obtained by the argument in \cite{GS}, we present details for reader's convenience.

Let $M$ be a rank $n$ free abelian group, $N$ be its dual, and let $M_\bR:=M\otimes_\bZ\bR, N_\bR:=N\otimes_\bZ\bR$. Then the cotangent bundle of the real torus $T^n:=M_\bR/M$ has a canonical trivialization $T^*T^n \simeq T^n\times N_\bR$. We set the quotient map $p\colon M_\bR\rightarrow T^n$.

We write $m\colon T^n\times T^n\rightarrow T^n$ for the multiplication of the torus. Note that $m$ is proper. Then $\lSh(T^n)$ have the convolution product $\star$ defined by 
\begin{equation}
E\star F:=m_!(E\boxtimes F) =m_!(p_1^{-1}E\otimes p_2^{-1}F)=m_*(p_1^{-1}E\otimes p_2^{-1}F)
\end{equation}
for $E,F\in \lSh(T^n)$ where $p_i$ is the projection from $T^n\times T^n$ to each factor. Since $p_i$ are proper, the push-forward $p_*$ is equal to the proper push-forward $p_!$.

The convolution has the right adjoint. We define
\begin{equation}
\cHom^\star(E,F):=p_{1*}\cHom(p_2^{-1}E, m^!F)
\end{equation}
for $E, F\in \lSh(T^n)$.

\begin{lemma}\label{staradjunction} For $E, F, G\in \lSh(T^n)$, we have
\begin{align}
\hom(E\star F, G)\simeq \hom(E, \cHom^\star(F, G)), \\
\cHom^\star(E, F)\simeq m_*\cHom(p_2^{-1}(-1)^{-1}E, p_1^!F),\label{explicit2}
\end{align}
where $-1\colon T^n\rightarrow T^n$ takes an element to its inverse.
\end{lemma}
\begin{proof}
The first one is clear from the definitions. The second one follows from the argument of \cite[Lemma 4.10]{GS}, namely, set $f:=(m, -p_2)\colon T^n\times T^n\rightarrow T^n\times T^n$. We can see $f\circ f=\id$. Since $f$ is an isomorphism, $f^{-1}=f^!$ and $f^{-1}=f_*$. Hence 
\begin{equation}
\begin{split}
\cHom^\star(E, F)&\simeq p_{1*}\cHom(p_2^{-1}E, m^!F)\\
&\simeq p_{1*}\cHom(f^{-1}p_2^{-1}(-1)^{-1}E, f^{!}p_1^{!}F)\\
&\simeq p_{1*}f^{-1}\cHom(p_2^{-1}(-1)^{-1}E, p_1^{!}F)\\
&\simeq (p_{1}\circ f)_*\cHom(p_2^{-1}(-1)^{-1}E, p_1^!F)\\
&\simeq m_*\cHom(p_2^{-1}(-1)^{-1}E, p_1^!F).
\end{split}
\end{equation}
This completes the proof.
\end{proof}

\begin{lemma}
For $E\in \lSh(T^n\times T^n)$, we have
\begin{align}
\musupp(m_!E)&\subset m_\#(\musupp(E)).\\
\end{align}
where $m_\#(\musupp(E))$ is defined by 
\begin{equation}
m_\#(\musupp(E)):=\lc(x, \xi)\relmid \exists (x_1,x_2)\in T^n\times T^n \text{ such that } x_1+x_2=x \text { and } (x_1,\xi, x_2,\xi)\in \musupp(E)\rc.
\end{equation}
\end{lemma}
\begin{proof}
This is a version of \cite[Theorem 2.16]{GS}. Although we can prove in a similar manner to \cite{GS}, we will give a different proof here.

Take a point $(x,\xi)\in T^*T^n\bs m_\#(\musupp(E))$. We trivialize the fibration $m$ as $T^n\times m^{-1}(x)\rightarrow T^n$. Since $(x,\xi)\not\in m_\#(\musupp(E))$, the microsupport $\musupp(E)$ does not have $\xi$ on $m^{-1}(x)$. Hence for any point $y\in m^{-1}(x)$, the microlocal stalk along $\xi$ vanishes. Since $m^{-1}(x)$ is compact, the radii of open disks centered on elements of $m^{-1}(x)$ on which one evaluates microlocal stalk (we wrote as $B$ in Section 3.1) have a lower bound. Hence we can take a neighbourhood $U=m^{-1}(V)$ of $m^{-1}(x)$ for some open set $V$ and a function $f$ on $V$ such that $f(x)=0$, $df(x)=\xi$, and $\mathrm{Graph}(df)\pitchfork m_\#(\musupp(E))$ with
\begin{equation}
\Gamma_{\{m\circ f\geq 0\}}(U, E)\simeq 0.
\end{equation}
This implies
\begin{equation}
m_{x,\xi}(E)\simeq \Gamma_{f\geq 0}(V,m_*E)\simeq \Gamma_{\{m\circ f\geq 0\}}(U, E)\simeq 0.
\end{equation}
Hence $(x,\xi)\not \in \musupp(m_*E)= \musupp(m_!E)$. This completes the proof.
\end{proof}

\begin{proposition}\label{mestimate}
Let $\gamma_i$ be closed cones in $N_\bR$ for $i=1,2$. For any objects $E_i\in \lSh_{T^n\times \gamma_i}(T^n)$ for  $i=1,2$, we have
\begin{align}
\musupp(E_1\star E_2)&\subset T^n\times (\gamma_1\cap \gamma_2),\\
\musupp(\cHom^\star(E_1,E_2))&\subset T^n \times (\gamma_1\cap \gamma_2).
\end{align}
\end{proposition}
\begin{proof}
This is a version of \cite[Corollary 4.14]{GS}. By the definition of $\star$, Lemma \ref{mestimate}, and Proposition \ref{ssestimate},
\begin{equation}
\begin{split}
\musupp(E_1\star E_2)&\subset m_\#\musupp(E_1\boxtimes E_2)\\
&\subset m_\#(\musupp(E_1) \times \musupp(E_2))\\
&\subset m_\#(T^n\times \gamma_1\times T^n\times \gamma_2)\\
&\subset T^n\times (\gamma_1\cap \gamma_2).
\end{split}
\end{equation}

Similarly, by using (\ref{explicit2}), Lemma \ref{mestimate}, and Proposition \ref{ssestimate},  we have
\begin{equation}
\begin{split}
\musupp(\cHom^\star(E_1,E_2))&\subset m_\#\musupp(\cHom(p_2^{-1}(-1)^{-1}E_1,p_1^!E_2))\\
&\subset m_\#(\musupp(E_2)\times \musupp(E_1))\\
&\subset T^n\times (\gamma_1\cap \gamma_2). 
\end{split}
\end{equation}
This completes the proof.
\end{proof}

As a direct corollary, we have the following, which is a counterpart of {\cite[Proposition 3.17]{GS}. Let $\gamma$ be a closed convex cone in a vector space $M_\bR$ and $\gamma^\vee\subset N_\bR$ be its polar dual:
\begin{equation}
\gamma^\vee:=\lc n\in N_\bR\relmid n(m)\geq 0 \text{ for any } m\in \gamma \rc.
\end{equation}

\begin{corollary}\label{cutoff} For an object $E\in \lSh(T^n)$, we have
\begin{equation}
\musupp(E\star p_!\bC_{\Int(\gamma)})\subset T^n\times (-\gamma^\vee).
\end{equation}
\end{corollary}

We set 
\begin{align}
Z_\gamma&:=T^*M_\bR\bs \lb M_\bR \times \Int(\gamma^\vee)\rb, \\
\hat{Z}_\gamma&:=T^*T^n\bs \lb T^n\times \Int(\gamma^\vee)\rb
\end{align}
where $\Int$ is the interior. We say a closed convex cone $\gamma$ is strictly convex if $\gamma\cap (-\gamma)=\{0\}$.
\begin{proposition}[\cite{Tam}, {\cite[Proposition 4.17]{GS}}]\label{GSprop4.17} Suppose that $\gamma$ is strictly convex. Then
\begin{equation}
\Hom(\bC_\gamma, E)\simeq 0
\end{equation}
for $E\in \lSh_{Z_\gamma}(M_\bR)$.
\end{proposition}

\begin{proposition}\label{prevanishing}
Suppose that $\gamma$ is strictly convex. Then
\begin{equation}
E\star p_!\bC_\gamma\simeq 0
\end{equation}
for $E\in \lSh_{\hat{Z}_\gamma}(T^n)$.
\end{proposition}
\begin{proof}
Consider cone of the restriction map $\bC_\gamma\rightarrow \bC_0$, we have an exact triangle
\begin{equation}
E\star p_!\bC_\gamma\rightarrow E\rightarrow E\star p_!\bC_{\gamma\bs\{0\}}\rightarrow.
\end{equation}
Then Proposition \ref{cutoff} implies that $\musupp(E\star p_!\bC_{\gamma})\subset \hat{Z}_\gamma$. On the other hand, we can prove $E\star p_!\bC_\gamma$ is in the left orthogonal of $\lSh_{\hat{Z}_\gamma}(T^n)$ in $\lSh(T^n)$. In fact, for $F\in \lSh_{\hat{Z}_\gamma}(T^n)$, we have 
\begin{equation}
\hom(E\star p_!\bC_\gamma, F)\simeq \hom(p_!\bC_\gamma, \cHom^\star(E, F))\simeq \hom(\bC_\gamma, p^!\cHom^\star(E, F)).
\end{equation}
Since $\musupp(p^!\cHom^*(E,F))\subset Z_\gamma$, vanishing of the right hand side follows from Proposition \ref{GSprop4.17}. Thus $E\star p_!\bC_\gamma$ belongs to both $\lSh_{\hat{Z}_\gamma}(T^n)$ and the left orthogonal, and this gives $E\star p_!\bC_\gamma\simeq 0$.
\end{proof}

\begin{proposition}\label{vanishing}Let $\gamma_0$ be a closed full-dimensional convex rational polyhedral cone in $M_\bR$. Let $\gamma'$ be a closed rational polyhedral cone such that $\gamma_0\subset \gamma'$ and $\gamma'^\vee$ is  not contained in any face of $\gamma_0^\vee$. 
For an object $E\in \lSh_{\hat{Z}_{\gamma_0}}(T^n)$, we have
\begin{equation}
E\star p_!\bC_{\gamma'}\simeq 0.
\end{equation}
\end{proposition}
\begin{proof} Since $\gamma_0^\vee\supset \gamma'^\vee$, we have $Z_{\gamma_0}\subset Z_{\gamma'}$. Then $E\in \lSh_{\hat{Z}_{\gamma'}}(T^n)$. Hence the statement has already been proved in Proposition \ref{prevanishing} with the assumption $\gamma'$ is strictly convex. 

Let us assume that $\gamma'$ is not strictly convex. We will prove by induction on the codimension of $\gamma'^\vee$. First, we assume $\codim\gamma'^\vee=1$. Then $\gamma'^\vee$ divides some full-dimensional cone $\gamma_1^\vee$ in $\gamma_0^\vee$ to two full-dimensinal closed cones $\gamma_2^\vee$ and $\gamma_3^\vee$, i.e., $\gamma_2^\vee\cap \gamma_3^\vee=\gamma'^\vee$ and $\gamma_2^\vee\cup \gamma_3^\vee=\gamma_1^\vee\subset\gamma_0^\vee$. Then we have $\hat{Z}_{\gamma_0}\subset \hat{Z}_{\gamma_1}=\hat{Z}_{\gamma_2}\cap \hat{Z}_{\gamma_3}$. It follows that $E\in \lSh_{\hat{Z}_{\gamma_0}}(T^n)$ satisfies
\begin{equation}
E\star p_!\bC_{\gamma_i}\simeq 0
\end{equation}
for $i=1,2,3$. Since we have an exact triangle
\begin{equation}\label{resolution}
\bC_{\gamma'}\rightarrow \bC_{\gamma_2}\oplus \bC_{\gamma_3}\rightarrow \bC_{\gamma_1}\rightarrow,
\end{equation}
$E\star p_!\bC_{\gamma'}\simeq 0$ holds.

Suppose that the statement holds up to $\codim=k-1\geq 1$. Since $\gamma'^\vee\subset \gamma_0^\vee$ is not a face, we can take ($k-1$)-codimensional cones $\gamma_i^\vee\subset \gamma_0^\vee$ for $i=1,2,3$ such that $\gamma_2^\vee\cap \gamma_3^\vee=\gamma'^\vee$ and $\gamma_2^\vee\cup \gamma_3^\vee=\gamma_1^\vee$. We can take a resolution of $\gamma_1^\vee$ by $\gamma_2^\vee$ and $\gamma_3^\vee$ as in (\ref{resolution}). Moreover, all those cones can be taken as being not faces of $\gamma_0^\vee$. Then the statement can be shown by an iteration of arguments similar to the previous paragraph.
\end{proof}

To use in Section \ref{generalcase}, we give some calculations of $\cHom^\star$ on $M_\bR$. Let $\sigma$ be a convex cone in $N_\bR$, $E$ be a constructible sheaf over $M_\bR$, $\tilde{p}_i, \tilde{m}\colon M_\bR\times M_\bR\rightarrow M_\bR$ ($i=1,2$) be the projections and the multiplication. We set
\begin{equation}
\begin{split}
E\star_\bR F&:=\tm_!(E\boxtimes F) =\tilde{m}_!(\tilde{p}_1^{-1}E\otimes \tilde{p}_2^{-1}F)\\
\cHom^{\star_\bR}(E,F)&:=\tilde{p}_{1*}\cHom(\tilde{p}_2^{-1}E, \tilde{m}^!F)
\end{split}
\end{equation}
for $E,F\in \lSh(M_\bR)$. As noted in the proof of Lemma \ref{staradjunction}, Lemma \ref{staradjunction} holds for $\star_\bR$ as proved in \cite[Lemma 4.10]{GS}. We write $\bD$ for the Verdier duality functor.
\begin{lemma}\label{Dformula}
There exists a quasi-isomorphism
\begin{equation}
\cHom^\star(E, \bC_{\Int(\sigma^\vee)})\simeq \tilde{m}_*(-1)^*\bD(\tilde{p}_2^{-1}E\otimes \tilde{p}_1^{-1}\bC_{-\sigma^\vee})\\
\end{equation}
\end{lemma}
\begin{proof}
We have
\begin{equation}
\begin{split}
\cHom^\star(\tE, \bC_{\Int(\sigma^\vee)})&:=\tm_*\cHom(\tp_2^{-1}(-1)^*E, \tp_1^!\bC_{\Int(\sigma^\vee)})\\
&\simeq \tm_*\bD(\tp_2^{-1}(-1)^*E\otimes \bD\tp_1^!\bC_{\Int(\sigma^\vee)})\\
&\simeq  \tm_*(-1)^*\bD(\tp_2^{-1}E\otimes \tp_1^{-1}\bC_{-\sigma^\vee})
\end{split}
\end{equation}
This completes the proof.
\end{proof}

Let $\{v_{\rho_1},..., v_{\rho_s}\}$ and $\{v_{\upsilon_1},...,v_{\upsilon_t}\}$ be sets of ray generators of rays (not necessarily being edges) in $\sigma$. Let $n_1,...,n_s,l_1,...,l_t$ be real numbers such that the subset $D$ in $M_\bR$ defined by the inequalities
\begin{equation}
\la m, v_{\rho_i}\ra> n_i \text{ for } i=1,...,s
\end{equation}
and
\begin{equation}
\la m, v_{\upsilon_j}\ra \leq l_i \text{ for } j=1,...,t
\end{equation}
is bounded and contained in a fundamental domain of $M_\bR\rightarrow T^n$. We further assume that there are no redundancy on $v$'s. We also write $\bD D$ for the set defined by the inequalities
\begin{equation}\label{DDclosed}
\la m, v_{\rho_i}\ra\geq  n_i \text{ for } i=1,...,s
\end{equation}
and
\begin{equation}\label{DDopen}
\la m, v_{\upsilon_j}\ra < l_j \text{ for } j=1,...,t.
\end{equation}
Then $-\bD D$ is defined by
\begin{equation}\label{-DDclosed}
\la m, v_{\rho_i}\ra\leq  -n_i \text{ for } i=1,...,s
\end{equation}
and
\begin{equation}\label{-DDopen}
\la m, v_{\upsilon_j}\ra >- l_j \text{ for } j=1,...,t.
\end{equation}

\begin{lemma}\label{polytopedual}
There exists a quasi-isomorphism
\begin{equation}
\cHom^\star(\bC_{D}, \bC_{\Int(\sigma^\vee)})\simeq \bC_{-\bD D}.
\end{equation}
\end{lemma}
\begin{proof}
First, note that $\tilde{p_2}^{-1}\bC_{D}\otimes \tilde{p}_1^{-1}\bC_{-\sigma^\vee}\simeq \bC_{D\times (-\sigma^\vee)}$. Hence we have
\begin{equation}
(-1)^*\bD (\tilde{p}_2^{-1}\bC_{D}\otimes \tilde{p}_1^{-1}\bC_{-\sigma^\vee})\simeq \bC_{-\bD D\times \Int(\sigma^\vee)}.
\end{equation}
We will prove that $\tilde{m}_!\bC_{-\bD D\times \Int(\sigma^\vee)}\simeq \bC_{-\bD D}$. To simplify the notation, we set $C:=-\bD D$. Then it suffices to show that the following: For $r\in M_\bR$, the set 
\begin{equation}
C_r:=\tilde{m}^{-1}(r) \cap (C\times \Int(\sigma^\vee)).
\end{equation}
is open and nonempty in $\tilde{m}^{-1}(r)$ if and only if $r\in C$, and if $C_r$ is not open and not empty then $C_r$ is locally closed and not closed. In particular, if $C_r$ is not open, then the cohomology of $\bC_{C_r}$ vanishes.

Assume that $r\in C$. Let $(x,y)$ be a point in $\tilde{m}^{-1}(r)$ but not in $C\times \Int(\sigma^\vee)$. Since the condition (\ref{-DDopen}) and $y\in \Int(\sigma^\vee)$ are open conditions, we only have to take care of (\ref{-DDclosed}). Assume $\la x, v_{\rho_i}\ra=-n_i$ for some $i$. Since $r\in C$, we also have $\la r, v_{\rho_i}\ra=\la x+y, v_{\rho_i}\ra \leq -n_i$. Then we have $\la y, v_{\rho_i}\ra <0$, but this implies $y\not\in \Int(\sigma^\vee)$. Hence $C_r$ is open in this case. Take $v\in \Int(\sigma^\vee)$. By choosing sufficiently small $\epsilon$, the vector $r-\epsilon v$ satisfies (\ref{DDopen}) and $\la r-\epsilon v, v_{\rho_i}\ra<-n_i$ for any $i$. This implies $(r-\epsilon v, \epsilon v)\in C_r$, and hence $C_r$ is nonempty.

Conversely, assume that $r\not\in C$. Then we have (i) $\la r, v_{\upsilon_j}\ra \leq -l_j$ for some $j$ or (ii) $\la r, v_{\rho_i}\ra >-n_i$ for some $i$. If $(i)$ holds, the vector $(x, y)\in C_r$ satisfies
\begin{equation}
\la x, v_{\upsilon_j}\ra \leq -l_j-\la y, v_{\upsilon_j}\ra 
\end{equation}
for some $j$. Since $\la y, v_{\upsilon_j}\ra >0$, we have $\la x, v_{\upsilon_j}\ra <-l_j$, but this contradicts to $x\in C$. Hence $C_r$ is empty.

Hereafter we assume that $C_r$ is nonempty, then (i) never occurs. Hence (ii) holds. Let $I$ be the subset of $\{1,..., s\}$ consisting of $i$ satisfying 
\begin{equation}
\la r, v_{\rho_i}\ra > -n_i.
\end{equation}
Take $(x,y)\in C_r$. Since $\la x, v_{\rho_i}\ra \leq-n_i$ for any $i$, there exists $\epsilon_i\in (0,1)$ such that 
\begin{equation}
\la (1-\epsilon_i)x+\epsilon_i r, v_{\rho_i}\ra =-n_i
\end{equation} 
for each $i\in I$. Let $\epsilon_0$ be the smallest one among $\epsilon_i$'s. Then the vector $(1-\epsilon)x+\epsilon r$ for $\epsilon\in (0, \epsilon_0]$ is contained in $C$, but not contained in $C$ for $\epsilon>\epsilon_0$. Hence $((1-\epsilon_0)x+\epsilon_0 r, (1-\epsilon_0)y)$ is the boundary point of $C_r$. Hence $C_r$ is not open. On the other hand, if we proceed $\epsilon$ to the minus direction, the condition (\ref{DDclosed}) is stable for $(1-\epsilon)x+\epsilon r$. Since $C$ is bounded, the vector $(1-\epsilon)x+\epsilon r$ eventually violates (\ref{DDopen}). The largest one among such $\epsilon$'s gives a boundary point of $C_r$ which is not contained $C_r$. Hence $C_r$ is not closed. Hence we have $\tm_!\bC_{-\bD D\times \Int(\sigma^\vee)}\simeq \bC_{-\bD D}$.

In Lemma \ref{Dformula}, we set $E:=\bC_D$. Since $\bC_D$ is compactly-supported, we have
\begin{equation}
\begin{split}
\cHom^\star(E, \bC_{\Int(\sigma^\vee)})&\simeq m_*(-1)^*\bD(\tilde{p}_2^{-1}E\otimes \tilde{p}_1^{-1}\bC_{-\sigma^\vee})\\
&\simeq m_!(-1)^*\bD(\tilde{p}_2^{-1}E\otimes \tilde{p}_1^{-1}\bC_{-\sigma^\vee})\\
&\simeq m_!\bC_{-\bD D\times \Int(\sigma^\vee)}\\
&\simeq \bC_{-\bD D}.
\end{split}
\end{equation}
This completes the proof.
\end{proof}

We give one more lemma which is a counterpart of $\cHom_{\cO_X}(\cE, \cF)\simeq \cHom_{\cO_X}(\cE, \cO)\otimes_{\cO_X}\cF$ on the coherent side.
\begin{lemma}\label{tensorfactor}
For $E, F\in \lSh(M_\bR)$ and $G\in \cSh(M_\bR)$ such that $G$ is compactly-supported, there exists a quasi-isomorphism
\begin{equation}
\cHom^{\star_\bR}(E, F\star_\bR G)\simeq \cHom^{\star_\bR}(E, F)\star G.
\end{equation}
\end{lemma}
\begin{proof}
We first prepare the notation. Let $M_i$ ($i=1,..,6$) be copies of $M_\bR$. Then we set the multiplication maps
\begin{align}
m_{12}&\colon M_1\times M_2\rightarrow M_3\\
m_{14}&\colon M_1\times M_4\rightarrow M_5\\
m_{34}&\colon M_3\times M_4\rightarrow M_6\\
m_{25}&\colon M_2\times M_5\rightarrow M_6\\
\end{align}
and projections
\begin{align}
p_{ij}^k &\colon M_{i}\times M_{j}\times M_k\rightarrow M_{i}\times M_{j}\\
p_i^j& \colon M_i\times M_j\rightarrow M_i\\
p_{i}^{jk}&\colon M_{i}\times M_{j}\times M_k\rightarrow M_{i}
\end{align}
and identities $\id_i\colon M_i\rightarrow M_i$.

Then we have
\begin{equation}
\begin{split}
\cHom^{\star_\bR}(E, F)\star_\bR G&\simeq m_{34!}((p_3^4)^{-1}m_{12*}\cHom((p_2^1)^{-1}(-1)^*E, (p_1^2)^!F)\otimes (p_4^3)^{-1}G)\\
&\simeq m_{34*}((p_3^4)^{-1}m_{12*}\cHom((p_2^1)^{-1}(-1)^*E, (p_1^2)^!F)\otimes (p_4^3)^{-1}G)
\end{split}
\end{equation}
by the assumption that $G$ is compactly-supported.

Since we have the pull-back diagram
\begin{equation}
\xymatrix{
M_1\times M_2\times M_4 \ar[d]_{m_{12}\times \id} \ar[r]^{p_{12}^4}& M_1\times M_2 \ar[d]^{m_{12}}\\
M_3\times M_4\ar[r]_{p_3^4}&M_3,
}
\end{equation}
the base change implies
\begin{equation}
\begin{split}
m_{34*}&((p_3^4)^{-1}m_{12*}\cHom((p_2^1)^{-1}(-1)^*E,  (p_1^2)^!F)\otimes (p_4^3)^{-1}G)\\
&\simeq m_{34*}((m_{12}\times \id_4)_*((p_{12}^4)^{-1}\cHom((p_2^1)^{-1}(-1)^*E, (p_1^2)^!F))\otimes (p_4^3)^{-1}G)\\
&\simeq m_{34*}(m_{12}\times \id_4)_*((p_{12}^4)^{-1}\cHom((p_2^1)^{-1}(-1)^*E, (p_1^2)^!F))\otimes (m_{12}\times \id_4)^{-1}(p_4^3)^{-1}G).
\end{split}
\end{equation}
By using $p_4^3\circ (m_{12}\times \id_4)=p_4^{12}$, $p_2^5\circ (m_{14}\times \id_2)=p_2^{14}$, and  $m_{34}\circ (m_{12}\times \id_4)= m_{25}\circ(m_{14}\times \id_2)$, we can further calculate as
\begin{equation}
\begin{split}
m_{34*}(m_{12}\times \id_4)_*&((p_{12}^4)^{-1}\cHom((p_2^1)^{-1}(-1)^*E, (p_1^2)^!F))\otimes (m_{12}\times \id_4)^{-1}(p_4^3)^{-1}G)\\
&\simeq m_{34*}(m_{12}\times \id_4)_*((p_{12}^4)^{-1}\cHom((p_2^1)^{-1}(-1)^*E, (p_1^2)^!F))\otimes (p_4^{12})^{-1}G)\\
&\simeq m_{34*}(m_{12}\times \id_4)_*((p_{12}^4)^{-1}\cHom((p_2^1)^{-1}(-1)^*E, (p_1^2)^{-1}F[n]))\otimes (p_4^{12})^{-1}G)\\
&\simeq m_{34*}(m_{12}\times \id_4)_*(\cHom((p_2^{14})^{-1}(-1)^*E, (p_1^{24})^{-1}F[n]))\otimes (p_4^{12})^{-1}G)\\
&\simeq m_{34*}(m_{12}\times \id_4)_*(\cHom((p_2^{14})^{-1}(-1)^*E, (p_{14}^2)^{-1}((p_1^{4})^{-1}F[n]\otimes (p_4^{1})^{-1}G)))\\
&\simeq m_{25*}(m_{14}\times \id_2)_*(\cHom((p_2^{14})^{-1}(-1)^*E, (p_{14}^2)^{!}((p_1^{4})^{-1}F\otimes (p_4^{1})^{-1}G)))\\
&\simeq m_{25*}(m_{14}\times \id_2)_!(\cHom((m_{14}\times \id_2)^{-1}(p_{2}^5)^{-1}(-1)^*E, (p_{14}^2)^{!}((p_1^{4})^{-1}F\otimes (p_4^{1})^{-1}G)))\\
&\simeq m_{25*}(\cHom((p_2^5)^{-1}E, (p_{5}^2)^{!}(F\star_\bR G))\\
&\simeq \cHom^{\star_\bR}(E, F\star_\bR G)\\
\end{split}
\end{equation}
This completes the proof.
\end{proof}

\begin{corollary}\label{cptadjunction}
For $F\in \lSh(M_\bR)$ and compactly-supported $E\in \cSh(M_\bR)$, there exits a quasi-isomorphism
\begin{equation}
\hom(E, F\star \bC_{D})\simeq \hom(E\star \bC_{-\bD D}, F\star \bC_{\Int(\sigma^\vee)}).
\end{equation}
\end{corollary}
\begin{proof}
Let $\cS$ be a stratification of $M_\bR$ which refines $\musupp(F)$. Then we can take $F_i \in \cSh_\cS(M_\bR)$ such that $\varhocolim F_i\simeq F$.

Let $B_r$ ($r\in \bZ$) be a ball in $M_\bR$ with radius $r$ and $i_r\colon B_r\hookrightarrow M_\bR$ be an open inclusion. Then we have a sequence $\{F^r:=i_{r!}i_r^{-1}F\}_{r\in \bN}$ satisfying $\varhocolim{r} F_r\simeq F$ where the colimit is taken with respect to the morphisms which correspond to the identity via the isomorphisms
\begin{equation}
\hom(i_{r!}i_{r}^{-1}F, i_{r'!}i_{r'}^{-1}F)\simeq \hom(i_{r!}i_{r}^{-1}F,i_{r!}i_{r}^{-1}F)
\end{equation}
for $r<r'$. Then $\varhocolim{i, r}F_i^r\simeq F$.

Set $\Lambda:=\musupp(E)\cup \bigcup_{i,r}\musupp(F^i_r\star_\bR \bC_D)$. Then $E$ is compact in $\lSh_{\Lambda}(T^n)$. By Lemma \ref{polytopedual} and Lemma \ref{tensorfactor}, we have
\begin{equation}
\begin{split}
\hom(E, F\star_\bR \bC_D)&\simeq\varhocolim{i,r}\hom(E, F^i_r\star_\bR \bC_D)\\
&\simeq \varhocolim{i,r}\hom(E, F^i_r\star_\bR \cHom^\star_\bR(\bC_{-\bD D}, \bC_{\Int(\sigma^\vee)}))\\
&\simeq \varhocolim{i,r}\hom(E, \cHom^\star_\bR(\bC_{-\bD D}, F^i_r\star_\bR \bC_{\Int(\sigma^\vee)}))\\
&\simeq \varhocolim{i,r}\hom(E\star_\bR \bC_{-\bD D}, F^i_r\star_\bR \bC_{\Int(\sigma^\vee)})\\
&\simeq \hom(E\star_\bR \bC_{-\bD D}, F\star_\bR \bC_{\Int(\sigma^\vee)}).
\end{split}
\end{equation}
This completes the proof.
\end{proof}

\section{Preliminaries on toric stacks}\label{toricstack}
Stacky generalization of toric varieties has been considered by various authors. Borisov--Chen--Smith \cite{BCS} defined a certain combinatorial object called {\em stacky fan} and used it to construct smooth Deligne--Mumford stacks called {\em toric DM stacks}. Later, Iwanari \cite{Iwanari} and Fantechi--Mann--Nironi \cite{FMN} give an intrinsic characterization of toric DM stacks of \cite{BCS}. However, toric DM stacks of \cite{BCS} does not include singular toric varieties. Tyomkin \cite{Tyomkin} generalized \cite{BCS} to include any toric varieties via local construction. Geraschenko--Satoriano \cite{GStoric} proposed more generalized version and gives a global construction as quotient stacks.

In this paper, we will use a special class of toric stacks in the sense of Tyomkin, but we follow the construction of Geraschenko--Satoriano which is easier to describe.

Let $L, N$ be free abelian groups and $\beta\colon L\rightarrow N$ be a homomorphism with a finite cokernel. Let $L^\vee$ and $M$ be the dual of $L$ and $N$ respectively.
Let further $\hat{\Sigma}$ and $\Sigma$ be finite fans consisting of rational strictly convex cones defined in $L_\bR:=L\otimes_\bZ\bR$ and $N_\bR$. As noted in Introduction, we always asuume Condition \ref{condition} on $\beta_\bR$: The map $\beta_\bR:=\beta\otimes_\bZ\bR$ induces a {\em combinatorial isomorphism} between $\hat{\Sigma}$ and $\Sigma$, i.e., $\beta_\bR(\hsigma)$ conicides with an element of $\Sigma$ for any $\hsigma\in\hSigma$ and the assignment $\hSigma\in \hsigma\mapsto\beta_\bR(\hsigma)\in \Sigma$ defines an isomorphism of posets between $\hSigma$ and $\Sigma$.
\begin{lemma}\label{combinatorial} Let $\beta, L, N$ as above. Then
\begin{enumerate}
\item $\dim\hsigma=\dim\beta_\bR(\hsigma)$, and
\item $\beta|_{\spanning{\hsigma}\cap L}$ is injective for any $\hsigma\in \hSigma$.
\end{enumerate}
\end{lemma}
\begin{proof}
Note that the dimension of a cone in a fan is inductively determined by the structure of the fan. This assertion and $\beta_\bR(0)=0$ imply (i).

If there is a nontrivial element in $\ker(\beta|_{\hsigma\cap L})$, this element gives a nontrivial kernel of $\beta_{\bR}|_{\spanning(\hsigma)}$ which contradicts to (i). This proves (ii).
\end{proof}

The map $\beta$ induces a surjective homomorphism $\T_\beta:=\beta\otimes_\bZ\bC^*$ between two tori
\begin{equation}
\T_\beta\colon \T_L:=L\otimes_\bZ\bC^*\rightarrow \T_N:=N\otimes_\bZ\bC^*.
\end{equation}
We set $G_\beta:=\ker\lb\T_\beta\rb$. We abbreviate $(L,N,\hSigma,\Sigma,\beta)$ as $(\hSigma,\beta)$.

\begin{definition}\label{ourtoricstack}
The toric stack $\cX_{\hSigma,\beta}$ associated to the above data $(\hat{\Sigma},\beta)$ is defined as the quotient stack
\begin{equation}
\cX_{\hSigma,\beta}:=\left[X_{\hat{\Sigma}}\left.\right/G_\beta\right].
\end{equation}
\end{definition}

\begin{remark}
\begin{enumerate}
\item Our data $(\hSigma,\beta)$ induces Tyomkin's {\em toric stacky data} \cite{Tyomkin}.
\item This definition without the assumption of combinatorial isomorphism is due to Geraschenko--Satoriano \cite{GStoric}. Moreover, our definition is a special case of their notion of {\em fantastack} \cite{GStoric}.
\end{enumerate}
\end{remark}

\begin{lemma}
The toric stack $\cX_{\hSigma,\beta}$ is a Deligne--Mumford stack without generic stabilizers.
\end{lemma}
\begin{proof} This lemma follows from the argument of Section 4 of Tyomkin \cite{Tyomkin}, but we reproduce it here for reader's convenience.

Since $G_\beta$ acts on $X_{\hSigma}$ faithfully, $\cXsb$ has no generic stabilizers. By Lemma \ref{finite} below, the action of $G_\beta$ has only finite stabilizers.
Since we are working over $\bC$, any finite group schemes are reduced. Then the statement follows from \cite[Corollary 2.2]{Edidin}. 
\end{proof}

\begin{example}
\begin{enumerate}
\item Let $\Sigma$ be a fan defined in $N_\bR$. We set $L=N$ and $\beta=\id$. Then $\cX_{\Sigma,\beta}=X_\Sigma$. Hence Definition \ref{ourtoricstack} includes all toric varieties.
\item Let $\Sigma$ be a simplicial fan defined in $N_\bR$. We write $\Sigma(1):=\{\rho_1,..,\rho_s\}$ for the set of 1-dimensional cones in $\Sigma$. Set $L:=\bigoplus_{i=1}^s\bZ\cdot e_i$ and define $\beta$ by choosing $\beta(e_i)\in \rho_i\cap N$ for $i=1,...,s$. Let $\hat{\Sigma}$ be the fan in $L_\bR$ consisting of cones 
\begin{equation}
\hat{\sigma}:=\Cone\lb e_i\relmid e_i\in \sigma\rb
\end{equation}
for any $\sigma\in \Sigma$. Then $\beta$ induces a combinatorial isomorphism between $\hSigma$ and $\Sigma$. The resulting toric stack is a toric DM stack without generic stabilizers in the sense of \cite{BCS}. As noted in Introduction, this is equivalent to toric orbifold in the sense of \cite{FLTZ3}.
\end{enumerate}
\end{example}

The toric stack $\cX_{\hSigma,\beta}$ can be written as the union of $[U_\hsigma/G_\beta]$ for all $\hsigma\in \hat{\Sigma}$, where $U_\hsigma$ is the affine toric subvariety of $X_\hSigma$ corresponding to $\hsigma$. We set $\cU_\sigma:=[U_\hsigma/G_\beta]$.

\section{The coherent-constructible correspondence}\label{cccformulation}
Let $(\hSigma, \beta)$ be a stacky fan satisfying Condition \ref{condition} and $\rank N=n$. We set $M:=\Hom_\bZ(N,\bZ)$. The cotangent bundle of the torus $M_\bR/M=:T^n$ has a canonical trivialization $T^*T^n\cong T^n\times N_\bR$.

Since $\beta$ has finite cokernel, the induced map $\beta_\bR$ is surjective. Hence there exists induced isomorphism $[\beta_\bR]\colon L_\bR/\ker(\beta_\bR)\cong N_\bR$. We set $N_\beta:=[\beta_\bR](L/L\cap \ker(\beta_\bR))$. We also set the dual of $\beta$ by $\beta^\vee$, then we have $\beta^\vee_\bR\colon M_\bR\xrightarrow{\cong} \beta^\vee_\bR(M_\bR)$. We set $M_\beta:=(\beta_\bR^\vee)^{-1}(L^\vee)$.

\begin{lemma}
\begin{enumerate}
\item There exists a canonical inclusion $M\subset M_\beta$.
\item The natural pairing induces an isomorphism $N_\beta^\vee\cong M_\beta$.
\end{enumerate}
\end{lemma}
\begin{proof}
Since $\beta^\vee_\bR$ is an inclusion $M_\bR\hookrightarrow L_\bR$, we only have to check that the pairing of $\beta_\bR^\vee(M)$ with $L$ is in integer. This is true since $\beta_\bR^\vee\cong \beta^\vee\otimes_\bZ\bR$. This completes the proof of (i).

Since $[\beta_\bR]$ is an isomorphsim, we have
\begin{equation}
N_\beta^\vee\cong (L/L\cap \ker(\beta_\bR))^\vee\cong L^\vee\cap \image(\beta_\bR^\vee)\cong (\beta_\bR^\vee)^{-1}(L^\vee).
\end{equation}
This completes the proof.
\end{proof}

Let $\la -,-\ra\colon M\times N\rightarrow \bZ$ be the natural pairing. For $\sigma\in \Sigma$, we set 
\begin{equation}
\begin{split}
\sigma^\perp&:=\lc m\in M_\bR\relmid \la m, s\ra=0 \text{ for any } s\in \sigma\rc, \\
\sigma^\vee&:=\lc m\in M_\bR\relmid \la m, s\ra\geq 0 \text{ for any } s\in \sigma\rc. \\
\end{split}
\end{equation}
We also define $\hsigma^\vee\subset L^\vee_\bR$ in a similar way. Then we have an induced inclusion map $M_\bR/\spanning(\sigma)^\perp\hookrightarrow L^\vee_\bR/\spanning(\hsigma)^\perp$. We define $M_{\sigma,\beta}$ the inverse image of $L^\vee/\spanning(\hsigma)^\perp$ under this induced map.

For $\chi\in M_{\sigma,\beta}$, we also set
\begin{equation}
\begin{split}
\sigma^\perp_\chi&:=\sigma^\perp+\chi, \\
\sigma^\vee_\chi&:=\sigma^\vee+\chi. \\
\end{split}
\end{equation}

Let $p$ be the quotient map $M_\bR\rightarrow T^n$. Both $p(\sigma_\chi^\perp)$ and $p(\sigma_\chi^\vee)$ depend only on the class of $\chi$ in $M_{\sigma, \beta}/M$. We write $[\chi]$ for the class of $\chi$ in $M_{\sigma,\beta}/M$.

We set 
\begin{equation}\label{Shard}
\Lambda_{\hSigma,\beta}:=\bigcup_{\sigma\in \Sigma}\bigcup_{[\chi]\in M_{\sigma,\beta}/M}p(\sigma_\chi^\perp)\times(-\sigma) \subset T^*T^n.
\end{equation}
The main result in this paper is the following.
\begin{theorem}\label{main}
There exists an equivalence of $\infty$-categories
\begin{equation}
\Indcoh\cX_{\hSigma,\beta}\simeq \lSh_{\Lambda_{\hSigma,\beta}}(T^n).
\end{equation}
\end{theorem}

\section{Affine case}\label{affinecase}
Let $(\hSigma(\hsigma),\beta)$ be as in Section \ref{toricstack} with the following additional assumption: $\hSigma(\hsigma)$ is a fan consisting of faces of a single rational strictly convex cone in $L_\bR$. Take $\chi\in M_{\sigma, \beta}$. Since $M_{\sigma,\beta}\subset L^\vee$, $\tilde{\chi}$ a lift of $\chi$ defines a character of $\T_L$. We define $\cO_{\cX_{\hSigma(\hsigma),\beta}}(\tilde{\chi})$ to be the structure sheaf on $\cX_{\hSigma(\hsigma),\beta}$ twisted by the character $\chi|_{G_\beta}$ of $G_\beta$. 
\begin{lemma}\label{welldef}
The sheaf $\cO_{\cX_{\hSigma(\hsigma),\beta}}(\tilde{\chi})$ depends only on the class $[\chi]$ of $\chi$ in $M_{\sigma, \beta}/M$.
\end{lemma}
\begin{proof}
For $\chi\in M$, a character of $\T_L$ given by the inclusion $M\hookrightarrow M_{\sigma,\beta}\subset L^\vee$ is given by the composition $\T_L\xrightarrow{\T_\beta}\T_N\xrightarrow{\chi}\bC^*$. By the definition of $G_\beta$, this composition is trivial.
\end{proof}

We set
\begin{equation}
\Theta'(\sigma,\chi):=\cO_{\cX_{\hSigma(\hsigma),\beta}}(\tilde{\chi}) \in \Qcoh(\cX_{\hSigma(\hsigma),\beta}),
\end{equation}
which depends only the class $[\chi]$ of $\chi$ in $M_{\sigma,\beta}/M$ by Lemma \ref{welldef}.
For a face $\tau$ of $\sigma$, we have an open substack $\cU_\tau$ of $\cX_{\hSigma(\hsigma),\beta}$. The restriction of the sheaf $\Theta'(\sigma,\chi)$ to $\cU_\tau$ is equal to $\Theta'(\tau,\chi)$.

We also set
\begin{equation}
\Theta(\sigma,\chi):=p_!\bC_{\Int(\sigma_\chi^\vee)}[n]\in \lSh_{\Lambda_{\hSs,\beta}}(T^n)
\end{equation}
where $\Int$ denotes the interior. 
This also depends only on $\sigma$ and the class $[\chi]\in M_{\sigma,\beta}/M$ since $p(\sigma_\chi^\vee)$ depends only on $[\chi]$.


\begin{lemma}\label{finite} There exists an isomorphism
\begin{equation}
\begin{split}
\cX_{\hSigma(\hsigma),\beta}:=\left[\Spec\bC[\hsigma^\vee\cap L^\vee]\big/G_\beta\right]
&\cong\left[\Spec\bC[(\sigma^\vee/\sigma^\perp)\cap M_{\sigma,\beta}]\big/H_\beta\right].
\end{split}
\end{equation}
where $H_\beta:=G_\beta''$ is defined in the proof, which is a finite group.
\end{lemma}
\begin{proof}
First, we note that there exists a splitting
\begin{equation}
L^\vee\cong \big(L^\vee\cap \hsigma^\perp\big)\oplus \big(L^\vee/L^\vee\cap \hsigma^\perp\big),
\end{equation}
which is the dual of the splitting
\begin{equation}
L\cong \big(L\cap\spanning(\hsigma)\big)\oplus (L/L\cap \spanning(\hsigma)).
\end{equation}
This split also induces a split $\beta\cong\beta'\oplus \beta''$ where $\beta'\colon L_\bR/\spanning(\hsigma)\rightarrow N_\bR/\spanning(\sigma)$ and $\beta''\colon \spanning(\hsigma)\rightarrow \spanning(\sigma)$.
According to these split, we also have a split $G_\beta\cong G_\beta'\times G_\beta''$.
Furthermore, we get a splitting
\begin{equation}
\hsigma^\vee\cap L^\vee=\Big(\hsigma^\perp\cap L^\vee\Big)\oplus \Big(\big(\hsigma^\vee\cap L^\vee\big)/\big(L^\vee\cap\hsigma^\perp\big)\Big),
\end{equation}
and 
\begin{equation}\label{product}
\Spec\bC[\hsigma^\vee\cap L^\vee]\cong \Spec\bC[\hsigma^\perp\cap L^\vee]\times \Spec\bC[\big(\hsigma^\vee \cap L^\vee\big)/\big(L^\vee\cap \hsigma^\perp\big)].
\end{equation}
The first component is nothing but $G_\beta'$.

Hence
\begin{equation}\begin{split}
\left[\Spec\bC[\hsigma^\vee\cap L^\vee]\big/G_\beta\right]&\cong\left[\Spec\bC[\big(\hsigma^\vee \cap L^\vee\big)/\big(L^\vee\cap \hsigma^\perp\big)]\big/G_\beta''\right].
\end{split}
\end{equation}
Moreover, since $\sigma^\vee/\sigma^\perp\cong \hsigma^\vee/\hsigma^\perp$ via $\beta^\vee_\bR$, we have $\big(\hsigma^\vee \cap L^\vee\big)/\big(L^\vee\cap \hsigma^\perp\big)\cong (\sigma^\vee/\sigma^\perp)\cap M_{\sigma,\beta}$.
\end{proof}

Let $\check{H}_\beta$ be the set of characters of $H_\beta$.

\begin{lemma}\label{character} There exists a natural identification $M_{\sigma,\beta}/M\cong \check{H}_\beta$.
\end{lemma}
\begin{proof}
The group $H_\beta:=G_\beta''$ is the kernel of $\T_{\beta''}$. Since the dual of $\beta''_\bR$ is $M_\bR/\spanning(\sigma)^\perp\rightarrow L^\vee_\bR/\spanning(\hsigma)^\perp$, the statement follows.
\end{proof}

Let $\tau$ be a face of $\sigma$. Take $\chi_i\in M_{\sigma, \beta}$ for $i=1,2$. 
\begin{proposition}\label{homtheta} 
There exist canonical isomorphisms:
\begin{equation}
H^i(\hom(\Theta'(\sigma,\chi_1),\Theta'(\tau,\chi_2)))\cong \begin{cases}
\bC[\tau^\vee \cap (M+\chi_2-\chi_1)] &\text{ if $i=0$},\\
0&\text{ otherwise.}
\end{cases}
\end{equation}
\begin{equation}\label{secondisom}
H^i(\hom(\Theta(\sigma,\chi_1),\Theta(\tau,\chi_2)))\cong \begin{cases}
\bC[\tau^\vee \cap (M+\chi_2-\chi_1)] &\text{ if $i=0$},\\
0&\text{ otherwise.}
\end{cases}
\end{equation}
such that the composition of morphisms are expressed by the sum of elements in $M_\bR$. 
\end{proposition}
\begin{proof}
This follows from Lemma \ref{finite}, Lemma \ref{character}, and the argument of \cite[Proposition 2.3]{Tr}, and hence omitted.
\end{proof}

We set $\Theta'(\sigma):=\bigoplus_{[\chi]\in M_{\sigma,\beta}/M}\Theta'(\sigma,\chi)$.
\begin{lemma}\label{generation1}There exists an equivalence of $\infty$-categories
\begin{equation}
\Qcoh(\cX_{\hSs,\beta}) \simeq \Mod\left(\End\left(\Theta'(\sigma)\right)\right).
\end{equation}
\end{lemma}
\begin{proof}
For each $[\chi]\in \check{H}_\beta$, there exists an idempotent $e_\chi\in \bC[(\sigma^\vee/\sigma^\perp)\cap M_{\sigma,\beta}]\rtimes H_\beta$ with $1=\sum_{[\chi]\in \check{H}_\beta}e_\chi$. Then by the definition of $e_\chi$, the $\bC[(\sigma^\vee/\sigma^\perp)\cap M_{\sigma,\beta}]\rtimes H_\beta$-module $\bC[(\sigma^\vee/\sigma^\perp)\cap M_{\sigma,\beta}]\rtimes H_\beta\cdot e_\chi$ corresponds to $\Theta'(\sigma,\chi)$ via the identification $\Qcoh(\cXssb)\simeq \Mod(\bC[(\sigma^\vee/\sigma^\perp)\cap M_{\sigma,\beta}]\rtimes H_\beta$. Hence the quasi-coherent sheaf $\Theta'(\sigma)=\bigoplus_{[\chi]\in\check{H}_\beta}\Theta'(\sigma,\chi)$ corresponds to the $\bC[(\sigma^\vee/\sigma^\perp)\cap M_{\sigma,\beta}]\rtimes H_\beta$-module $\bC[(\sigma^\vee/\sigma^\perp)\cap M_{\sigma,\beta}]\rtimes H_\beta$.
\end{proof}

\begin{lemma}\label{lem:cpt} Let $\Lambda$ be a conic Lagrangian subset in $T^*T^n$ contained in $T^n\times (\sigma)$.
The sheaf $\Theta(\sigma,\chi)$ is compact in $\lSh_{\Lambda}(T^n)$ for any $\chi\in M_{\sigma,\beta}$.
\end{lemma}
\begin{proof}
We have $\Theta(\sigma,\chi)=p_!\bC_{\Int(\sigma^\vee_\chi)}[n]=p_!\bC_{\Int(\sigma^\vee+\chi)}[n]$. For $E\in \lSh_{\Lambda}(T^n)$, we have
\begin{equation}
\hom(\Theta(\sigma_\chi), E)\simeq \Gamma(\Int(\sigma^\vee+\chi), p^!E)[-n].
\end{equation}
Via the canonical trivialization $T^*M_\bR\cong M_\bR\times N_\bR$, the conormal direction of $\musupp(p^!E)$ is contained in $-\sigma$. Choose a splitting $\sigma^\vee=\left(\sigma^\perp\right)\times \left(\sigma^\vee/\sigma^\perp\right)$. Then for any $m\in \left(\sigma^\vee/\sigma^\perp\right)$, the restriction of $p^!E$ to $\left(\sigma^\perp\right)\times\{m\}$ is a derived local system. Hence we have
\begin{equation}
\Gamma(\Int(\sigma^\vee+\chi), p^!E)\simeq \Gamma(\Int(\sigma^\vee+\chi)\cap (\sigma^\vee/\sigma^\perp), p^!E).
\end{equation}

An element $n$ in $\Int(\sigma)$ induces a family of open sets $\{V_s\}_{s>v}$ in $\Int(\sigma^\vee+\chi)\cap (\sigma^\vee/\sigma^\perp)$ where
\begin{equation}
V_s:=\Int(\sigma^\vee+\chi)\cap (\sigma^\vee/\sigma^\perp) \cap \lc m\in M_\bR\relmid \la m, n\ra <s\rc
\end{equation}
and $v$ is an element of $\bR$ satisfying $V_v\neq \varnothing$.
Since the family $\{V_s\}_{s\in\bR}$ satisfies the conditions of the non-characteristic deformation lemma \cite[Proposition 2.7.2]{KS}, we have 
\begin{equation}
\Gamma(\Int(\sigma^\vee+\chi), p^!E)\simeq \Gamma(V_v, p^!E)\simeq \hom(p_!\bC_{V_v}, E). 
\end{equation}
Hence we have an isomorphism of functors 
\begin{equation}\label{eq:functorsnat}
\hom(\Theta(\sigma,\chi), (-))\simeq \hom(p_!\bC_{V_v}, (-))[-n]\colon \lSh_{\Lambda}(T^n)\rightarrow\Mod(\bC).
\end{equation}

Since $p_!\bC_{V_v}$ is compact in $\lSh_{\Lambda\cup \musupp(p_!\bC_{V_v})}(T^n)$, the functor $\hom(p_!\bC_{V_v}, (-))$ is cocontinuous. Hence the restriction of $\hom(p_!\bC_{V_v}, (-))$ to $\lSh_{T^n\times (-\sigma)}(T^n)$ is also cocontinuous. By (\ref{eq:functorsnat}), the statement follows.
\end{proof}

We also set ${\Theta}(\sigma):=\bigoplus_{\chi\in M_{\sigma,\beta}/M}{\Theta}(\sigma,\chi)$. Suppose that $\hsigma$ is smooth. By adding some vectors to the set of ray generators of $\hsigma$, we can take a trivialization $L\cong \bZ^n$. This trivialization induces trivializations $M_{\sigma,\beta}\cong \bZ^n$ and $M_\bR\cong \bR^n$, which are compatible. We also write $M$ for the image of $M$ under this trivialization. Since $\beta$ is a combinatorial isomorphism, the cone $\sigma$ is also identified with $\bR_{\geq 0}^r$ by using the above trivialization where $r\leq n$ is an integer. Let $(\hSigma_{\bA^r\times (\bC^{*})^{n-r}},\id)$ be the standard fan of $\bA^r\times (\bC^*)^{n-r}$, where $\hSigma_{\bA^r\times (\bC^{*})^{n-r}}$ is formed by the set of faces of the first quadrant of $\bR^r$ in $\bR^n$. We set
\begin{equation}
\Lambda_{r,n}:=\bigcup_{\sigma\in \hSigma_{\bA^r\times (\bC^{*})^{n-r}}}\sigma^\perp\times (-\sigma)\subset \bR^n\times \bR^n\cong T^*\bR^n.
\end{equation}

As a result of the above trivializations, we have
\begin{equation}
\Lambda_{\hSigma,\beta}\cong \bigcup_{s\in M_{\sigma,\beta}} dp\lb s+\Lambda_{r,n}\rb
\end{equation}
in $T^*M_\bR/M\cong T^*\bR^n/M$.

\begin{lemma}\label{generation2}
If $\hsigma$ is smooth, the sheaf $\Theta(\sigma)$ compactly generates $\lSh_{\Lsb}(T^n)$.
\end{lemma}
\begin{proof}
The case of $(\Lambda_{\hSigma_{\bA^{r}\times (\bC^*)^{n-r}},\id})$ is proved in \cite{Nadwrapped, IK}. Consider $T^*\bR^n$ and set
\begin{equation}
\Lambda:=\bigcup_{s\in\bZ^n}s+\Lambda_{r,n}.
\end{equation}
For $E\in \lSh_{\Lambda}(\bR^n)$, suppose that $\hom(\bigoplus_{s\in \bZ}\bC_{s+\bR_{>0}^{r}\times \bR^{n-r}},E)\simeq 0$. For $s\in \bZ^n$, we write $t_s\colon \bZ^n\rightarrow \bZ^n$ for the translation by $s$. Then we have
\begin{equation}\label{translation}
0\simeq \hom_{\lSh_{\Lambda}(\bR^n)}(\bC_{s+\bR_{>0}^{r}\times \bR^{n-r}},E)\simeq \hom_{\lSh_{\Lambda}(\bR^n)}(\bC_{\bR_{>0}^{r}\times \bR^{n-r}}, t_{s}^*E).
\end{equation}
Hence we also have
\begin{equation}
\begin{split}
\hom_{\lSh_{\Lambda_{\hSigma_{\bA^n},\id}}(T^n)}(p_!\bC_{\bR_{>0}^{r}\times \bR^{n-r}}, p_!E)&\simeq
\hom_{\lSh_{\Lambda}(\bR^n)}(\bC_{\bR_{>0}^{r}\times \bR^{n-r}}, p^{-1}p_!E)\\ &\simeq
\hom_{\lSh_{\Lambda}(\bR^n)}\left(\bC_{\bR_{>0}^{r}\times \bR^{n-r}}, \bigoplus_{s\in \bZ}t^*_sE\right)\simeq 0
\end{split}
\end{equation}
In the last equality, we used (\ref{translation}) and the fact $\bC_{\bR_{>0}^{r}\times \bR^{n-r}}$ is compact, which is deduced from the same argument as in the proof of \prettyref{lem:cpt}. Since $\lSh_{\Lambda_{\hSigma_{\bA^n},\id}}(T^n)$ is generated by $p_!\bC_{\bR_{>0}^{r}\times \bR^{n-r}}$, we have $p_!E\simeq 0$. This further implies $0\simeq p^{-1}p_!E\simeq \bigoplus_{s\in \bZ}t^*_sE$, hence $E\simeq 0$.

For $E\in \lSh_{\Lsb}(\bR^n/M)$, suppose that $\hom(\Theta(\sigma), E)\simeq 0$. Note that $\bZ^n$ in $\bR^n$ in the above trivialization is mapped to $M_{\sigma,\beta}$ by $p$. Then we have
\begin{equation}
\hom(\bC_{s+\bR_{>0}^{r}\times \bR^{n-r}}, p^{-1}E)\simeq \hom(p_!\bC_{s+\bR_{>0}^{r}\times \bR^{n-r}}, E)\simeq 0.
\end{equation}
Hence $p^{-1}E\simeq 0$. This further implies $E\simeq 0$. Since $\Theta(\sigma)$ is compact by Lemma \ref{lem:cpt}, this completes the proof.
\end{proof}

\begin{proposition}\label{affine}
There exists a fully faithful embedding of $\infty$-categories 
\begin{equation}
\kappa_{\hsigma,\beta}\colon\Qcoh(\cX_{\hSigma,\beta})\hookrightarrow \lSh_{\Lambda_{\hSigma,\beta}}(T^n)
\end{equation}such that
\begin{equation}
\kappa_{\hsigma,\beta}(\Theta'(\sigma, \chi))\simeq\Theta(\sigma,\chi).
\end{equation}
If $\hsigma$ is smooth, this is an equivalence.
\end{proposition}
\begin{proof}
By Proposition \ref{homtheta}, Lemma \ref{generation1}, and Lemma \ref{generation2}, we have a morphism
\begin{equation}\label{sequenceoffunctors}
\begin{split}
\Qcoh(\cXsb)&\simeq \Mod(\oplus_{[\chi]\in M_{\sigma,\beta}/M}\Theta'(\sigma,\chi)))\\
&\simeq \Mod(\oplus_{[\chi]\in M_{\sigma,\beta}/M}\Theta(\sigma,\chi)))\hookrightarrow \lSh_{\Lsb}(T^n).
\end{split}
\end{equation}
By Lemma \ref{generation2}, the last inclusion in (\ref{sequenceoffunctors}) is a compact generataion if $\hsigma$ is smooth. This completes the proof.
\end{proof}

As a consquence of Proposition \ref{homtheta} and Proposition \ref{affine}, we have the following corollary.
\begin{corollary}\label{affinefunctorial}
Let $\tau$ be a face of $\sigma$ and $\chi$ be an element of $M_{\sigma,\beta}$. Then we have
\begin{equation}
\kappa_{\hsigma,\beta}(\Theta'(\tau,\chi))\simeq \Theta(\tau,\chi).
\end{equation}
\end{corollary}

\section{Gluing description on the coherent side}\label{section:gluecoh}
Let $G$ be a reductive group and $X$ be a scheme over $\bC$ with a $G$-action. We set $\cX:=\left[X/G\right]$. Let $\{U_i\}_{i\in I}$ be a Zariski open covering of $X$ which is finite and $G$-invariant. We set $C(I)_0:=2^I\bs\{\varnothing\}$ where $2^I$ is the power set of $I$. We introduce the poset structure on $C(I)_0$ by the inclusion relation. This poset will be called {\em the ${\check{C}ech}$ poset of $I$}. 
We also view the poset as a category by assigning a morphism $\bm{i}_1\rightarrow \bm{i}_2$ to $\bm{i}_1,\bm{i}_2\in C(I)_0$ satisfying $\bm{i}_1\subset \bm{i}_2$. Let further $C(I)$ be the nerve of the $\mathrm{\check{C}ech}$ poset of $I$.

For $\bm{i}\in C(I)_0$, we set $U_{\bm{i}}:=\bigcap_{i\in \bm{i}}U_i$ and $\cU_{\bm{i}}:=\left[ U_{\bm{i}}\big/G\right]$. Let $\iota_{\bm{i}_1\bm{i}_2}$ be the open inclusion $\cU_{\bm{i}_2}\hookrightarrow \cU_{\bm{i}_1}$ for $\bm{i}_1, \bm{i}_2\in C(I)_0$ satisfying $\bm{i}_1\subset \bm{i}_2$. Then there exists a morphism of $\infty$-categories $\Indcoh \cU_\bullet \colon C(I)\rightarrow \Mod_{\Mod(H\bC)}(\cP r^L_{st,\omega})$ sending $\bm{i}$ to $\Indcoh \cU_{\bm{i}}$ and $\bm{i}_1\subset \bm{i}_2$ to $\Ind\iota_{\bm{i}_1,\bm{i}_2}^*$. We write $\iota_{\bm{i}_1,\bm{i}_2}^*$ for $\Ind\iota_{\bm{i}_1,\bm{i}_2}^*$ for simplicity.

The following proposition is proved in \cite{GaitsgoryIndCoh}, but we present a detailed proof for our later discussion.
\begin{proposition}[{\cite[Proposition 4.2.1]{GaitsgoryIndCoh}}]\label{gluingcohgeneral}
There exists an equivalence of $\infty$-categories
\begin{equation}
\Indcoh \cX\simeq \lim{\substack{\longleftarrow \\C(I)}}\Indcoh\cU_\bullet.
\end{equation}
\end{proposition}

We use the Morita model structure of the dg-category of dg-categories to prove Proposition \ref{gluingcohgeneral}.
\begin{proof}
In this proof, we view $C(I)$ as a category. Let $i$ be an element of $I$. The morphisms $\iota_{\bm{i},\bm{i}\cup\{i\}}$ for $\bm{i}\in C(I\bs\{i\})_0\subset C(I)_0$ together give a morphism
\begin{equation}\label{7.2}
\holim{\substack{\longleftarrow \\C(I\bs\{i\})}}\Indcoh \cU_\bullet\rightarrow \holim{\substack{\longleftarrow \\C(I\bs \{i\})}}\Indcoh \cU_{\bullet\cup\{i\}}.
\end{equation}
The morphisms $\iota_{\{i\},\bm{i}\cup \{i\}}$ also give a morphism
\begin{equation}\label{7.3}
\Indcoh \cU_i \rightarrow \holim{\substack{\longleftarrow \\C(I\bs \{i\})}}\Indcoh \cU_{\bullet\cup\{i\}}.
\end{equation}
By the universality of homotopy limits, the homotopy pullback of (\ref{7.2}) and (\ref{7.3}) satisfies
\begin{equation}
\holim{\substack{\longleftarrow \\C(I)}}\Indcoh \cU_\bullet\simeq \Indcoh \cU_{i} \tim{\holim{\substack{\longleftarrow \\C(I\bs\{i\})}}\Indcoh \cU_{\bullet\cup\{i\}}}{h}\holim{\substack{\longleftarrow \\C(I\bs\{i\})}}\Indcoh \cU_\bullet.
\end{equation}
Hence by induction on the cardinality of $I$, it suffices to show the case $|I|=2$. Namely, for a $G$-invariant covering $\{U, V\}$ of $X$, we have to show
\begin{equation}
\Indcoh[X/G]\simeq \Indcoh[U/G]\tim{\Indcoh [(U\cap V)/G]}{h}\Indcoh [V/G].
\end{equation}
Note that $\Indcoh [X/G]$ is equivalent to the dg-category $\Indcoh_G X$ of $G$-equivariant ind-coherent sheaves over $X$.

Denote the open inclusions by
\begin{equation}\label{7.6}
\xymatrix{
X &\ar[l]_-{i_U} U \\
V  \ar[u]^{i_V}  & \ar[l]^-{j_V} U\cap V. \ar[u]_{j_U}\ar[ul]_{k}
}
\end{equation}
By using the functorial factorization of the Morita model structure, we have a factorization
\begin{equation}
\Indcoh_G U\xrightarrow{\gamma(j_{U}^*)} \cA \xrightarrow{\delta(j_U^*)} \Indcoh_G (U\cap V)
\end{equation} 
where $\gamma$ and $\delta$ are the functorial factorizations which produce a trivial cofibration and a fibration respectively. Since the dg-categories of $G$-equivariant ind-coherent sheaves are idempotent-complete pretriangulated dg-categories, they are fibrant objects in the Morita model structure by Proposition \ref{propertyMorita}. We have
\begin{equation}
\Indcoh_G U\tim{\Indcoh_G(U\cap V)}{h}\Indcoh_G V\simeq \cA\tim{\Indcoh_G (U\cap V)}{} \Indcoh_G V
\end{equation}
where the right hand side is the strict pullback. We have a functor $l\colon \cA\rightarrow \Indcoh_G U$ which fits into the commutative diagram
\begin{equation}
\xymatrix{
\Indcoh_G U \ar[r]_-{\id} \ar[d]_{\gamma(j_U^*)}   & \Indcoh_G U \ar[d]\ar[r]_{i^{\Ind}_{U*}}& \Indcoh_G X \\
\cA  \ar[r]\ar@{-->}[ru]_{l}& \ast &
}
\end{equation}
since $\gamma(j_U^*)$ is a trivial cofibration and $\Indcoh_G U$ is a fibrant object.

We have two canonical morphisms coming from adjunctions:
\begin{equation}
\begin{split}
l&\rightarrow j^{\Ind}_{U*}j_U^*l \text{ and }\\
\id &\rightarrow j^{\Ind}_{V*}j_V^*.
\end{split}
\end{equation}
By applying $i^{\Indcoh}_{U*}$ to the first line and $i^{\Ind}_{V*}$ to the second line, we have
\begin{equation}
\begin{split}
i^{\Ind}_{U*}l&\rightarrow k^{\Ind}_*\delta(j_U^*), \\
i^{\Ind}_{V*}&\rightarrow k^{\Ind}_*j_{V}^*. \\
\end{split}
\end{equation}
An object in $\cA\tim{\Indcoh_G (U\cap V)}{} \Indcoh_G V$ is an object $(a,\cE)$ of $\cA\times \Indcoh_G V$ satisfying $\delta(j_U^*)(a)=j_V^*\cE$. Hence we have a diagram
\begin{equation}
i^\Ind_{U *}l(a)\rightarrow k^\Ind_*\delta(j_U^*)(a)=k^\Ind_*j_V^*(\cE)\leftarrow i^\Ind_{V*}(\cE)
\end{equation}
in $\Indcoh_G X$. Therefore, there exists a functor $\phi\colon \cA\tim{\Indcoh_G (U\cap V)}{} \Indcoh_G V\rightarrow \Indcoh_G X$, which is defined on objects as
\begin{equation}
\phi(a,\cE):=i^\Ind_{U*}l(a)\tim{k^\Ind_*j_V^*(\cE)}{h}i^\Ind_{V*}(\cE).
\end{equation}
We show that the functor $\psi\colon \Indcoh_G X\rightarrow \cA\tim{\Indcoh_G (U \cap V)}{} \Indcoh_G V$ given by 
\begin{equation}
\psi(\cE):=(\gamma(j_U^*)i_{U}^*(\cE), i_{V}^*(\cE))
\end{equation}
on objects is the quasi-inverse of $\phi$.
For $\cE\in \Indcoh_G X$, we have
\begin{equation}
\phi\circ \psi(\cE)=i^\Ind_{U*}l\gamma(j_U^*)i_{U}^*(\cE)\tim{k^\Ind_*j_V^* i_{V}^*(\cE)}{h}i^\Ind_{V*} i_{V}^*(\cE)=i^\Ind_{U*}i_{U}^*(\cE)\tim{k^\Ind_*k^*(\cE)}{h}i^\Ind_{V*} i_{V}^*(\cE),
\end{equation}
We write $\indlim \cE_i$ for $\cE$, then we have
\begin{equation}
i^\Ind_{U*}i_{U}^*(\cE)\tim{k^\Ind_*k^*(\cE)}{h}i^\Ind_{V*} i_{V}^*(\cE)\simeq \indlim i_{U*}i_{U}^*(\cE_i)\tim{k_*k^*(\cE_i)}{h}i_{V*} i_{V}^*(\cE_i),
\end{equation}
where we used Proposition \ref{propertyofphi} and Proposition \ref{pushind} in order to deduce $i^{\Ind}_{\bullet*}i_{}^*\cE_i\simeq i_{\bullet*}i^*\cE_i$ for $\bullet=U, V$ and $k_*^{\Ind} k_{}^*\cE_i\simeq k_* k^*\cE_i$

Moreover, we have the right hand side fits into an exact triangle in $\Coh_G X$
\begin{equation}
i_{U*}i_{U}^*(\cE_i)\tim{k_*k^*(\cE)}{h}i_{V*} i_{V}^*(\cE_i)\rightarrow i_{U*}i_{U}^*(\cE_i)\oplus i_{V*} i_{V}^*(\cE_i)\rightarrow k_*k^*(\cE_i)\rightarrow.
\end{equation}
This is nothing but a Mayer-Vietoris sequence, hence $\cE_i\simeq i_{U*}i_{U}^*(\cE_i)\tim{k_*k^*(\cE_i)}{h}i_{V*} i_{V}^*(\cE_i)$. This gives the essential surjectivity of $H^0(\phi)$.

Set $\cC:=\cA\tim{\Indcoh_G(U\cap V)}{} \Indcoh_G V$ for short.
For $(a_i,\cE_i)\in \cC$ for $i=1,2$, we have
\begin{equation}
\hom_\cC((a_1,\cE_1), (a_2,\cE_2))\simeq \hom_\cA(a_1, a_2)\tim{\hom_{\Indcoh_G (U\cap V)}(j_V^*\cE_1, j_V^*\cE_2)}{}\hom_{\Indcoh_G V}(\cE_1,\cE_2).
\end{equation}
by the definition of strict pullback of dg-categories, there the pullback of the left hand side is taken in $\Mod(\bC)$. Note the fact the Morita model structure is a left Bousfield localization of the quasi-equivalent model structure. Hence fibrations of the Morita model structure are fibrations of the quasi-equivalent model structure. Since fibrations of the quasi-equivalent model structure are in particular surjective, they are fibrations in the projective model structure of the dg-category of dg-vector spaces $\Mod(\bC)$. Moreover, all objects are fibrant objects in the projective model structure. Hence we have
\begin{equation}
\begin{split}
&\hom_\cA(a_1, a_2)\tim{\hom_{\Indcoh_G (U\cap V)}(j_V^*\cE_1, j_V^*\cE_2)}{}\hom_{\Indcoh_G V}(\cE_1,\cE_2)\\
&\quad \simeq \hom_\cA(a_1, a_2)\tim{\hom_{\Indcoh_G (U\cap V)}(j_V^*\cE_1, j_V^*\cE_2)}{h}\hom_{\Indcoh_G V}(\cE_1,\cE_2).
\end{split}
\end{equation}
Since $l$ is a quasi-equivalence, we further have
\begin{equation}\label{7.19}
\begin{split}
&\hom_\cA(a_1, a_2)\tim{\hom_{\Indcoh_G (U\cap V)}(j_V^*\cE_1, j_V^*\cE_2)}{h}\hom_{\Indcoh_G V}(\cE_1,\cE_2)\\
&\quad \simeq \hom_{\Indcoh_G U}(l(a_1), l(a_2))\tim{\hom_{\Indcoh_G (U\cap V)}(j_V^*\cE_1, j_V^*\cE_2)}{h}\hom_{\Indcoh_G V}(\cE_1,\cE_2).
\end{split}
\end{equation}

On the other hand, we have 
\begin{equation}
\hom\lb i^\Ind_{U*}l(a_1)\tim{k^\Ind_*j_V^*(\cE_1)}{h}i^\Ind_{V*}(\cE_1),i^\Ind_{U*}l(a_2)\tim{k^\Ind_*j_V^*(\cE_2)}{h}i^\Ind_{V*}(\cE_2)\rb \simeq X_1\tim{X_3}{h} X_2
\end{equation}
where
\begin{equation}
\begin{split}
X_1&:=\hom\lb i^\Ind_{U*}l(a_1)\tim{k^\Ind_*j_V^*(\cE_1)}{h}i^\Ind_{V*}(\cE_1),i^\Ind_{U*}l(a_2)\rb, \\
X_2&:=\hom \lb i^\Ind_{U*}l(a_1)\tim{k^\Ind_*j_V^*(\cE_1)}{h}i^\Ind_{V*}(\cE_1), i^\Ind_{V*}(\cE_2)\rb, \text{ and}\\
X_3&:=\hom \lb i^\Ind_{U*}l(a_1)\tim{k^\Ind_*j_V^*(\cE_1)}{h}i^\Ind_{V*}(\cE_1), k^\Ind_*j_V^*(\cE_2)\rb.
\end{split}
\end{equation}
Since 
\begin{align}
i_U^*\lb i^\Ind_{U*}l(a_1)\tim{k^\Ind_*j_V^*(\cE)}{h}i^\Ind_{V*}(\cE_1) \rb&\simeq l(a_1),\\
i_V^*\lb i^\Ind_{U*}l(a_1)\tim{k^\Ind_*j_V^*(\cE)}{h}i^\Ind_{V*}(\cE_1)\rb &\simeq \cE_1, \text{ and}\\
k^*\lb  i^\Ind_{U*}l(a_1)\tim{k^\Ind_*j_V^*(\cE)}{h}i^\Ind_{V*}(\cE_1)\rb&\simeq j_V^*\cE_1,
\end{align}
we have a quasi-isomorphism
\begin{equation}\label{7.26}
X_1\tim{X_3}{h} X_2\simeq \hom(l(a_1),l(a_2))\tim{\hom(j_V^*\cE_1, j_V^*\cE_2)}{h}\hom (\cE_1,\cE_2).
\end{equation}
The equations (\ref{7.19}) and (\ref{7.26}) show that $H^0(\phi)$ is fully faithful.
\end{proof}

Let $(\hSigma,\beta)$ be a stacky fan satisfying Condition \ref{condition}. Let $\hSigma_{\mathrm{max}}$ be the set of maximal cones in $\Sigma$. The set $\hSigma_{\mathrm{max}}$ gives the index set of the open $G_\beta$-invariant covering $\{U_{\hsigma}\}_{\sigma\in \hSigma_{\mathrm{max}}}$ of $X_{\hSigma}$. We set $C(\Sigma):=C(\hSigma_\mathrm{max})$. Then we have the following corollary of Proposition \ref{gluingcohgeneral}.

\begin{corollary}\label{gluingcoh}
We have an equivalence of $\infty$-categories
\begin{equation}
\Indcoh \cXsb\simeq \lim{\substack{\longleftarrow \\C(\Sigma)}}\Indcoh\cU_\bullet.
\end{equation}
\end{corollary}

\section{Gluing the functors $K_{\hsigma,\beta}$}\label{functor}
Let $(\hSigma,\beta)$ be a stacky fan satisfying Condition \ref{condition}. In the present section, we further assume the following Condition \ref{condition2} below.
For $\ssigma=\{\sigma_1,...,\sigma_s\}\in C(\Sigma)_0$, we set $|\ssigma|:=\bigcap_{i=1}^s\sigma_i\in \Sigma$ and
\begin{equation}
\Lambda_\ssigma:=\Lambda_{\hSigma(\widehat{|\ssigma|}),\beta}.
\end{equation}
Evidently $\bigcup_{\ssigma}\Lambda_\ssigma=\Lambda_{\hSigma,\beta}$ and $\Lambda_\sigma\cap \Lambda_\tau=\Lambda_{\sigma\cap \tau}$ for $\sigma, \tau\in \Sigma$.

\begin{condition}\label{condition2}
For any $\hsigma\in\hSigma$, there exists an equivalence
\begin{equation}
K_{\hsigma,\beta}\colon \Coh \cX_{\hSs, \beta}\xrightarrow{\simeq} \wSh_{\hSs,\beta}(T^n)
\end{equation}
which satisfies the following:
\begin{enumerate}
\item Set the image of $i_{\sigma_1\sigma_2}^*$ under $K_{\hsigma,\beta}$ as
\begin{equation}
I^{\sigma_1\sigma_2}:=\Ind K_{\hsigma_1,\beta}\circ i^*_{\sigma_1\sigma_2}\circ \Ind K_{\hsigma_2,\beta}^{-1}\colon \lSh_{\Lambda_{\sigma_2}}(T^n)\rightarrow \lSh_{\Lambda_{\sigma_1}}(T^n) 
\end{equation}
for  faces $\sigma_2\subset \sigma_1\subset \sigma$. Then there exists a natural isomorphism
\begin{equation}
I^{\sigma_1\sigma_2}(-)\simeq (-)\star \Theta(\sigma_1,0).
\end{equation}
\item The restriction of $K_{\hsigma,\beta}$ to $\perf \cX_{\hSs,\beta}\otimes \omega_{\cX_{\hSs,\beta}}$ is $\kappa_{\hsigma,\beta}\circ D\circ \frakD$ for any $\sigma\in \Sigma$ where $D:=\cHom(-, \cO_{\cX_{\hSs,\beta}})$ and $\frakD:=\cHom(-, \omega_{\cX_{\hSs,\beta}})$.
\end{enumerate}
\end{condition}

We assume Condition \ref{condition2}. We set $I^{\ssigma_1\ssigma_2}:=I^{|\ssigma_1||\ssigma_2|}$. As in Section \ref{section:gluecoh}, let $\lSh_{\Lambda_\bullet}\colon C(\Sigma)\rightarrow  \Mod_{\Mod(H\bC)}(\cP r^L_{st,\omega})$ be a morphism of $\infty$-categories sending $\ssigma\mapsto \lSh_{\Lambda_\ssigma}(T^n)$ and $(\ssigma_1\subset \ssigma_2)\mapsto I^{\ssigma_1\ssigma_2}\colon \lSh_{\Lambda_{\ssigma_1}}(T^n)\rightarrow \lSh_{\Lambda_{\ssigma_2}}(T^n)$.

The adjunction $\id\rightarrow i^\Ind_{\sigma_1\sigma_2*} \circ i_{\sigma_1\sigma_2}^*$ induces a natural transformation
\begin{equation}
A(\sigma_1\sigma_2)\colon \id\rightarrow \Ind K_{\hsigma_1,\beta}\circ i^{\Ind}_{\sigma_1\sigma_2*}\circ i^*_{\sigma_1\sigma_2}\circ \Ind K_{\hsigma_1,\beta}^{-1}.
\end{equation}

Then by using Proposition \ref{gluingcoh}, we can define
\begin{equation}\label{defofkappa}
K_{\hSigma,\beta}\colon 
\Indcoh \cXsb\xrightarrow{\simeq} \lim{\substack{\longleftarrow \\C(\Sigma)}}\Indcoh \cU_\bullet\xrightarrow[\lim{\substack{\longleftarrow \\C(\Sigma)}}\Ind K_{\hsigma,\beta}]{\simeq} \lim{\substack{\longleftarrow \\C(\Sigma)}} \lSh_{\Lambda_{\hSigma(\bullet),\beta}}(T^n)\rightarrow \lSh_{\Lsb}(T^n)
\end{equation}
where the rightmost arrow is given by the gluing of the natural inclusions $\lSh_{\Lambda_{\hSigma,\beta}}(T^n)\hookrightarrow \lSh_{\Lambda_{\hSigma,\beta}}(T^n)$ and the natural transformations $A(\sigma_i\sigma_j)$ in the same manner as the $\mathrm{\check{C}ech}$ resolution on the coherent side (see Section \ref{section:main} for details).

Let $i_{\sigma}\colon \cU_\sigma\rightarrow \cXsb$ be the open inclusion for $\sigma\in \Sigma$.

\begin{lemma}\label{fullyfaithful}
The functor $K_{\hSigma,\beta}$ is fully faithful.
\end{lemma}
\begin{proof}
For $\cE,\cF\in\Coh \cXsb$, by the $\mathrm{\check{C}ech}$ resolution, we have
\begin{equation}\label{eq11}
\hom_{\Coh \cXsb}(\cE,\cF)\simeq \holim{\substack{\longrightarrow \\ \ssigma\in C(\Sigma)}} \holim{\substack{\longleftarrow \\ \ssigma'\in C(\Sigma)}} \hom_{\Ind\Coh \cXsb}(\cE|_{\cU_\ssigma},\cF|_{\cU_{\ssigma'}}).
\end{equation}
By the adjunction, we have 
\begin{equation}
\holim{\substack{\longrightarrow \\ \ssigma\in C(\Sigma)}} \holim{\substack{\longleftarrow \\ \ssigma'\in C(\Sigma)}} \hom_{\Ind\Coh \cXsb}(\cE|_{\cU_\ssigma},\cF|_{\cU_{\ssigma'}}) \simeq \holim{\substack{\longrightarrow \\ \ssigma\in C(\Sigma)}} \holim{\substack{\longleftarrow \\ \ssigma'\in C(\Sigma)}}\hom_{\Ind\Coh \cU_{\ssigma'}}(\cE|_{\cU_{\ssigma\cap \ssigma'}},\cF|_{\cU_{\ssigma'}}).
\end{equation}
Since $K_{\hsigma,\beta}$ is an equivalence by Condition \ref{condition2}, we have
\begin{equation}
\begin{split}
&\holim{\substack{\longrightarrow \\ \ssigma\in C(\Sigma)}} \holim{\substack{\longleftarrow \\ \ssigma'\in C(\Sigma)}}\hom_{\Ind\Coh \cU_{\ssigma'}}(\cE|_{\cU_{\ssigma\cap \ssigma'}},\cF|_{\cU_{\ssigma'}})\\
&\quad \simeq \holim{\substack{\longrightarrow \\ \ssigma\in C(\Sigma)}} \holim{\substack{\longleftarrow \\ \ssigma'\in C(\Sigma)}}\hom_{\lSh_{\Lambda_{\hSigma(\hat{|\ssigma'|}),\beta}}(T^n)}(K_{\hSigma(\hat{|\ssigma'|}),\beta}(\cE|_{\cU_{\ssigma\cap \ssigma'}}),K_{\hSigma(\hat{|\ssigma'|}),\beta}(\cF|_{\cU_{\ssigma'}})).
\end{split}
\end{equation}
By Condition \ref{condition2} (i), we further have
\begin{equation}
\begin{split}
&\holim{\substack{\longrightarrow \\ \ssigma\in C(\Sigma)}} \holim{\substack{\longleftarrow \\ \ssigma'\in C(\Sigma)}}\hom_{\lSh_{\Lambda_{\hSigma(\hat{|\ssigma'|}),\beta}}(T^n)}(K_{\hSigma(\hat{|\ssigma'|}),\beta}(\cE|_{\cU_{\ssigma\cap \ssigma'}}),K_{\hSigma(\hat{|\ssigma'|}),\beta}(\cF|_{\cU_{\ssigma'}}))\\
&\quad\simeq \holim{\substack{\longrightarrow \\ \ssigma\in C(\Sigma)}} \holim{\substack{\longleftarrow \\ \ssigma'\in C(\Sigma)}}\hom_{\lSh_{\Lambda_{\hSigma(\hat{|\ssigma'|}),\beta}}(T^n)}(K_{\hSigma(\hat{|\ssigma|}),\beta}(\cE|_{\cU_{\ssigma}})\star \Theta(\sigma',0),K_{\hSigma(\hat{|\ssigma'|}),\beta}(\cF|_{\cU_{\ssigma'}})).
\end{split}
\end{equation}
The adjunction in Lemma \ref{staradjunction} implies
\begin{equation}
\begin{split}
&\holim{\substack{\longrightarrow \\ \ssigma\in C(\Sigma)}} \holim{\substack{\longleftarrow \\ \ssigma'\in C(\Sigma)}}\hom_{\lSh_{\Lambda_{\hSigma(\hat{|\ssigma'|}),\beta}}(T^n)}(K_{\hSigma(\hat{|\ssigma|}),\beta}(\cE|_{\cU_{\ssigma}})\star \Theta(\sigma',0),K_{\hSigma(\hat{|\ssigma'|}),\beta}(\cF|_{\cU_{\ssigma'}}))\\
&\quad\simeq \holim{\substack{\longrightarrow \\ \ssigma\in C(\Sigma)}} \holim{\substack{\longleftarrow \\ \ssigma'\in C(\Sigma)}}\hom_{\lSh(T^n)}(K_{\hSigma(\hat{|\ssigma|}),\beta}(\cE|_{\cU_{\ssigma}}),\cHom^\star(\Theta(\sigma',0),K_{\hSigma(\hat{|\ssigma'|}),\beta}(\cF|_{\cU_{\ssigma'}}))).
\end{split}
\end{equation}
Lemma \ref{lemmaeq} below shows
\begin{equation}
\begin{split}
&\holim{\substack{\longrightarrow \\ \ssigma\in C(\Sigma)}} \holim{\substack{\longleftarrow \\ \ssigma'\in C(\Sigma)}}\hom_{\lSh(T^n)}(K_{\hSigma(\hat{|\ssigma|}),\beta}(\cE|_{\cU_{\ssigma}}),\cHom^\star(\Theta(\sigma',0),K_{\hSigma(\hat{|\ssigma'|}),\beta}(\cF|_{\cU_{\ssigma'}})))\\
&\quad\simeq \holim{\substack{\longrightarrow \\ \ssigma\in C(\Sigma)}} \holim{\substack{\longleftarrow \\ \ssigma'\in C(\Sigma)}}\hom_{\lSh(T^n)}(K_{\hSigma(\hat{|\ssigma|}),\beta}(\cE|_{\cU_{\ssigma}}),K_{\hSigma(\hat{|\ssigma'|}),\beta}(\cF|_{\cU_{\ssigma'}})).
\end{split}
\end{equation}
Finally, by the definition of $K_{\hSigma,\beta}$, we have
\begin{equation}
\holim{\substack{\longrightarrow \\ \ssigma\in C(\Sigma)}} \holim{\substack{\longleftarrow \\ \ssigma'\in C(\Sigma)}}\hom_{\lSh(T^n)}(K_{\hSigma(\hat{|\ssigma|}),\beta}(\cE|_{\cU_{\ssigma}}),K_{\hSigma(\hat{|\ssigma'|}),\beta}(\cF|_{\cU_{\ssigma'}}))
\simeq \hom_{\lSh(T^n)}(K_{\hSigma,\beta}(\cE),K_{\hSigma,\beta}(\cF)).
\end{equation}
This completes the proof.
\end{proof}

We set another functor
\begin{equation}\label{khsb}
\kappa_{\hSigma,\beta}:=\varhocolim{} \Ind (K_{\hSigma_\beta}\circ ((-)\otimes \omega_{\cXsb}))\colon \Qcoh \cX_{\hSigma,\beta}\rightarrow \lSh_{\Lambda_{\hSigma,\beta}}(T^n)
\end{equation}
where $\frakD:=\cHom(-,\omega_{\cX_{\hSigma,\beta}})$, which means the composition of the functors
\begin{align}
\Ind (K_{\hSigma_\beta}\circ ((-)\otimes \omega_{\cXsb}))&\colon  \Qcoh \cX_{\hSigma,\beta}\simeq \Ind(\perf(\cX_{\hSigma,\beta}))\rightarrow \Ind(\lSh_{\Lambda_{\hSigma,\beta}}(T^n))\\
\text{and  }\varhocolim{} &\colon \Ind(\lSh_{\Lambda_{\hSigma,\beta}}(T^n)) \rightarrow \lSh_{\Lambda_{\hSigma,\beta}}(T^n)
\end{align}
where the latter functor is just taking colimits.

\begin{remark}
By Lemma \ref{affine} and Condition \ref{condition2}~(ii), the functor $\kappa_{\hSigma,\beta}$ induces the functor of Fang--Liu--Treumann--Zaslow's functor for toric varieties and orbifolds \cite{FLTZ, FLTZ2, Tr, SS}.
\end{remark}

Let $\beta_i\colon L_i\rightarrow N_i$ for $i=1,2$ be homomorphisms of free abelian groups and $(\hSigma_i,\beta_i)$ be stacky fans satisfying Condition \ref{condition}. Let $f_L\colon L_1\rightarrow L_2$ and $f_N\colon N_1\rightarrow N_2$ be injections which are compatible with $\beta_i$'s and $f_{L}(\hsigma)\in \subset \hsigma_2$ for some $\hsigma_2\in \hSigma_2$ (resp. $f_{N}(\sigma)\subset \sigma_2$ for some $\sigma_2\in \hSigma_2$) for any $\hsigma\in \hSigma_1$ (resp. $\sigma\in \Sigma_1$). We assume that the inverse image of each cone in $\hSigma_2$ (resp. $\Sigma_2$) under $f_L$ (resp. $f_N$) is written as a union of cones in $\hSigma_1$ (resp. $\Sigma_1$). Then there exists a morphism $f\colon \cX_{\hSigma_1,\beta_1}\rightarrow \cX_{\hSigma_2,\beta_2}$ induced by $f_L$ and $f_N$. Let $[f_N^\vee]$ be the induced morphism $(M_2)_\bR/M_2\rightarrow (M_1)_\bR/M_1$ where $M_i:=\Hom_\bZ(N_i,\bZ)$ for $i=1,2$.
\begin{proposition}[\cite{FLTZ, Tr, SS}]\label{prop:functoriality}
There exists a commutative diagram.
\begin{equation}
\xymatrix{
\Qcoh \cX_{\hSigma_2,\beta_2} \ar[d]_{f^*}  \ar[r]^-{\kappa_{\hSigma_2,\beta_2}} &\lSh_{\Lambda_{\hSigma_2,\beta_2}}((M_2)_\bR/M_2) \ar[d]^{[f_N^\vee]_!} \\
\Qcoh \cX_{\Lambda_{\hSigma_1,\beta_1}}  \ar[r]_-{\kappa_{\hSigma_1,\beta_1}} &\lSh_{\Lambda_{\hSigma_1,\beta_1}}((M_2)_\bR/M_2). 
}
\end{equation}
\end{proposition}
\begin{proof}
By the definition of $\ksb$, we only have to prove the case $\hSigma=\hSigma(\hsigma)$, i.e., the fan consisting of faces of a single cone $\hsigma$. Since the category $\Qcoh \cX_{\hSigma(\hsigma),\beta}$ is generated by $\Theta'(\sigma)$ by Proposition \ref{generation1}, it suffices to show $\kappa_{\hSigma_1,\beta_1}\circ f^*\Theta'(\sigma,\chi)\simeq [f^\vee]_!\Theta(\sigma, \chi)$ for any $\chi\in M_{\sigma,\beta}$. This can be proved in the same manner as in \cite[Theorem 3.8]{FLTZ} and hence omitted. 
\end{proof}

 We follow the notation in Section \ref{technique}.
\begin{proposition}[\cite{FLTZ, Tr, SS}]\label{monoidal}Suppose that $\hSigma$ is smooth. For $\cE_1,\cE_2\in \Qcoh \cX_{\hSigma,\beta}$, we have
\begin{equation}
\ksb(\cE_1\otimes \cE_2)\simeq \ksb(\cE_1)\star \ksb(\cE_2).
\end{equation}
\end{proposition}
\begin{proof}
This proposition follows from Proposition \ref{prop:functoriality} and the proof of \cite[Corollary 3.13]{FLTZ}. 
\end{proof}

Hereafter in this section, we will assume $\hSigma$ is smooth.
Proposition \ref{monoidal} has the following important corollary.
\begin{corollary}\label{idfunctor}Assume $\hsigma$ is smooth. The functor $\kappa_{\hsigma,\beta}$ satisfies Condition \ref{condition2}.
\end{corollary}
\begin{proof}
Condition \ref{condition2} (ii) follows by the definition.

Proposition \ref{monoidal} implies that $\kappa_{\hsigma,\beta}$ is a monoidal equivalence. Hence
\begin{equation}
\begin{split}
\Theta(\sigma_1,\chi)\star\Theta(\sigma_2,0) &\simeq \kappa_{\hsigma_1,\beta}(\Theta'(\sigma_1,\chi)\otimes \Theta'(\sigma_2,0))\\
&\simeq  \kappa_{\hsigma_1,\beta}(\Theta'(\sigma_1,\chi))\\
&\simeq \Theta(\sigma_1,\chi)
\end{split}
\end{equation}
for any $\chi\in M_{\sigma_1}$. It also holds that $I^{\sigma_1\sigma_2}(\Theta(\sigma_2,\chi))\simeq  \Theta(\sigma_1,\chi)$. Since $\bigoplus_{\chi\in M_{\beta}/M}\Theta(\sigma_2, \chi)$ generates the whole category $\Qcoh\cX_{\hSigma(\hsigma_2),\beta}$, Condition \ref{condition2} (i) follows.
\end{proof}

\section{Unit object for $\star$-product}\label{section:identity}
Let $(\hSigma,\beta)$ be a stacky fan satisfying Condition \ref{condition} and Condition \ref{condition2}. Let $[0]\in T^n$ be the unit element of $T^n=M_\bR/M$.

\begin{lemma}\label{lem:unit}
If $\Sigma$ is complete. we have
\begin{equation}
\ksb(\cO_\cXsb)\simeq \bC_{[0]}.
\end{equation}
\end{lemma}
\begin{proof}
Let $\Sigma(i)$ be the set consisting of $i$-dimensional cones in $\Sigma$ for $i=1,..., n$. 
We have an exact sequence of $\cO_{\cXsb}$ as
\begin{equation}\label{cechcomp}
0\rightarrow \cO_{\cX_{\hSigma_\sigma,\beta}}\rightarrow \bigoplus_{\sigma\in\Sigma(n)}\Theta'(\sigma,0)\rightarrow \bigoplus_{\sigma\in \Sigma(n-1)} \Theta'(\sigma,0)\rightarrow\cdots\rightarrow \Theta'(\{0\},0)\rightarrow 0.
\end{equation}
where the diffrentials are sums of appropriately signed restriction maps.
Hence $\cO_{\cXsb}$ is quasi-isomorphic to the complex $\bigoplus_{i=1}^n\bigoplus_{\sigma\in \Sigma(i)} \Theta'(\sigma,0)[i-n]$ with the differentials induced by restriction maps. 
By the definition of $\ksb$ and Proposition \ref{affine}, we have $\ksb(\Theta'(\sigma))\simeq p_!\bC_{\Int(\sigma^\vee)}[n]$. By Lemma \ref{homtheta}, the images of the restriction maps are induced by inclusion maps $\sigma^\vee\hookrightarrow \tau^\vee$ for $\tau\subset \sigma$. Hence the sheaf $\ksb(\cO_\cXsb)$ is quasi-isomorphic to $\bigoplus_{i=1}^n\bigoplus_{\sigma\in \Sigma(i)} \Theta(\sigma,0)[i]$ with the differentials induced by inclusion maps. Since the sheaf $\bigoplus_{i=1}^{n-1}\bigoplus_{\sigma\in \Sigma(i)} \bC_{\Int(\sigma^\vee)}[i]$ is precisely the $\mathrm{\check{C}ech}$ resolution of $\bC_{M_\bR\bs\{0\}}$ by the open covering $\{\Int(\rho^\vee)\}_{\rho\in \Sigma(1)}$, we have 
\begin{equation}
\ksb(\cO_\cXsb)\simeq p_!\Cone\left(\bigoplus_{\rho\in \Sigma(1)}\bC_{\Int(\rho^\vee)}\rightarrow \bC_{M_\bR}\right)\simeq p_!\bC_{0}=\bC_{[0]}.
\end{equation}
This completes the proof.
\end{proof}

\begin{lemma}\label{lemmaeq}There exists quasi-isomorphisms
\begin{align} 
&E\star \kappa_{\hSigma,\beta}(\cO_{\cX_{\hSigma,\beta}})\simeq E \text{ and} \label{starunit}\\
&\cHom^\star(\kappa_{\hSigma,\beta}(\cO_{\cX_{\hSigma,\beta}}),E)\simeq E \label{adjunit}
\end{align}
for $E\in\lSh_{\Lambda_{\hSigma,\beta}}(T^n)$.
\end{lemma}
\begin{proof}
Since $\cO_{\hSigma,\beta}$ is a perfect complex, the first equalities of (\ref{starunit}) and (\ref{adjunit}) follow from Condition \ref{condition2} (ii).

The quasi-isomorphism (\ref{adjunit}) follows from (\ref{starunit}), the fact $\cHom^\star$ is the left adjoint of $\star$, and the Yoneda lemma for $\lSh_{\Lambda_{\hSigma,\beta}}(T^n)$.

If $\Sigma$ is complete, then $\kappa_{\hSigma,\beta}(\cO_{\cX_{\hSigma,\beta}})$ is quasi-isomorphic to $\bC_{[0]}$ by Lemma \ref{lem:unit}. Since the sheaf $\bC_{[0]}$ is the monoidal unit of $\star$, (\ref{starunit}) holds for complete fans.

For non-complete fans, we will prove by induction. Let $(\hSigma', \beta)$ be another stacky fan satisfying Condition \ref{condition} and assume that there exists a maximal cone $\sigma\in \Sigma'$ such that $\Sigma=\Sigma'\bs\{\sigma\}$. As the induction hypothesis, we assume that (\ref{starunit}) holds for $(\hSigma',\beta)$.

We set $d:=\dim\sigma$. Let $\{\rho^1,..., \rho^r\}$ be the set of 1-dimensional faces of $\sigma$ and $n^i$ be the primitive generator of $\rho^i$ for $i=1,...,r$. Set $\rho_c:=\bR\cdot\left(\sum_{i=1}^rn^i\right)$. Let $\{f^1,...,f^r\}$ be the set of facets ($(d-1)$-dimensional face) of $\sigma$. We define $\sigma_i$ to be the convex hull of $f_i$ and $\rho_c$ for any $i$. Then the union of the set of faces of all $\sigma_i$ for $i=1,...,r$ and $\Sigma$ is a fan refining $\Sigma'$. We write  $\Sigma^\sigma$ for this fan.

Since $\beta_\bR|_{\hsigma}\colon \hsigma\rightarrow \sigma$ is bijective by Lemma \ref{combinatorial}, the inverse image $\hat{\rho}_c$ of $\rho_c$ under $\beta_\bR|_{\hsigma}$ is again a 1-dimensional cone in $\hsigma$. We define a fan $\hSigma^\sigma$ refining $\hSigma'$ in the same manner as the definition of $\Sigma^\sigma$. Then $(\hSigma^\hsigma,\beta)$ gives a stacky fan satisfying Condition \ref{condition}. Let $\pi\colon \cX_{{\hSigma}^\sigma,\beta}\rightarrow \cX_{\hSigma',\beta}$ be the morphism induced by this refinement. Then Proposition \ref{prop:functoriality} and the induction hypothesis on $\Sigma'$ imply that 
\begin{equation}\label{blowup}
E\star \kappa_{\hSigma^\sigma,\beta}(\cO_{\cX_{\hSigma^\sigma,\beta}})\simeq E\star\kappa_{\hSigma^\sigma,\beta}(\pi^*\cO_{\cX_{\hSigma',\beta}}) \simeq E\star \kappa_{\hSigma',\beta}(\cO_{\cX_{\hSigma',\beta}})\simeq E
\end{equation}
for $E\in \lSh_{\Lambda_{\hSigma,\beta}}(T^n)$.

Let $\Sigma_\sigma$ be the subfan of ${\Sigma}^\sigma$ consisting of $\sigma_1,...,\sigma_r$ and their faces. Then there exists an exact triangle
\begin{equation}\label{exact9.5}
\cO_{\cX_{\hSigma^\sigma,\beta}}\rightarrow \cO_{\cX_{\hSigma_\sigma,\beta}}\oplus \cO_{\cX_{\hSigma,\beta}}\rightarrow \cO_{\cX_{\hSigma'',\beta}}\rightarrow 
\end{equation}
where $\S'':=\Sigma\cap\Sigma_\sigma$. The exact triangle (\ref{exact9.5}) and (\ref{blowup}) imply that it suffices to show that 
\begin{equation}\label{suffice1}
E\star \kappa_{\hSigma_\sigma,\beta}(\cO_{\cX_{\S_\sigma,\beta}})\simeq E\star \kappa_{\hSigma'',\beta}(\cO_{\cX_{\hSigma''\beta}}).
\end{equation}
for $E\in \lSh_{\Lambda_{\hSigma,\beta}}(T^n)$ to prove (\ref{lemmaeq}) for $\Sigma$.
We have the $\mathrm{\check{C}ech}$ resolution of $\cO_{\cX_{\hSigma_\sigma,\beta}}$ associated to the covering $\{{\cU_{\sigma_i}}\}_{i=1}^r$
\begin{equation}\label{cech1}
0\rightarrow \cO_{\cX_{\hSigma_\sigma,\beta}}\rightarrow \bigoplus_{i}\Theta'(\sigma_i,0)\rightarrow \bigoplus_{i<j} \Theta'(\sigma_i\cap\sigma_j,0)\rightarrow\cdots.
\end{equation}
For $\cO_{\cX_{\hSigma'',\beta}}$, we consider the $\mathrm{\check{C}ech}$ resolution associated to the covering $\{\cX_{\hSigma'',\beta}\cap \cU_{\sigma_i}\}_{i=1}^r$. Note that $\cX_{\hSigma'',\beta}\cap \cU_{\sigma_i}$ can be written as $\cX_{\hSigma'',\beta}\cap \cU_{\sigma_i}=\cU_{\sigma_i'}$ by some $\sigma_i'\in \Sigma''$. Hence we have the $\mathrm{\check{C}ech}$ resolution
\begin{equation}\label{cech2}
0\rightarrow \cO_{\cX_{\hSigma'',\beta}}\rightarrow \bigoplus_{i}\Theta'(\sigma_i',0)\rightarrow \bigoplus_{i<j} \Theta'(\sigma_i'\cap\sigma_j',0)\rightarrow\cdots.
\end{equation}
It follows from (\ref{cech1}) and (\ref{cech2}), in order to obtain (\ref{suffice1}), it suffices to show that
\begin{equation}\label{suffice2}
E\star \kappa_{\hSigma_\sigma,\beta}(\Theta'(\sigma_{i_1}\cap\sigma_{i_2}\cap \cdots,0))\simeq E\star \kappa_{\hSigma_\sigma,\beta}(\Theta'(\sigma_{i_1}'\cap\sigma_{i_2}'\cap\cdots,0))
\end{equation}
for any $i_1<i_2<\cdots$ and $E\in \lSh_{\Lambda_{\hSigma,\beta}}(T^n)$. We  set $\sigma_{\bm{i}}:=\bigcap_{i_j\in \bm{i}}\sigma_{i_j}\in \Sigma_\sigma$ and $\sigma_{\bm{i}}':=\bigcap_{i_j\in \bm{i}}\sigma_{i_j}'\in \Sigma'$ for $\bm{i}=(i_1<\cdots<i_k)$. Then by the construction of $\ksb$ and Proposition \ref{affine}, we can rewrite (\ref{suffice2}) as 
\begin{equation}\label{suffice3}
E\star \Theta(\sigma_{\bm{i}},0)\simeq E\star \Theta(\sigma_{\bm{i}}',0).
\end{equation}
By the definition of $\Theta$, the equality (\ref{suffice3}) is equivalent to
\begin{equation}\label{suffice4}
E\star p_!\bC_{\Int(\sigma_{\bm{i}}^\vee)}\simeq E\star p_!\bC_{\Int(\sigma_{\bm{i}}'^\vee)}.
\end{equation}
Note that $\Int(\sigma_{\bm{i}}^\vee)\subset \Int(\sigma_{\bm{i}}'^\vee)$. Hence there is an exact triangle
\begin{equation}
E\star p_!\bC_{\Int(\sigma_{\bm{i}}^\vee)}\rightarrow E\star p_!\bC_{\Int(\sigma_{\bm{i}}'^\vee)}\rightarrow E\star p_!\bC_{\Int(\sigma_{\bm{i}}'^\vee)\bs\Int(\sigma_{\bm{i}}^\vee)}\rightarrow.
\end{equation}
Therefore, it suffices to prove
\begin{equation}\label{suffice5}
E\star p_!\bC_{\Int(\sigma_{\bm{i}}'^\vee)\bs\Int(\sigma_{\bm{i}}^\vee)}\simeq 0
\end{equation}
for $E\in \lSh_{\Lambda_{\hSigma,\beta}}(T^n)$.

Let $\{\rho_1,..., \rho_s\}$ the set of 1-dimensional faces of $\sigma_{\bm{i}}$, where $s$ depends on $\bm{i}$. Note that $\{\rho_1,...,\rho_s\}$ is a subset of $\{\rho^1,..., \rho^r,\rho_c\}$ and always contains $\rho_c$. Without loss of generality, we can assume that $\rho_s=\rho_c$. Then $\{\rho_1,...,\rho_{s-1}\}$ is the set of 1-dimensional cones of $\sigma_{\bm{i}}'$. We set 
\begin{align}
H_{\rho_i>0}&:=\lc m\in M_\bR\relmid \la m,\rho_i\ra>0\rc \text{ and } \\
H_{\rho_i\leq 0}&:=\lc m\in M_\bR\relmid \la m,\rho_i\ra \leq 0\rc.
\end{align}
Since
\begin{align}
\Int(\sigma_{\bm{i}}'^\vee)&=\bigcap_{i=1}^{s-1}H_{\rho_i>0} \text{ and}\\
\Int(\sigma_{\bm{i}}^\vee)&=\bigcap_{i=1}^sH_{\rho_i>0},
\end{align}
we have
\begin{equation}
\Int(\sigma_{\bm{i}}'^\vee)\bs\Int(\sigma_{\bm{i}}^\vee)=\bigcap_{i=1}^{s-1}H_{\rho_i>0}\cap H_{\rho_s\leq 0}.
\end{equation}
We set 
\begin{equation}
H(i_1,...,i_l; i_{l+1},..., i_t):=\bigcap_{j=1}^lH_{\rho_{i_j}>0}\cap\bigcap_{j=l+1}^tH_{\rho_{i_j}\leq 0}
\end{equation}
so that $H(1,...,{s-1};s)=\Int(\sigma_{\bm{i}}'^\vee)\bs\Int(\sigma_{\bm{i}}^\vee)$. Each $H(\varnothing; {i_1},..., {i_t},s)$ is a closed cone and its dual cone is contained in $-\sigma$ but are not contained in proper faces of $-\sigma$. Hence by Proposition \ref{vanishing} and the assumption $\musupp(E)\cap \left(T^n\times \Int(-\sigma)\right)=\varnothing$, we have 
\begin{equation}\label{vanishing2}
E\star p_!\bC_{H(\varnothing; {i_1},..., {i_t},s)}\simeq 0
\end{equation}
for any $\{i_1,...,i_t\}\subset \{1,....,s-1\}$.

In the following, we prove (\ref{suffice5}) by induction.
Suppose that 
\begin{equation}
E\star p_!\bC_{H({i_1},..., {i_k}; i_{k+1},...,{i_t},s)}\simeq 0
\end{equation}
for fixed $k$ and any $t$ and $i_{j}$ as the induction hypothesis. There exists an equality
\begin{equation}\label{9.20}
H(i_1,..., i_k;i_{k+2},...,i_t,s)\bs H(i_1,...,i_{k},i_{k+1};i_{k+2},...,i_t,s)=H(i_{1},...,i_{k};i_{k+1},i_{k+2},...,i_t,s).
\end{equation}
By the induction hypothesis and (\ref{9.20}), we have
\begin{equation}
E\star p_!\bC_{H(i_1,...,i_{k},i_{k+1};i_{k+2},...,i_t,s)} \simeq 0.
\end{equation}
Hence by induction from (\ref{vanishing2}), we have (\ref{suffice5}).
\end{proof}

\begin{corollary}\label{lemma9.2}
Condition \ref{condition2} implies the following: For a face inclusion $\sigma_1\subset \sigma_2$, there exists a natural isomorphism
\begin{equation}
\Ind K_{\hsigma_2,\beta}\circ i^\Ind_{\sigma_1\sigma_2*}\simeq \Ind K_{\hsigma_1,\beta}.
\end{equation}
\end{corollary}
\begin{proof}
By Condition \ref{condition2} (i), we have
\begin{equation}
\begin{split}
\hom(E, \Ind K_{\hsigma_2,\beta}\circ i_{\sigma_1\sigma_2*}^\Ind (\cE))
&\simeq \hom(i^*_{\sigma_1\sigma_2}\Ind K_{\hsigma_2,\beta}^{-1}(E), \cE)\\
&\simeq \hom(\Ind K_{\hsigma_1,\beta}i^*_{\sigma_1\sigma_2}\Ind K_{\hsigma_2,\beta}^{-1}(E), \Ind K_{\hsigma_1,\beta}(\cE))\\
&\simeq \hom(E\star \Theta(\sigma_1,0), \Ind K_{\hsigma_1,\beta}(\cE))\\
&\simeq \hom(E, \cHom^\star(\Theta(\sigma_1,0),\Ind K_{\hsigma_1,\beta}(\cE)))\\
&\simeq \hom(E, \Ind K_{\hsigma_1,\beta}(\cE)),
\end{split}
\end{equation}
where the last equality follows from Lemma \ref{lemmaeq}. This completes tha proof.
\end{proof}

\section{Gluing equivalences}\label{section:main}
We prove Theorem \ref{main} and Theorem \ref{main2} below under Condition \ref{condition2} simultaneously. 
\begin{theorem}\label{main2}There exists an equivalence of $\infty$-categories
\begin{equation}
\lSh_{\Lsb}(T^n)\simeq  \lim{\substack{\longleftarrow \\C(\Sigma)}}\lSh_{\Lambda_\bullet}(T^n).
\end{equation}
\end{theorem}

\begin{proposition}\label{mainaux}
Theorem \ref{main} and Theorem \ref{main2} hold under Condition \ref{condition2}.
\end{proposition}
\begin{proof}
Let $\hSigma_i$ for $i=0,1,2$ be subfans of $\hSigma$ with $\hSigma_0=\hSigma_1\cup \hSigma_2$ and $\hSigma_{12}:=\hSigma_1\cap \hSigma_2$. In this proof, we use the notation in the proof of Proposition \ref{gluingcohgeneral} with $X=X_{\hSigma_0}$, $U=X_{\hSigma_1}$, and $V=X_{\hSigma_2}$. 

We set
\begin{equation}
J^{i}:=(-)\star\kappa_{\hSigma_i,\beta}(\cO_{\cX_{\hSigma_{12},\beta}})\colon \lSh_{\Lsbd{i}}(T^n)\rightarrow \lSh_{\Lsbd{12}}(T^n)
\end{equation}
and 
\begin{equation}
I^i:=(-)\star \kappa_{\hSigma_0,\beta}(\cO_{\cX_{\hSigma_i,\beta}})\colon \lSh_{\Lsb}(T^n)\rightarrow \lSh_{\Lsbd{i}}(T^n)
\end{equation}
for $i=1,2$.
The well-definedness of $J^i$ and $I^i$ follows from the assumption that $\kappa_{\hSigma_1,\beta}$ and $\kappa_{\hSigma_2,\beta}$ are equivalences and the following: By Corollary \ref{gluingcoh} and Condition \ref{condition2} (i), we have
\begin{align}
J^i\circ K_{\hSigma_{i},\beta}&\simeq K_{\hSigma_{12},\beta} \circ j^*_{X_{\hSigma_i}} \text{ and}\\
I^i\circ K_{\hSigma,\beta}&\simeq K_{\hSigma_{i},\beta} \circ i^*_{X_{\hSigma_i}}.
\end{align}
As in the proof of Proposition \ref{gluingcohgeneral}, it suffices to show that 
\begin{equation}\label{ctglue}
\lSh_{\Lsbd{0}}(T^n)\simeq \lSh_{\Lsbd{1}}(T^n)\tim{\lSh_{\Lambda_{\hSigma_{12},\beta}}(T^n)}{h}\lSh_{\Lsbd{2}}(T^n)
\end{equation}
where the right hand side is defined by using $J^i$ for $i=1,2$.
Moreover, there exists the diagram of inclusions
\begin{equation}
\xymatrix{
\lSh_{\Lsbd{0}}(T^n) &\ar[l]_-{I_1} \lSh_{\Lsbd{1}}(T^n) \\
\lSh_{\Lsbd{2}}(T^n)  \ar[u]^{I_2}  & \ar[l]^-{J_2} \lSh_{\Lsbd{12}}(T^n). \ar[u]_{J_1}\ar[ul]_{K}
}
\end{equation}
By Corollay \ref{lemma9.2}, we have 
\begin{equation}\label{openpushcomm}
\begin{split}
J_i\circ K_{\hSigma_{12},\beta}&\simeq K_{\hSigma_{i},\beta} \circ j^\Ind_{X_{\hSigma_i}*}, \\
I_i\circ K_{\hSigma_{i},\beta}&\simeq K_{\hSigma,\beta} \circ i^\Ind_{X_{\hSigma_i}*}\text{,and}\\
K\circ  K_{\hSigma_{12},\beta}&\simeq K_{\hSigma,\beta} \circ k^\Ind_*
\end{split}
\end{equation}
where $k$ is the notation in (\ref{7.6}).

We write
\begin{equation}
\lSh_{\Lsbd{1}}(T^n)\xrightarrow{\gamma(J^1)} \cB \xrightarrow{\delta(J^1)} \lSh_{\Lsbd{12}}(T^n).
\end{equation}
for the functorial factorization of $J^1$ in the Morita model structure.
Then we can calculate the right hand side of (\ref{ctglue}) as
\begin{equation}
\lSh_{\Lsbd{1}}(T^n)\tim{\lSh_{\Lsbd{12}}(T^n)}{h}\lSh_{\Lsbd{2}}(T^n)\simeq\cB\tim{\lSh_{\Lsbd{12}}(T^n)}{} \lSh_{\Lsbd{2}}(T^n),
\end{equation}
since $\lSh_\Lambda$'s are idempotent-complete pretriangulated dg-categories. 

We have a functor $L\colon \cB\rightarrow \lSh_{\Lsbd{1}}(T^n)$ which fits into the commutative diagram
\begin{equation}
\xymatrix{
\lSh_{\Lsbd{1}}(T^n) \ar[r]_-{\id} \ar[d]_{\gamma(J^1)}   & \lSh_{\Lsbd{1}}(T^n) \ar[d]\ar[r]_{I_1}& \lSh_{\Lsbd{0}}(T^n) \\
\cB  \ar[r]\ar@{-->}[ru]_{L}& \ast. &
}
\end{equation}
By using $L$, there exists a functor $\Phi\colon \cB\tim{\lSh_{\Lsbd{12}}(T^n)}{}\lSh_{\Lsbd{2}}(T^n)\rightarrow \lSh_{\Lsbd{0}}(T^n)$ which is defined on objects as
\begin{equation}
\Phi(b,E):= I_1L(b)\tim{I_2J_2J^2(E)}{h} I_2(E).
\end{equation}

We show the functor $\Psi\colon \lSh_{\Lsbd{0}}(T^n)\rightarrow \cB\tim{\lSh_{\Lsbd{12}}(T^n)}{}\lSh_{\Lsbd{2}}(T^n)$ given by
\begin{equation}
\Psi(E):=\lb\gamma(J^1)I^1(E), I^2(E)\rb
\end{equation}
on objects is the quasi-inverse of $\psi$.

For $E\in \lSh_{\Lsbd{0}}(T^n)$, we have
\begin{equation}
I_1I^1L\gamma(J^1)I^1(E)\tim{I_2J_2J^2I^2(E)}{h}I_2I^2(E)= I_1I^1(E)\tim{KK'(E)}{h}I_2I^2(E).
\end{equation}
where we set $K':=J^2I^2=J^1I^1$.
This fits into the exact triangle
\begin{equation}\label{tri1}
I_1I^1(E)\tim{KK'(E)}{h}I_2I^2(E)\rightarrow I_1I^1(E)\oplus I_2I^2(E)\rightarrow KK'(E)\rightarrow.
\end{equation}
The exact triangle (\ref{tri1}) can be rewritten as 
\begin{equation}
I_1I^1(E)\tim{K'(E)}{h}I_2I^2(E)\rightarrow E\star \kappa_{\hSigma_0,\beta}(\cO_{\cX_{\hSigma_1,\beta}}\oplus \cO_{\cX_{\hSigma_2,\beta}})\rightarrow E\star \kappa_{\hSigma_{0},\beta}(\cO_{\cX_{\hSigma_{12},\beta}})\rightarrow.
\end{equation}
Since $\Cone\left((\cO_{\cX_{\hSigma_1,\beta}}\oplus \cO_{\cX_{\hSigma_2,\beta}})\rightarrow \cO_{\cX_{\hSigma_{12},\beta}}\right)\simeq \cO_{\cX_{\hSigma_0},\beta}$, we have
\begin{equation}\label{9.34}
I^1(E)\tim{K(E)}{h}I^2(E)\simeq E\star \kappa_{\hSigma_0,\beta}(\cO_{\cX_{\hSigma_0},\beta}).
\end{equation}
By Lemma \ref{lemmaeq} and (\ref{9.34}), we have
\begin{equation}
I_1I^1(E)\tim{KK'(E)}{h}I_2I^2(E)\simeq E.
\end{equation}
Hence the functor $H^0(\Phi)$ is essentially surjective.

We consider the following diagram
\begin{equation}\label{thediagram}
\xymatrix{
\Indcoh \cX_{\hSigma_1,\beta}\tim{\Indcoh \cX_{\hSigma_{12},\beta}}{h}\Indcoh \cX_{\hSigma_2,\beta} \ar[d]^{K_{\hSigma_1,\beta}\tim{K_{\hSigma_{12},\beta}}{h}K_{\hSigma_2,\beta}}  \ar[r]_-{\phi} &\Indcoh \cX_{\hSigma_0,\beta} \ar[d]_{K_{\hSigma_0,\beta}} \\
\lSh_{\Lambda_{\hSigma_1,\beta}}(T^n)\tim{\lSh_{\Lambda_{\hSigma_{12},\beta}}(T^n)}{h}\lSh_{{\Lambda_{\hSigma_2,\beta}}}(T^n)  \ar[r]^-{\Phi} &\lSh_{\Lambda_{\hSigma_0,\beta}}(T^n). 
}
\end{equation}
By the construction of $K_{\hSigma,\beta}$ as the homotopy limit (\ref{defofkappa}) and (\ref{openpushcomm}), the diagram (\ref{thediagram}) is homotopy commutative.
Now we prove the theorems by induction on $\Sigma$. We assume that Theorem \ref{main2} and Theorem \ref{main} hold for proper subfans of $\Sigma_0$ as the induction hypothesis. Then in the diagram (\ref{thediagram}), the functors $\phi$ and $K_{\hSigma_1,\beta}\tim{K_{\hSigma_{12},\beta}}{h}K_{\hSigma_2,\beta}$ are equivalences. Since $H^0(\Phi)$ is essentially surjective, the functor $K_{\hSigma_0,\beta}$ is also essentially surjective. Hence the functor $K_{\hSigma_0,\beta}$ is an equivalence by Lemma \ref{fullyfaithful}. Therefore, $\Phi$ is also an equivalence.
\end{proof}

\begin{corollary}\label{smoothmain}
Assume $\hSigma$ is smooth. Then Theorem \ref{main} and Theorem \ref{main2} hold.
\end{corollary}
\begin{proof}
Since smooth $\hSigma$ satisfies Condition \ref{condition2} by Corollary \ref{idfunctor}, Proposition \ref{mainaux} completes the proof.
\end{proof}

\section{General case}\label{generalcase}
Let $\hSigma(\hsigma)$ be affine. Then we have a fully faithful functor $\kappa_{\hsigma,\beta} \colon \perf\cX_{\hSs,\beta}\rightarrow\wSh_{\Lambda_{\hSs,\beta}}(T^n)$. Set $D:=\hom(-, \cO_{\cX_{\hSigma(\hsigma),\beta}})\colon \perf\cX_{\hSs,\beta}\rightarrow (\perf\cX_{\hSs,\beta})^\op$ where $(-)^\op$ denotes the opposite category. We define the composition
\begin{equation}
\kappa_{\hsigma,\beta}^{D}:=\iota\circ \Ind (\kappa_{\hsigma,\beta})^\op \circ \Ind D\circ \iota_{\Coh}\colon \Coh \cX_{\hSigma,\beta}\rightarrow \Fun(\lSh_{\Lambda_{\hSigma,\beta}}(T^n), \Mod(\bC)).
\end{equation}
where $\Fun(\lSh_{\Lambda_{\hSigma,\beta}}(T^n), \Mod(\bC))$ is the functor category from $\lSh_{\Lambda_{\hSigma,\beta}}(T^n)$ to $\Mod(\bC)$, and
\begin{align}
\iota_{\Coh}&\colon \Coh \cX_{\hSs,\beta}\hookrightarrow \Qcoh \cX_{\hSs,\beta}\simeq \Ind\perf \cX_{\hSs,\beta},\\
\Ind D\colon&\Ind\perf \cX_{\hSs,\beta}\xrightarrow{\Ind D}  \Ind(\perf \cX_{\hSs,\beta})^\op, \\
{\Ind (\kappa_{\hsigma,\beta})^\op}\colon & \Ind(\perf \cX_{\hSs,\beta})^\op \rightarrow\Ind(\wSh_{\Lambda_{\hSs,\beta}}(T^n))^{\op}, \\
\iota\colon &\Ind (\wSh_{\Lambda_{\hSs,\beta}}(T^n))^{\op}\rightarrow \Ind (\lSh_{\Lambda_{\hSs,\beta}}(T^n))^{\op}\hookrightarrow \Fun(\lSh_{\Lambda_{\hSs,\beta}}(T^n), \Mod(\bC)).
\end{align}

\begin{lemma}\label{fullyfaithful2}
The functor $\kappa_{\hsigma,\beta}^{D}$ is a fully faithful functor.
\end{lemma}
\begin{proof}
This is clear from the definition of the functor, since $\kappa_{\hsigma,\beta}^{D}$ is a composition of fully faithful functors.
\end{proof}

More explicitly, the functor $\kappa_{\hsigma,\beta}^{D}$ can be described as follows: Let $\cE$ be a coherent sheaf over $\cX_{\hSigma,\beta}$. There exists a sequence of perfect complexes $\cE_i$ such that $\varhocolim{i}\cE_i$ in $\Qcoh \cX_{\hSs,\beta}\simeq \cE$. Then we have
\begin{equation}\label{formula}
\kappa_{\hsigma,\beta}^{D}(\cE)\simeq \varhocolim{i}\hom(\kappa_{\hsigma,\beta}\circ D(\cE_i), -).
\end{equation}

There exists another presentation of $\kappa^{D}_{\hsigma,\beta}$ as follows. We will use notations in Section \ref{affinecase}. For an object $F\in \lSh_{\Lambda_{\hSigma,\beta}}(T^n)$, the space 
\begin{equation}
\hom(\Theta(\sigma), F):=\hom(\oplus_{[\chi]\in M_{\sigma,\beta}/M}\Theta(\sigma,-\chi), F)
\end{equation}
is naturally equipped with a $\bC[\sigma^\vee\cap M_{\sigma,\beta}]$-module structure by the composition with 
\begin{equation}
\hom(\Theta(\sigma,\chi_1), \Theta(\sigma,\chi_2))\simeq \bC[(\sigma^\vee\cap M)+(\chi_2-\chi_1)].
\end{equation}
We assign a $H_\beta$-weight $-\chi$ to $\hom(p_!\bC_{\Int(\sigma^\vee)-\chi}, F)$, then $\hom(\Theta(\sigma), F)$ becomes a $\bC[\sigma^\vee\cap M_{\sigma,\beta}]\rtimes H_\beta$-module.
\begin{lemma}\label{natisom}
There exists a natural isomorphism
\begin{equation}
\begin{split}
\kappa_{\hsigma,\beta}^{D}(\cE)
&\simeq (\cE\otimes_{\bC[\sigma^\vee\cap M_{\sigma,\beta}]}\hom(\Theta(\sigma),-))^{H_\beta}\\
&\simeq \bigoplus_{[\chi]\in M_{\sigma,\beta}/M}(\cE\cdot e_{\chi})\otimes_{\bC[\sigma^\vee\cap M]} \hom(\Theta(\sigma,-\chi), -)
\end{split}
\end{equation}
for $\cE\in \Coh \cX_{\hSs,\beta}$ where $(-)^{H_\beta}$ is $H_\beta$-invariant.
\end{lemma}
\begin{proof}
For a coherent sheaf $\cE$, take a projective resolution
\begin{equation}
0\leftarrow \cE\leftarrow \bigoplus_{[\chi]\in M_{\sigma,\beta}/M}\cO(\chi)^{\oplus n^{\chi}_1}\xleftarrow{d_1} \bigoplus_{[\chi]\in M_{\sigma,\beta}/M}\cO(\chi)^{\oplus n^\chi_2}\xleftarrow{d_2} \bigoplus_{[\chi]\in M_{\sigma,\beta}/M}\cO(\chi)^{\oplus n^\chi_3}\leftarrow \cdots.
\end{equation}
We set
\begin{equation}
\cE_i:=\lb \bigoplus_{[\chi]\in M_{\sigma,\beta}/M}\cO(\chi)^{\oplus n^\chi_1}\xleftarrow{d_1} \cdots \xleftarrow{d_{i-1}} \bigoplus_{[\chi]\in M_{\sigma,\beta}/M}\cO(\chi)^{\oplus n^\chi_i}\rb,
\end{equation}
then we have $\cE=\colim\cE_i$. Then we have
\begin{equation}\label{cones}
\kappa_{\hsigma,\beta}\circ D(\cE_i)=\lb \bigoplus_{[\chi]\in M_{\sigma,\beta}/M}p_!\bC_{\Int(\sigma^\vee)-\chi}^{\oplus n^\chi_1} \xrightarrow{\kappa_{\hsigma,\beta}(d_1^t)}\cdots \xrightarrow{\kappa_{\hsigma,\beta}(d_i^t)} \bigoplus_{[\chi]\in M_{\sigma,\beta}/M} p_!\bC_{\Int(\sigma^\vee)-\chi}^{\oplus n^\chi_i} \rb
\end{equation}
where $d_i^t$ is the transpose of $d_i$. Then we have
\begin{equation}
\begin{split}
&\hom(\kappa_{\hsigma,\beta}\circ D(\cE_i), \cF)\\
&\simeq \lb\lb \bigoplus_{[\chi]\in M_{\sigma,\beta}/M} \bC[\sigma^\vee\cap M_{\sigma,\beta}]\cdot e_{\chi}^{\oplus n^\chi_1}\xleftarrow{d_1}\cdots \xleftarrow{d_i}\bigoplus_{[\chi]\in M_{\sigma,\beta}/M} \bC[\sigma^\vee\cap M_{\sigma,\beta}]\cdot e_{\chi}^{n_i^\chi}\rb\tens{\bC[\sigma^\vee\cap M_{\sigma,\beta}]}{} \hom(\Theta(\sigma), F)\rb^{H_\beta}\\
&\simeq\lb \cE_i\tens{\bC[\sigma^\vee\cap M_{\sigma,\beta}]}{}\hom(\Theta(\sigma), F)\rb^{H_\beta}
\end{split}
\end{equation}
by (\ref{cones}). Since the colimit $\varhocolim{i} \cE_i$ is $H_\beta$-equivariant, we have
\begin{equation}
\varhocolim{i}\hom(\kappa_{\hsigma,\beta}\circ D(\cE_i), F)\simeq \lb\cE\tens{\bC[\sigma^\vee\cap M_{\sigma,\beta}]}{}\hom(\Theta(\sigma), F)\rb^{H_\beta}
\end{equation}
by (\ref{formula}).
\end{proof}

Before going further, we prepare some lemmas. Let us assume $\sigma$ to be full-dimensional until the end of the proof of Lemma \ref{Davanishing}. Let $\{v_1,...,v_k\}$ the set of primitive generators of edges of $\sigma$. We index the set of connected components of $M_\bR\bs \lb \bigcup_{m\in M_{\sigma,\beta}}\bigcup_{i=1}^kv_i^\perp+m\rb$ by $\bN$. For $a\in \bN$, the corresponding component $D'_a$ have a presentation
\begin{equation}\label{Dapresentation}
D'_a=\bigcap_{\alpha=1}^{\alpha_a} H_{v_{i_\alpha}>k_{i_\alpha}}\cap \bigcap_{\beta=1}^{\beta_a} H_{v_{i_\beta}<l_{i_\beta}}
\end{equation}
where $H_{v\gtrless k}:=\lc m\in M_\bR\relmid \la m, v\ra \gtrless k\rc$ and each $k_i$ and $l_j$ is some integer. We assume that there is no redundancy in the presentation (\ref{Dapresentation}) i.e. each $H_{v_i=k}$ ($k=k_i, l_i$) defines a facet of the closure of $D_a'$. Each $D_a'$ is a bounded open polytope by the assumption that $\sigma$ is full-dimensional. We set
\begin{equation}
D_a:=\bigcap_{\alpha=1}^{\alpha_a} H_{v_{i_\alpha}>k_i}\cap \bigcap_{\beta=1}^{\beta_a} H_{v_{i_\beta}\leq l_{i_{\beta}}}.
\end{equation}
Then we have
\begin{equation}
M_\bR=\bigsqcup_{a} D_a.
\end{equation}
We also set
\begin{equation}
\bD D_a:=\bigcap_{\alpha=1}^{\alpha_a} H_{v_{i_\alpha}\geq k_{i_\alpha}}\cap \bigcap_{\beta=1}^{\beta_a} H_{v_{i_\beta} <l_{i_\beta}}.
\end{equation}
Let $dp\colon T^*M_\bR\rightarrow T^*T^n$ be the differential of the quotient map $M_\bR\rightarrow T^n$.
\begin{lemma}\label{Davanishing}
For $E\in \lSh_{dp^{-1}\Lambda_{\hSs,\beta}}(M_\bR)$, there exists a quasi-isomorphism
\begin{equation}
\hom(\bC_{-\bD D_a}, E)\simeq 0.
\end{equation}
if $\bD D_a$ does not intersect with $M_{\sigma,\beta}$.
\end{lemma}
\begin{proof}
First, assume that $\bigcap_{\alpha=1}^{\alpha_a} H_{v_{i_\alpha}\leq -k_{i_\alpha}}$ has no bounded faces. Hence $\alpha_a$ is less than the dimension of $M_\bR$. The fact $D_a$ is bounded implies that $\Cone(\{v_{i_1},...,, v_{i_{\alpha_a}}\})$ does not form any face of $\sigma$. By the non-characteristic deformation, $\hom(\bC_{-\bD D_a}, E)$ measures the microsupport of $E$ over a point in $\bigcap_{\alpha=1}^{\alpha_a} H_{v_{i_\alpha}= -k_{i_\alpha}}\cap (-\bD D_a)$ to the direction $\Cone(\{v_1,...,, v_{i_\alpha}\})$, however this vanishes by $E\in \lSh_{dp^{-1}\Lambda_{\hSs,\beta}}(M_\bR)$.

Then assume that $\bigcap_{\alpha=1}^{\alpha_a} H_{v_{i_\alpha}\leq -k_{i_\alpha}}$ has at least one bounded face. 
By the non-characteristic deformation, there are nontrivial contribution to this local cohomology only from the bounded faces of $\bigcap_{\alpha=1}^{\alpha_a} H_{v_{i_\alpha}\leq -k_{i_\alpha}}$. The subset $\{v_{i_{s_1}},..., v_{i_{s_k}}\}$ of $\{v_{i_1},...,, v_{i_{\alpha_a}}\}$ forming a bounded face has one of the following properties: (i) $\Cone(\{v_{i_{s_1}},..., v_{i_{s_k}}\})$ does not form any face of $\sigma$, or (ii) $\{v_{i_{s_1}},..., v_{i_{s_k}}\}$ is the set of primitive generators of $\sigma$. For the case (i), the local cohomology contributing to $\hom(\bC_{-\bD D_a}, E)$ vanishes by $E\in \lSh_{dp^{-1}\Lambda_{\hSs,\beta}}(M_\bR)$. For the case (ii), the local cohomology contributing to $\hom(\bC_{-\bD D_a}, E)$ also vanishes by the assumption $-\bD D_a$ does not intersect with $M_{\sigma,\beta}$, since $E$ can have microsupport to the direction $-\sigma$ only on $M_{\sigma,\beta}$.
\end{proof}

Let $\hSigma$ be a smooth refinement of $\hsigma$. Let $f$ be the morphism $\cX_{\hSigma,\beta}\rightarrow \cX_{\hSs,\beta}$ associated to the refinement. Let further $I$ be the inclusion $\lSh_{\Lambda_{\hSs,\beta}}(T^n)\hookrightarrow \lSh_{\Lambda_{\hSigma,\beta}}(T^n)$. Note that the $K_{\hSigma,\beta}$ is an equivalence by Corollary \ref{smoothmain}.
\begin{lemma}\label{okgo}
There exists a quasi-isomorphism 
\begin{equation}
f^*\hom(\Theta(\sigma), F)\simeq \kappa^{-1}_{\hSigma,\beta}(IF)
\end{equation}
for $F\in \lSh_{\Lambda_{\hSs,\beta}}(T^n)$.
\end{lemma}
\begin{proof}
To prove this lemma, it suffices to show that 
\begin{equation}\label{tauisom}
f^*\hom(\Theta(\sigma), F)|_{\cU_\tau}\simeq \kappa^{-1}_{\hSigma,\beta}(IF)|_{\cU_\tau}
\end{equation}
for any $\tau\in \Sigma$.
Since $\Qcoh \cU_\tau$ is generated by $\Theta(\tau)$, the quasi-isomorphism (\ref{tauisom}) is equivalent to 
\begin{equation}
\hom(\Theta(\tau), f^*\hom(\Theta(\sigma), F)|_{\cU_\tau})\simeq \hom(\Theta(\tau), \kappa^{-1}_{\hSigma,\beta}(IF)|_{\cU_\tau}).
\end{equation}

The right hand side of (\ref{tauisom}) is written as
\begin{equation}
\begin{split}
\hom(\oplus_{[\chi]\in M_{\sigma,\beta}/M}\cO_{\cU_\tau}(-\chi), &\kappa^{-1}_{\hSigma,\beta}(IF)|_{\cU_\tau})\\
&\simeq \hom(\oplus_{[\chi]\in M_{\sigma,\beta}/M}\cO_{\cX_{\hSigma,\beta}}(-\chi), i_{\cU_\tau*}(\kappa^{-1}_{\hSigma,\beta}(IF)|_{\cU_\tau}))\\
&\simeq \hom(\oplus_{[\chi]\in M_{\sigma,\beta}/M}p_!\bC_{\Int(\sigma^\vee)-\chi}, F\star p_!\bC_{\Int(\tau^\vee)})
\end{split}
\end{equation}
by Corollary \ref{lemma9.2}.

The left hand side of (\ref{tauisom}) is calculated as follows: First, we have
\begin{equation}
f^*\hom(\Theta(\sigma), F)|_{\cU_\tau}\simeq \bC[\tau^\vee\cap M_{\sigma,\beta}]\tens{\bC[\sigma^\vee\cap M_{\sigma,\beta}]}{}\hom(\Theta(\sigma), F).
\end{equation}
Then we have
\begin{equation}
\begin{split}
\hom(\Theta(\tau), f^*\hom(\Theta(\sigma), F)|_{\cU_\tau})&\simeq \hom(\Theta(\tau),\bC[\tau^\vee\cap M_{\sigma,\beta}]\tens{\bC[\sigma^\vee\cap M_{\sigma,\beta}]}{}\hom(\Theta(\sigma), F))\\
&\simeq \bC[\tau^\vee\cap M_{\sigma,\beta}]\tens{\bC[\sigma^\vee\cap M_{\sigma,\beta}]}{}\hom(\Theta(\sigma), F).
\end{split}
\end{equation}
Since $\kappa_{\hSigma,\beta}$ is an equivalence, we have
\begin{equation}
\begin{split}
\hom(p_!\bC_{\Int(\sigma^\vee)-\chi}, F)&\simeq \hom(\kappa_{\hSigma,\beta}^{-1}(p_!\bC_{\Int(\sigma^\vee)-\chi}), \kappa_{\hSigma,\beta}^{-1}(IF))\\
&\simeq \hom(\cO_{\cX_{\hSigma,\beta}}(-\chi), \kappa_{\hSigma,\beta}^{-1}(IF)).
\end{split}
\end{equation}

Hence we will prove 
\begin{equation}\label{11.16}
\begin{split}
\bC[{\tau}^\vee\cap M_{\sigma,\beta}]\otimes_{\bC[\sigma^\vee\cap M_{\sigma,\beta}]}&\hom(\oplus_{[\chi]\in M_{\sigma,\beta}/M}\cO_{\cX_{\hSigma,\beta}}(-\chi), \kappa_{\hSigma,\beta}^{-1}(IF))\\
&\simeq  \hom(\oplus_{[\chi]\in M_{\sigma,\beta}/M}p_!\bC_{\Int(\sigma^\vee)-\chi}, F\star p_!\bC_{\Int(\tau^\vee)}).
\end{split}
\end{equation}
for any $\hat{\tau}\in\hSigma$.

Set $S:=\bigcup_{m\in \tau^\vee\cap M_{\sigma,\beta}}(\Int(\sigma^\vee)+m)$ and $C:=\Int(\tau^\vee)\bs S$. Note tha $C$ does not contain any elements of $M_{\sigma,\beta}$. Then we have an exact triangle
\begin{equation}\label{cstriangle}
\bC_{S}\rightarrow \bC_{\Int(\sigma^\vee)}\rightarrow \bC_{C}\rightarrow.
\end{equation}
First, we will prove 
\begin{equation}\label{firstvanishing12}
\hom(p_!\bC_{\Int(\sigma^\vee)-\chi}, F\star p_!\bC_C)\simeq 0.
\end{equation}
We can assume that $\sigma$ is full-dimensional  to prove (\ref{firstvanishing12}), since all appearing sheaves are constant to the direction $\sigma^\perp$. We can also assume $\chi=0$ by replacing $F$.

The set $C\cap (\Int(\sigma^\vee)+m)$ for $m\in M_\bR$ is bounded. Then we have a presentation of $\bC_C$ as a colimit of compactly supported constructible sheaves with microsupports in $T^n\times (-\sigma)$ :
\begin{equation}
\bC_{C}\simeq \varhocolim{m\in M_{\sigma,\beta}}\bC_{C\cap (\Int(\sigma^\vee)+m)}
\end{equation}
where the colimit is taken with respect to maps $\bC_{C\cap (\Int(\sigma^\vee)+m_1)}\rightarrow \bC_{C\cap (\Int(\sigma^\vee)+m_2)}$ for $\Int(\sigma^\vee)+m_1\subset \Int(\sigma_2^\vee)+m_2$ which correspond to the identity via the isomorphisms
\begin{equation}
\hom(\bC_{C\cap (\Int(\sigma^\vee)+m_1)},\bC_{C\cap (\Int(\sigma^\vee)+m_2)})\simeq \hom(\bC_{C\cap (\Int(\sigma^\vee)+m_1)},\bC_{C\cap (\Int(\sigma^\vee)+m_1)}).
\end{equation}

Set $G_m:=\bC_{C\cap (\Int(\sigma^\vee)+m)}$. Then we have
\begin{equation}
\begin{split}
\hom(p_!\bC_{\Int(\sigma^\vee)}, F\star p_!\bC_C)&\simeq \hom(p_!\bC_{\Int(\sigma^\vee)}, \varhocolim{m\in M_{\beta}} (F\star p_!G_m))\\
&\simeq \varhocolim{m\in M_{\beta}}\hom(p_!\bC_{\Int(\sigma^\vee)}, F\star p_!G_m)\\
&\simeq \varhocolim{m\in M_{\beta}}\hom(\bC_{\Int(\sigma^\vee)}, p^{-1}(F\star p_!G_m))\\
&\simeq \varhocolim{m\in M_{\beta}}\hom(\bC_{\Int(\sigma^\vee)}, p^{-1}F\star_{\bR} p^{-1}p_!G_m))\\
&\simeq \varhocolim{m\in M_{\sigma,\beta}}\bigoplus_{m'\in M}\hom(\bC_{\Int(\sigma^\vee)+m'}, p^{-1}F\star_{\bR} G_m)
\end{split}
\end{equation}
by Lemma \ref{lem:cpt}.

Let $B$ be a bounded open neighborhood of $0$. 
By the non-characteristic deformation as in the proof of Lemma \ref{lem:cpt}, we have
\begin{equation}
\begin{split}
\hom(\bC_{\Int(\sigma^\vee)+m'}, p^{-1}F\star_{\bR} G_m)\simeq \hom(\bC_{(\Int(\sigma^\vee)\cap B)+m'}, p^{-1}F\star_{\bR} G_m).
\end{split}
\end{equation}
Since $G_m$ can be written as a finite extension of $\bC_{D_a}$'s, Lemma \ref{cptadjunction} implies
\begin{equation}
\varhocolim{j}\hom(\bC_{(\Int(\sigma^\vee)\cap B)+m'}, F_j\star_{\bR} G_m)
\simeq \hom(\bC_{\Int(\sigma^\vee)+m'}\star(-1)^*\bD G_m,F).
\end{equation}
Here $\bC_{\Int(\sigma^\vee)+m'}\star(-1)^*\bD G_m$ can be written as a finite extension of $\bC_{-\bD D_a}$ without intersections with $M_{\sigma,\beta}$. Hence, by Lemma \ref{Davanishing}, this is quasi-isomorphic to $0$. This proves (\ref{firstvanishing12}).

Hence we have 
\begin{equation}\label{11.20}
\hom(p_!\bC_{\Int(\sigma^\vee)-\chi}, F\star p_!\bC_{\Int(\tau^\vee)})\simeq \hom(p_!\bC_{\Int(\sigma^\vee)-\chi}, F\star p_!\bC_{S}).
\end{equation}
by (\ref{firstvanishing12}) and (\ref{cstriangle}). We will calculate the right hand side of (\ref{11.20}). We have
\begin{equation}\label{11.21}
\begin{split}
\hom(p_!\bC_{\Int(\sigma^\vee)-\chi}, F\star p_!\bC_{S})&\simeq \hom(\bC_{\Int(\sigma^\vee)-\chi}, p^{-1}(F\star p_!\bC_{S}))\\
&\simeq \hom(\bC_{\Int(\sigma^\vee)-\chi}, \tilde{m}_!(p^{-1}F\boxtimes p^{-1}p_!\bC_{S}))\\
&\simeq \bigoplus_{m\in M}\hom(\bC_{\Int(\sigma^\vee)+m-\chi},  p^{-1}F\star_{\bR}\bC_{S})
\end{split}
\end{equation}
by Lemma \ref{lem:cpt}.

We fix $m\in M\bs \tau^\vee$ and set $S^m:=S\cap (\sigma^\vee+m)$. Then $S^m\cap M_{\sigma,\beta}$ is a finitely-generated module over the semigroup $\sigma^\vee\cap M_{\sigma,\beta}$ by Gordon's lemma. We can take the following resolution of $S^m \cap M_{\sigma,\beta}$: First, take a set of generators $m_1,...,m_{s}$ of $S^m \cap M_{\sigma,\beta}$ which are distinct from each other. For each pair $(i,j)\in \{1,...,s\}^{\times 2}$, we consider the set $(\sigma^\vee+m_i)\cap (\sigma^\vee+m_j)\cap M_{\sigma,\beta}$, which is the relation between the submodules generated by $m_i$ and $m_j$ respectively. The set $(\sigma^\vee+m_i)\cap (\sigma^\vee+m_j)$ is again a rational convex polyhedron hence gives a finitely generated module over $\sigma^\vee\cap M_{\sigma,\beta}$. Then we take a set of generators $\{m^{ij}_{1},...., m^{ij}_{s_{ij}}\}$ of $(\sigma^\vee+m_i)\cap (\sigma^\vee+m_j)\cap M_{\sigma,\beta}$ for each pair $(i,j)\in \{1,...,s\}^{\times 2}$ which are distinct from each other. Then again the relations between the generators are given by $(\sigma^\vee+m_{k}^{ij})\cap (\sigma^\vee+m_{l}^{ij})\cap M_{\sigma,\beta}$ for each $(k,l)\in \{1,..., s_{ij}\}^{\times 2}$. Iterating these process, we have a sequence
\begin{equation}\label{setsequence}
S_m, \{\sigma^\vee+m_i\}_{i=1}^s, \{\sigma^\vee+m^{ij}_{k}\}_{k, (i,j)}, \cdots
\end{equation}
of sets of subsets of $M_\bR$ and a resolution of $S^m\cap M_{\sigma,\beta}$
\begin{equation}\label{resolutionsemigroup}
0\leftarrow S^m\cap M_{\sigma,\beta} \leftarrow \bigoplus_{i=1}^s ((\sigma^\vee\cap M_{\sigma,\beta})+m_i) \leftarrow \bigoplus_{k, (i,j)}((\sigma^\vee\cap M_{\sigma,\beta})+m^{ij}_k)\leftarrow\cdots
\end{equation}
as a $(\sigma^\vee\cap M_{\sigma,\beta})$-module. Note that we can write the sequence (\ref{resolutionsemigroup}) as
\begin{equation}\label{resolutionsemigroup2}
0\leftarrow S^m\cap M_{\sigma,\beta}  \leftarrow \bigoplus_{i=1}^s (\sigma^\vee\cap M_{\sigma,\beta}) \leftarrow \bigoplus_{k, (i,j)}(\sigma^\vee\cap M_{\sigma,\beta})\leftarrow\cdots
\end{equation}
by using the identification $((\sigma^\vee\cap M_{\sigma,\beta})+m)\cong (\sigma^\vee\cap M_{\sigma,\beta})$ as $(\sigma^\vee\cap M_{\sigma,\beta})$-modules. To simplify the notation, we set
\begin{align}
\bigoplus_{i=1}^{n_1}((\sigma^\vee\cap M_{\sigma,\beta})+m_i^1)&:=\bigoplus_{i=1}^{s}((\sigma^\vee\cap M_{\sigma,\beta})+m_i)\\
\bigoplus_{i=1}^{n_2}((\sigma^\vee\cap M_{\sigma,\beta})+m^2_i)&:=\bigoplus_{k, (i,j)}((\sigma^\vee\cap M_{\sigma,\beta})+m^{ij}_k)
\end{align}
and so on. 

Set $\bC[S^m\cap M_{\sigma,\beta}]:=\bigoplus_{m\in S^m\cap M_{\sigma,\beta}}\bC\cdot\chi^m$, which is a finitely generated $\bC[\sigma^\vee\cap M_{\sigma,\beta}]$-module and $\varhocolim{m\in M}\bC[S^m\cap M_{\sigma,\beta}]\simeq \bC[\tau^\vee\cap M_{\sigma,\beta}]$ where the colimit is taken with respect to natural inclusions $\bC[S^{m_1}\cap M_{\sigma,\beta}]\hookrightarrow \bC[S^{m_2}\cap M_{\sigma,\beta}]$ for $\sigma^\vee+m_1\hookrightarrow \sigma^\vee+m_2$.  Then the resolution (\ref{resolutionsemigroup2}) implies a resolution of $\bC[S^m\cap M_{\sigma,\beta}]$ as a $\bC[\sigma^\vee\cap M_{\sigma,\beta}]$-module:
\begin{equation}\label{12resolution}
0\leftarrow \bC[S^m\cap M_{\sigma,\beta}]\xleftarrow{d_1} \bC[\sigma^\vee\cap M_{\sigma,\beta}]^{\oplus n_1}\xleftarrow{d_2} \bC[\sigma^\vee\cap M_{\sigma,\beta}]^{\oplus n_2}\leftarrow \cdots
\end{equation}
We also have a complex
\begin{equation}\label{12resolution2}
E:=0\leftarrow \bC_{S^m}\xleftarrow{\delta_1}\bigoplus_{i=1}^{n_1} \bC_{\Int(\sigma^\vee)+{m^1_i}} \xleftarrow{\delta_2} \bigoplus_{i=1}^{n_2}\bC_{\Int(\sigma^\vee)+{m^2_i}}\leftarrow \cdots,
\end{equation}
from the sequence (\ref{setsequence}). Note that the presentation (\ref{resolutionsemigroup}) implies that $\kappa_{\hsigma,\beta}(d_i)\simeq p_!(\delta_i)$ for $i>1$.

We claim that
\begin{equation}\label{11.24}
\hom(\bC_{\Int(\sigma^\vee)+m-\chi}, p^{-1}F\star_{\bR}E)\simeq 0.
\end{equation}
We can again assume that $\sigma^\vee$ is full-dimensional and $\chi=0$ to prove (\ref{11.24}).

Take an increasing filtration
\begin{equation}
S_1\subsetneq S_2\subsetneq S_3\subsetneq\cdots \subset \Int(\sigma^\vee)
\end{equation}
by $S_i$ such that each $S_i$ is closed and bounded in $\Int(\sigma^\vee)$ and is a union some $D_a$'s. Then we have two sequence of maps
\begin{align}
&0\leftarrow \bC_{S_1}\leftarrow \bC_{S_2}\leftarrow \cdots\text{ and }\\
&0\leftarrow \bC_{S^m\cap S_1}\leftarrow \bC_{S^m\cap S_2}\leftarrow \cdots
\end{align}
with
\begin{align}
\bC_{\Int(\sigma^\vee)}&\simeq \varholim{i} \bC_{S_i} \text{ and }\\
\bC_{S^m}&\simeq \varholim{i} \bC_{S_i\cap S^m}.
\end{align}
Since the morphisms of $E$ can be restricted to each $S_i$, we have
\begin{equation}\label{11.33}
E\simeq \varholim{i}(0\leftarrow \bC_{S^m\cap S_i}\leftarrow \bigoplus_{j=1}^{n_1}\bC_{S_i+m^1_j} \leftarrow \bigoplus_{j=1}^{n_2}\bC_{S_i+m^2_j} \leftarrow \cdots).
\end{equation}
We write $E_i$ for $i$th term of the limit (\ref{11.33}). Since we take generators $m^k_i$ which are distinct from each other in each step of constructing (\ref{setsequence}), we have $\sigma^\vee+m_i^k\cap \sigma^\vee+m_{j}^k\subsetneq \sigma^\vee+m_i^k, \sigma^\vee+m_{j}^k$. This observation and the convexity of $\sigma^\vee$ together imply that, for a bounded set $B$ in $M_\bR$, there exists a large $K$ such that $m_i^k\not\in B$ for $k>K$ and any $i$. This implies that each $E_i$ is a finite limit.

Note that the complex (\ref{12resolution2}) is not exact in general while (\ref{12resolution}) is exact. However the exactness of (\ref{resolutionsemigroup}) implies the following: Consider $D_a$ such that $\bD D_a$ contains a point of $M_{\sigma,\beta}$. We write $m_a$ for the point of $M_{\sigma,\beta}$ contained in $\bD D_a$. The stalk of $\bC_{\Int(\sigma^\vee)+m_i^k}$ on $D_a$ is rank 1 if and only if $m_a\in \sigma^\vee+m_i^k$ otherwise zero. This exactly corresponds to whether $\sigma^\vee+m_i^k$ contains an element $m_a$ viewed as a $(\sigma^\vee\cap M_{\sigma,\beta})$-submodule of $M$ or does not. Since the resolution (\ref{resolutionsemigroup}) occurs in $M$ as submodules, the exactness of (\ref{resolutionsemigroup}) implies the vanishing of the stalk of $\bC_{\Int(\sigma^\vee)+m_i^k}$ over such $D_a$. Hence each $E_i$ can be written by a finite sequence of cones of $\bC_{D_a}$'s such that each $\bD D_a$ does not contain points of $M_{\sigma,\beta}$.

Set $i_{m'}\colon \Int(\sigma^\vee)+m'\hookrightarrow M_\bR$ for $m'\in M_{\sigma,\beta}$. Then we have $\varhocolim{m'\in M_{\sigma,\beta}} i_{m'!}i_{m'}^{-1}p^{-1}F\simeq p^{-1}F$ where the colimit is taken with respect to maps  $i_{m'_1!}i_{m'_1}^{-1}p^{-1}F\rightarrow i_{m'_2!}i_{m'_2}^{-1}p^{-1}F$ for $\Int(\sigma^\vee)+m'_1\subset \Int(\sigma_2^\vee)+m'_2$ which correspond to the identity via the isomorphisms
\begin{equation}
\hom(i_{m'_1!}i_{m'_1}^{-1}p^{-1}F, i_{m'_2!}i_{m'_2}^{-1}p^{-1}F)\simeq \hom(i_{m'_1!}i_{m'_1}^{-1}p^{-1}F,i_{m'_1!}i_{m'_1}^{-1}p^{-1}F).
\end{equation}
We set $F_{m'}:=i_{m'!}i_{m'}^{-1}p^{-1}F$. Then we have
\begin{equation}\label{11.37}
\begin{split}
\hom(\bC_{\Int(\sigma^\vee)+m}, p^{-1}F\star_\bR E)&\simeq 
\hom(\bC_{\Int(\sigma^\vee)+m}, \varhocolim{m'\in M_{\sigma,\beta}}F_{m'}\star_\bR E)\\
&\simeq \varhocolim{m'\in M_{\sigma,\beta}} \hom(\bC_{\Int(\sigma^\vee)+m}, F_{m'}\star_\bR E),
\end{split}
\end{equation}
since $\bC_{\Int(\sigma^\vee)+m}$ is compact in $\lSh_{\Lambda}(M_\bR)$ where $\Lambda:=\bigcup_{m'}\SS(F_{m'})$ and $F_{m'}\in \lSh_{T^n\times (-\sigma)}(M_\bR)$. Since the support of $E$ and $F_{m'}$ is in a translation of $\Int(\sigma^\vee)$, the map $\tilde{m}$ is proper over the support of $E\boxtimes F_{m'}$ by the assumption $\sigma$ is full-dimensional which is equivalent to $\sigma^\vee$ is strictly convex. Hence we have $E\star_\bR F_{m'}\simeq \varholim{i}(E_i\star_\bR  F_{m'})$. As a result, we have
\begin{equation}
\begin{split}
\hom(\bC_{\Int(\sigma^\vee)+m}, E\star_\bR F_{m'})&\simeq \hom(\bC_{\Int(\sigma^\vee)+m}, \varholim{i} (E_i\star_\bR F_{m'}))\\
&\simeq \varholim{i}\hom(\bC_{\Int(\sigma^\vee)+m}, E_i\star_\bR F_{m'}).\\
\end{split}
\end{equation}
Since $E_i\star_\bR F_{m'}\in \lSh_{T^n\times (-\sigma)}$, we again have
\begin{equation}
\hom(\bC_{\Int(\sigma^\vee)+m}, E_i\star_\bR F_{m'})\simeq \hom(\bC_{(\Int(\sigma^\vee)\cap B)+m}, E_i\star_\bR F_{m'}).
\end{equation}
Since $E_i$ is a finite sequence of cones of $\bC_{D_a}$, we further have
\begin{equation}
\begin{split}
\hom(\bC_{(\Int(\sigma^\vee)\cap B)+m}, E_i\star_\bR F_{m'})&\simeq \hom(\bC_{(\Int(\sigma^\vee)\cap B)+m} \star_\bR (-1)^*\bD E_i, F_{m'}\star \bC_{\Int(\sigma^\vee)})\\
&\simeq  \hom(\bC_{(\Int(\sigma^\vee)\cap B)+m} \star_\bR (-1)^*\bD E_i, F_{m'}).
\end{split}
\end{equation}
by Corollary \ref{cptadjunction}. We further have
\begin{equation}
\begin{split}
\hom(\bC_{(\Int(\sigma^\vee)\cap B)+m} \star_\bR (-1)^*\bD E_i, F_{m'})&\simeq \hom(\bC_{(\Int(\sigma^\vee)\cap B)+m},\cHom^{\star_\bR}((-1)^*\bD E_i, F_{m'}))\\
&\simeq  \hom(\bC_{\Int(\sigma^\vee)+m},\cHom^{\star_\bR}((-1)^*\bD E_i, F_{m'}))\\
&\simeq \hom(\bC_{\Int(\sigma^\vee)+m}\star_\bR(-1)^*\bD E_i, F_{m'})\\
&\simeq \hom((-1)^*\bD E_i, F_{m'}^m)
\end{split}
\end{equation}
where $F_{m'}^m$ is the translation of $F_{m'}$ by $m$. Hence we have
\begin{equation}\label{11.40}
\hom(\bC_{\Int(\sigma^\vee+m)}, p^{-1}F\star_\bR E)\simeq \varhocolim{m'\in M_{\sigma,\beta}}\hom(\varhocolim{i}(-1)^*\bD E_i, F_{m'}^m).
\end{equation}
Temporally, let us replace $F^m$ by $F$, which has no effects to prove (\ref{11.24}). Let $j_{m'}$ be the inclusion map of the complement of $i_{m'}$. Then there exists an exact triangle
\begin{equation}\label{12keytriangle2}
\begin{split}
j_{m'!}j_{m'}^!&\cHom(\varhocolim{i}(-1)^*\bD E_i, F_{m'})\\
&\rightarrow \cHom(\varhocolim{i}(-1)^*\bD E_i, F_{m'})\\
&\rightarrow  i_{m'*}i_{m'}^{-1}\cHom(\varhocolim{i}(-1)^*\bD E_i, F_{m'}))\rightarrow.
\end{split}
\end{equation}
Note that each term in (\ref{12keytriangle2}) has maps induced by $F_{m_1'}\rightarrow F_{m_2'}$ for $\Int(\sigma^\vee)+m_1'\subset \Int(\sigma^\vee)+m_2'$. By taking global sections and colimits, we have
\begin{equation}\label{12keytriangle}
\begin{split}
\varhocolim{m'\in M_{\sigma,\beta}}&\Gamma(M_\bR, j_{m'!}j_{m'}^!\cHom(\varhocolim{i}(-1)^*\bD E_i, F_{m'}))\\
&\rightarrow \varhocolim{m'\in M_{\sigma,\beta}}\hom(\varhocolim{i}(-1)^*\bD E_i, F_{m'})\\
&\rightarrow \varhocolim{m'\in M_{\sigma,\beta}}\hom(i_{m'}^{-1}\varhocolim{i}(-1)^*\bD E_i, i_{m'}^{-1}p^{-1}F)\rightarrow 
\end{split}
\end{equation}
The last term in (\ref{12keytriangle})
\begin{equation}
\begin{split}
\varhocolim{m'\in M_{\sigma,\beta}}\hom(i_{m'}^{-1}\varhocolim{i}(-1)^*\bD E_i, i_{m'}^{-1}p^{-1}F)&\simeq \varhocolim{m'\in M_{\sigma,\beta}}\varholim{i}\hom(i_{m'}^{-1}(-1)^*\bD E_i, i_{m'}^{-1}p^{-1}F)\\
&\simeq \varhocolim{m'\in M_{\sigma,\beta}}\varholim{i}\hom(i_{m'!}i_{m'}^{-1}(-1)^*\bD E_i, p^{-1}F)\\
\end{split}
\end{equation}
By the definition, $i_{m'!}i_{m'}^{-1}\bC_{-\bD D_a}\simeq \bC_{-\bD D_a}$ if $-\bD D_a\subset \Int(\sigma^\vee)+m'$. If not $-\bD D_a\subset \Int(\sigma^\vee)+m'$, we have $i_{m'!}i_{m'}^{-1}\bC_{-\bD D_a}\simeq 0$.  Hence, by Lemma \ref{Davanishing}, the last term in (\ref{12keytriangle}) vanishes.

On the other hand, the complex $j_{m'!}j_{m}^!\cHom(\varhocolim{i}(-1)^*\bD E_i, F_{m'})$ is supported in the boundary of $\sigma+m'$. Hence the maps appeared in the colimit of the first term of (\ref{12keytriangle}) eventually vanishes when the boundary of $\Int(\sigma^\vee+m_1')$ does not intersect with the boundary of $\Int(\sigma^\vee+m_2')$. Therefore the first term of (\ref{12keytriangle}) also vanishes. Then we also have the vanishing of the middle of (\ref{12keytriangle}). From (\ref{11.37}) - (\ref{11.40}), the vanishing (\ref{11.24}) follows.

By (\ref{11.24}), (\ref{12resolution2}), and (\ref{11.21}), we have
\begin{equation}
\begin{split}
\hom&(\oplus_{[\chi]\in M_{\sigma,\beta}/M}p_!\bC_{\Int(\sigma^\vee)-\chi}, F\star p_!\bC_{S^m})\\
& \simeq (\oplus_{i=1}^{n_1}\hom(\oplus_{[\chi]\in M_{\sigma,\beta}/M}p_!\bC_{\Int(\sigma^\vee)-\chi}, F) \leftarrow\oplus_{i=1}^{n_2} \hom(\oplus_{[\chi]\in M_{\sigma,\beta}/M}p_!\bC_{\Int(\sigma^\vee)-\chi}, F)\leftarrow \cdots)\\
&\simeq \hom(\oplus_{[\chi]\in M_{\sigma,\beta}/M}p_!\bC_{\Int(\sigma^\vee)-\chi}, F)\otimes_{\bC[\sigma^\vee\cap M_{\sigma,\beta}]}\bC[S^m \cap M_{\sigma,\beta}].
\end{split}
\end{equation}
By taking colimits with respect to $m$ and applying (\ref{11.20}), this completes the proof.
\end{proof}

Let $\cY_{\lSh_{\Lambda_{\hSigma,\beta}}(T^n)}\colon (\lSh_{\Lambda_{\hSigma,\beta}}(T^n))^\op\rightarrow \Fun(\lSh_{\Lambda_{\hSigma,\beta}}(T^n), \Mod(\bC))$ be the Yoneda embedding.
\begin{lemma}\label{commIl}
There exists a commutative diagram
\begin{equation}
\xymatrix{
(\perf \cX_{\hSigma,\beta})^\op \ar[d]_{\cY_{\lSh_{\Lambda_{\hSigma,\beta}}(T^n)}\circ\kappa_{\hSigma,\beta}((-)\otimes \omega_{\cXsb}^{-1})}\ar[rr]_-{f_*} && (\Coh \cX_{\hSs,\beta})^\op\ar[d]^{\kappa^D_{\hsigma,\beta}\circ \frakD_{\cX_{\hSs,\beta}}} \\
\Fun(\lSh_{\Lambda_{\hSigma,\beta}}(T^n),\Mod(\bC)) \ar[rr]^--{\circ I} &&\Fun(\lSh_{\Lambda_{\hSs,\beta}}(T^n), \Mod(\bC)). 
}
\end{equation}
\end{lemma}
\begin{proof}
For $\cE\in \Coh \cX_{\hSigma,\beta}$ and $F\in \lSh_{\Lambda_{\hSs,\beta}}(T^n)$, we have
\begin{equation}
\begin{split}
\hom(\kappa^D_{\hsigma,\beta}(\frakD f_*\cE), F)&
\simeq (\frakD f_*\cE\otimes_{\bC[\sigma^\vee\cap M_{\sigma,\beta}]}\hom(\Theta(\sigma), F))^{H_\beta}\\
&\simeq (f_*\frakD\cE\otimes_{\bC[\sigma^\vee\cap M_{\sigma,\beta}]}\hom(\Theta(\sigma), F))^{H_\beta}.
\end{split}
\end{equation}
since $f$ is proper. By Lemma \ref{okgo}, we have
\begin{equation}
\begin{split}
(f_*\frakD\cE\otimes_{\bC[\sigma^\vee\cap M_{\sigma,\beta}]}&\hom(\Theta(\sigma), F))^{H_\beta}\\
&\simeq f_*(\frakD\cE\otimes_{\cO_{\cX_{\hSigma,\beta}}}\kappa^{-1}_{\hSigma,\beta}(IF))^{H_\beta}\\
&\simeq \hom(\cE\otimes \omega_{\cX_{\hSigma,\beta}}^{-1}, \kappa_{\hSigma,\beta}^{-1}(IF))\\
&\simeq \hom(\kappa_{\hSigma,\beta}(\cE\otimes_{\cO_{\cX_{\hSigma,\beta}}} \omega_{\cX_{\hSigma,\beta}}^{-1}), IF).
\end{split}
\end{equation}
This completes the proof.
\end{proof}

\begin{corollary}\label{representability}
For $\cE\in \Coh\cX_{\hSs,\beta}$, the image $\kappa^D_{\hsigma, \beta}(\frakD_{\cX_{\hSs,\beta}}\cE)$ is representable by an object of $\wSh_{\Lambda_{\hSs,\beta}}(T^n)$.
\end{corollary}
\begin{proof}
By the presentation (\ref{formula}), the functor $\kappa^D_{\hsigma, \beta}(\frakD\cE)$ is cocontinuous. Hence it suffices to show that this functor is representable.

Let $\cE^\bullet$ be a resolution of $\cE$ by locally free sheaves. This resolution gives a sequence $\{\cE_i\}$ in $\perf\cX_{\hSs,\beta}$ such that $\varhocolim{i}\cE_i\simeq \cE$. By Lemma~\ref{commIl}, we have

\begin{equation}
\begin{split}
\hom(\kappa^D_{\hsigma,\beta}(\frakD\cE), F)&\simeq (\frakD \cE\otimes_{\bC[\sigma^\vee\cap M_{\sigma,\beta}]}\hom(\Theta(\sigma), F))^{H_\beta}\\
&\simeq (\varholim{i}(\frakD (\cE_i))\otimes_{\bC[\sigma^\vee\cap M_{\sigma,\beta}]}\hom(\Theta(\sigma), F))^{H_\beta}\\
&\simeq \varholim{i} \hom(\kappa_{\hSigma,\beta}(\cE_i\otimes_{\cO_{\cX_{\hSigma,\beta}}} \omega_{\cX_{\hSigma,\beta}}^{-1}), IF)\\
&\simeq \hom(I^l\varhocolim{i}\kappa_{\hSigma,\beta}(\cE_i\otimes_{\cO_{\cX_{\hSigma,\beta}}} \omega_{\cX_{\hSigma,\beta}}^{-1}), F).
\end{split}
\end{equation}
This completes the proof.
\end{proof}

By Lemma \ref{fullyfaithful} and Lemma \ref{representability}, we have a fully faithful functor $\kappa^{D}_{\hsigma,\beta}\circ \frakD \colon  \Coh \cX_{\hSs,\beta}\rightarrow \Fun(\wSh_{\Lambda_{\hSs,\beta}}(T^n), \Mod(\bC))$, which is representable by compact objects. Then we set
\begin{equation}
K_{\hsigma,\beta}:=\kappa^{D}_{\hsigma,\beta}\circ \frak D_{\cX_{\hSigma(\hsigma),\beta}}\colon \Coh \cX_{\hSs,\beta}\rightarrow \wSh_{\Lambda_{\hSs,\beta}},
\end{equation}
which is a fully faithful functor. As noted in Section \ref{indcoherent}, we also write $K_{\hsigma,\beta}\colon \Indcoh \cX_{\hSs,\beta}\rightarrow \lSh_{\Lambda_{\hSs.\beta}}(T^n)$ for the induced functor on ind-objects.

\begin{corollary}\label{singularequivalence}
The functor $K_{\hsigma,\beta}$ is an equivalence.
\end{corollary}
\begin{proof}
By Lemma \ref{fullyfaithful2}, the functor $K_{\hsigma,\beta}$ is fully faithful. Since $I^l$ is essentially surjective and $K_{\hSigma,\beta}$ is an equivalence in Lemma \ref{commIl}, $K_{\hsigma,\beta}$ is also essentially surjective.
\end{proof}

\begin{corollary}\label{anyaffine}
Any $\hsigma$ satisfies Condition \ref{condition2}.
\end{corollary}
\begin{proof}
By Corollary \ref{singularequivalence}, the functor $K_{\hsigma,\beta}$ is an equivalence.  Take an object $\cE\in \Coh \cX_{\hSs,\beta}$. For a face inclusion of cones $\sigma_1\subset \sigma_2=\sigma$, we have quasi-isomorphisms of functors $\lSh_{\Lambda_{\hSigma(\sigma_1),\beta}}(T^n)\rightarrow \Mod(\bC)$,
\begin{equation}
\begin{split}
\hom&(K_{\hsigma_1,\beta}(i^*_{\sigma_1\sigma_2}\cE), -)\\
&\simeq (\frakD i_{\sigma_1\sigma_2}^*\cE\otimes_{\bC[\sigma_1^\vee\cap M_{\sigma,\beta}]}\hom_{\lSh_{\Lambda_{\hSigma(\sigma_1),\beta}}(T^n)}(\Theta(\sigma_1),-))^{H_\beta}\\
&\simeq (i_{\sigma_1\sigma_2}^*\frakD\cE\otimes_{\bC[\sigma_1^\vee\cap M_{\sigma,\beta}]}\hom_{\lSh_{\Lambda_{\hSigma(\sigma_1),\beta}}(T^n)}(\Theta(\sigma_1),-))^{H_\beta}\\
&\simeq (i_{\sigma_1\sigma_2}^*\frak D\cE\otimes_{\bC[\sigma_1^\vee\cap M_{\sigma,\beta}]}\hom_{\lSh_{\Lambda_{\hSigma(\sigma_1),\beta}}(T^n)}(\Theta(\sigma_2),-))^{H_\beta}\\
&\simeq (\frak D\cE\otimes_{\bC[\sigma_2^\vee\cap M_{\sigma,\beta}]}\bC[\sigma_1^\vee\cap M_{\sigma,\beta}]\otimes_{\bC[\sigma_1^\vee\cap M_{\sigma,\beta}]}\hom_{\lSh_{\Lambda_{\hSigma(\sigma_1),\beta}}(T^n)}(\Theta(\sigma_2),-))^{H_\beta}\\
&\simeq (\frak D\cE\otimes_{\bC[\sigma_2^\vee\cap M_{\sigma,\beta}]} \bC[\sigma_1^\vee\cap M_{\sigma,\beta}]\otimes_{\bC[\sigma_1^\vee\cap M_{\sigma,\beta}]}\hom_{\lSh_{\Lambda_{\hSigma(\sigma_1),\beta}}(T^n)}(\Theta(\sigma_2),-))^{H_\beta}.
\end{split}
\end{equation}
By seeing the $\bC[\sigma^\vee_1\cap M_{\sigma,\beta}]$-module structure of $\hom_{\lSh_{\Lambda_{\hSigma(\sigma_1),\beta}}(T^n)}(\Theta(\sigma_2),-)$, it follows that
\begin{equation}
\begin{split}
\bC[\sigma_{2}^\vee\cap M_{\sigma,\beta}]\otimes_{\bC[\sigma_2^\vee\cap M_{\sigma,\beta}]}\bC[\sigma_1^\vee\cap M_{\sigma,\beta}]&\otimes_{\bC[\sigma_1^\vee\cap M_{\sigma,\beta}]}\hom_{\lSh_{\Lambda_{\hSigma(\sigma_1),\beta}}(T^n)}(\Theta(\sigma_2),-)\\
&\simeq \hom_{\lSh_{\Lambda_{\hSigma(\sigma_1),\beta}}(T^n)}(\Theta(\sigma_2),-)
\end{split}
\end{equation}
as $\bC[\sigma_2^\vee\cap M]$-modules where the right hand side is equipped with the canonical 
$\bC[\sigma_2^\vee\cap M]$-module structure. Then we have
\begin{equation}
\begin{split}
\hom(K_{\hsigma_1,\beta}(i^*_{\sigma_1\sigma_2}\cE), -)&\simeq (\frakD\cE\otimes_{\bC[\sigma_2^\vee\cap M_{\sigma,\beta}]}\hom_{\lSh_{\Lambda_{\hSigma(\sigma_1),\beta}}(T^n)}(\Theta(\sigma_2),-))^{H_\beta}\\
&\simeq \hom_{\lSh_{\Lambda_{\hSigma(\sigma_1),\beta}}(T^n)}(K_{\hsigma_2,\beta}(\cE), -).
\end{split}
\end{equation}
Then we further have
\begin{equation}
\begin{split}
\hom_{\lSh_{\Lambda_{\hSigma(\sigma_1),\beta}}(T^n)}(K_{\hsigma_2,\beta}(\cE), -)&\simeq \hom(K_{\hsigma_2,\beta}(\cE), \cHom^\star(\kappa_{\hsigma_1,\beta}(\cO_{\cX_{\hSigma(\hsigma_1),\beta}}), (-)))\\
&\simeq \hom(K_{\hsigma_2,\beta}(\cE) \star\kappa_{\hsigma_1,\beta}(\cO_{\cX_{\hSigma(\hsigma_1),\beta}}), (-))\\
\end{split}
\end{equation}
by Lemma \ref{lemmaeq}. By taking a smooth refinement $\hSigma_2$ of $\hSigma(\hsigma_2)$, we also have $K_{\hsigma_2,\beta}(\cE)\in \lSh_{\Lambda_{\hSigma_2,\beta}}(T^n)$. Since $\hSigma_2$ is smooth, the functor $\kappa_{\hSigma_2,\beta}$ is an equivalence. Let $\hSigma_1$ be the associated refinement of $\hSigma(\hsigma_1)$. Then, by Proposition \ref{monoidal}, we have
\begin{equation}
K_{\hsigma_2,\beta}(\cE) \star\kappa_{\hsigma_1,\beta}(\cO_{\cX_{\hSigma(\hsigma_1),\beta}})\simeq \kappa_{\hSigma_2,\beta}(\kappa_{\hSigma_2,\beta}^{-1}(K_{\hsigma_2,\beta}(\cE))\otimes \cO_{\cX_{\hSigma_1,\beta}}).
\end{equation}
Then Corollary \ref{affinefunctorial} implies $K_{\hsigma_2,\beta}(\cE) \star\kappa_{\hsigma_1,\beta}(\cO_{\cX_{\hSigma(\hsigma_1),\beta}})\in \lSh_{\Lambda_{\hSigma_1,\beta}}(T^n)$. We also have  $K_{\hsigma_2,\beta}(\cE) \star\kappa_{\hsigma_1,\beta}(\cO_{\cX_{\hSigma(\hsigma_1),\beta}})\in \lSh_{\Lambda_{\hSigma(\hsigma_2),\beta}}(T^n)$. This can be deduced, for example, from Remark \ref{indmonoidal} below, which implies
\begin{equation}
K_{\hsigma_2,\beta}(\cE) \star\kappa_{\hsigma_1,\beta}(\cO_{\cX_{\hSigma(\hsigma_1),\beta}})\simeq K_{\hsigma_2,\beta}(\cE\indotimes K_{\hsigma_2,\beta}^{-1}(\kappa_{\hsigma_1,\beta}(\cO_{\cX_{\hSigma(\hsigma_1),\beta}}))).
\end{equation}
Since $\Lambda_{\hSigma(\hsigma_2),\beta}\cap \Lambda_{\hSigma_1, \beta}=\Lambda_{\hSigma(\hsigma_1),\beta}$, this completes the proof of Condition \ref{condition2} (i).

For $\cE\in \perf \cX_{\hSs,\beta}$, we have
\begin{equation}
\begin{split}
& (\frakD(\cE\otimes \omega_{\cX_{\hSs,\beta}})\otimes_{\bC[\sigma^\vee\cap M_{\sigma,\beta}]} \hom(\Theta(\sigma), -))^{H_\beta}\\
&\simeq (D(\cE) \otimes_{\bC[\sigma^\vee\cap M_{\sigma,\beta}]} \hom(\Theta(\sigma), -))^{H_\beta},
\end{split}
\end{equation}
which is represented by $\kappa_{\hsigma,\beta}(\cE)$. This proves Condition 9.1(ii).
\end{proof}

\begin{proof}[Proof of Theorem \ref{main} and Theorem \ref{main2}]
Since Condition \ref{condition2} holds for any fans by Corollary \ref{anyaffine}, Proposition \ref{mainaux} holds for any fans. This completes the proof.
\end{proof}

We recall the following duality theorem between coherent sheaves and perfect complexes.
\begin{theorem}[{\cite[Theorem 1.1.3]{BZNP}}]\label{dualityperfect}
Assume $\cXsb$ is complete. Then 
\begin{equation}
\perf \cXsb\simeq \Fun{}^{ex}(\Coh\cXsb, \mod(\bC))
\end{equation}
given by $\cE\mapsto \hom(D\cE, -)$ where $D:=\cHom(-, \cO_{\cX_{\hSigma,\beta}})$.
\end{theorem}

\begin{corollary}\label{complete}
Assume that $\cXsb$ is complete. Then $\kappa_{\hSigma,\beta}$ gives an equivalence
\begin{equation}\label{perfequiv}
\perf \cXsb \simeq \cSh_\Lsb(T^n).
\end{equation}
In particular, if $\cXsb$ is smooth,  $\wSh_{\Lsb}(T^n)\simeq \cSh_{\Lsb}(T^n)$.
\end{corollary}
\begin{proof}
By Theorem \ref{main} and the Grothendieck duality, we have
\begin{equation}
\Coh \cXsb \simeq (\Coh \cXsb)^\op\simeq (\wSh_{\Lsb}(T^n))^\op
\end{equation}
given by $K_{\hSigma,\beta}\circ \frakD$.
By taking $\Fun^{ex}(-,\mod(\bC))$ and using Theorem \ref{dualityperfect} and Theorem \ref{dualityconst}, we have
\begin{equation}
\cSh_{\Lsb}(T^n)\xrightarrow{\simeq} \Fun{}^{ex}( (\wSh_{\Lsb}(T^n))^\op, \mod(\bC))\xrightarrow{\simeq} \Fun{}^{ex}(\Coh\cXsb, \mod(\bC))\xrightarrow{\simeq}   \perf \cXsb 
\end{equation}
given by
\begin{equation}
E\mapsto \hom(E,-)\mapsto \hom(E, K_{\hSigma,\beta}\circ \frakD(-))\mapsto D\circ \frakD\circ K_{\hSigma,\beta}^{-1}(E).
\end{equation}
Hence the equivalence (\ref{perfequiv}) is given by $K_{\hSigma,\beta}\circ \frakD\circ D\simeq K_{\hSigma,\beta}\circ ((-)\otimes \omega_\cXsb)$. By Condition \ref{condition2}(ii), we have $K_{\hSigma,\beta}\circ( (-)\otimes \omega_\cXsb)\simeq \kappa_{\hSigma,\beta}$ on $\perf\cXsb$. This completes the proof.
\end{proof}

\begin{remark} Assume $\cXsb$ is complete. The equivalence of Corollary \ref{complete} and Theorem \ref{main} is compatible with the dualities Theorem \ref{dualityconst} and Thereom \ref{dualityperfect} in the following sense. Indeed, for $\cE\in \perf\cXsb$ and $\frakD_\cXsb\cF\in \Coh\cXsb$, the pairing in Theorem \ref{dualityperfect} is
\begin{equation}\label{12.63}
\begin{split}
\hom(D\cE, \frakD\cF)&\simeq \hom(\cF, \cE\otimes \omega_\cXsb)\\
&\simeq \hom(K_{\hSigma,\beta}(\cF), K_{\hSigma,\beta}(\cE\otimes \omega_\cXsb))\\
&\simeq \hom(K_{\hSigma,\beta}(\cF), \kappa_{\hSigma,\beta}(\cE)).
\end{split}
\end{equation}
Since $\kappa_{\hSigma,\beta}(\cE)\in \cSh_{\Lsb}(T^n)$ by \cite{FLTZ} and $K_{\hSigma,\beta}(\cF)\in \wSh_\Lsb(T^n)$, the last line of (\ref{12.63}) is the pairing of Theorem \ref{dualityconst}. This remark is inspired by an implication by Harold Williams.
\end{remark}

\begin{remark}\label{indmonoidal}
We also remark about the monoidality of the functor $K_{\hSigma,\beta}$ similar to $\kappa_{\hSigma,\beta}$. The category $\Indcoh \cXsb$ is symmetric monoidal with the product $(-)\indotimes (-):=\Delta^!((-)\boxtimes (-))$ where $\Delta$ is the diagonal map and the unit $\omega_{\cX_{\hSigma,\beta}}$ by \cite[Corollary 5.6.8]{GaitsgoryIndCoh}. Let $\hSigma'$ be a smooth refinement of $\hSigma$ and $f\colon \cX_{\hSigma',\beta}\rightarrow \cX_{\hSigma,\beta}$ be the associated morphism. Then by the projection formula (\cite[3.2.5]{DriGai}) and the fact any toric variety is Cohen-Macaulay imply that 
\begin{equation}\label{projection}
\begin{split}
f^\Ind_*(f^! \cE\indotimes f^!\cF)&\simeq \cE\indotimes f^\Ind_*f^!\cF\\
&\simeq \cE\indotimes \cF\indotimes f_*\omega_{\cX_{\hSigma',\beta}}\\
&\simeq \cE\indotimes \cF\indotimes \omega_{\cX_{\hSigma,\beta}}\\
&\simeq  \cE\indotimes \cF.
\end{split}
\end{equation}
Since $\kappa_{\hSigma', \beta}$ is monoidal (Proposition \ref{monoidal}) and $\hSigma'$ is smooth, we have
\begin{equation}
\begin{split}
K_{\hSigma', \beta}(f^! \cE\indotimes f^!\cF)&\simeq \kappa_{\hSigma', \beta}(f^! \cE\otimes f^!\cF\otimes \omega^{-2}_{\cX_{\hSigma',\beta}})\\
&\simeq K_{\hSigma', \beta}(f^! \cE)\star K_{\hSigma', \beta}(f^!\cF).
\end{split}
\end{equation}
By Lemma \ref{commIl} and the fact $K_{\hSigma,\beta}$ and $K_{\hSigma',\beta}$ are both equivalences, we have $K_{\hSigma',\beta}\circ f^!\simeq I\circ K_{\hSigma,\beta}$. Combining with (\ref{projection}), we have
\begin{equation}
\begin{split}
K_{\hSigma, \beta}(\cE\indotimes \cF)
&\simeq K_{\hSigma, \beta}(f^\Ind_*(f^! \cE\indotimes f^!\cF))\\
&\simeq I^lK_{\hSigma', \beta}(f^! \cE\indotimes f^!\cF)\\
&\simeq I^l(IK_{\hSigma,\beta}(\cE)\star IK_{\hSigma,\beta}(\cF))\\
&\simeq I^l\circ I(K_{\hSigma,\beta}(\cE)\star K_{\hSigma,\beta}(\cF))\\
&\simeq K_{\hSigma,\beta}(\cE)\star K_{\hSigma,\beta}(\cF).
\end{split}
\end{equation}
\end{remark}

{\small 
\bibliographystyle{amsalpha}
\bibliography{cccarxiv8.bbl}
}

\noindent
Tatsuki Kuwagaki

Graduate School of Mathematical Sciences,
The University of Tokyo,
3-8-1 Komaba,
Meguro-ku,
Tokyo,
153-8914,
Japan.

{\em e-mail address}\ : \  kuwagaki@ms.u-tokyo.ac.jp
\ \vspace{0mm} \\

\end{document}